\documentclass{amsart}

\usepackage[english]{babel}
\usepackage{csquotes}
\usepackage[style=alphabetic,
            isbn=false,
            doi=false,
            url=false,
            backend=bibtex,
            maxnames=5
           ]{biblatex}
\bibliography{sources}
\usepackage{amsmath,amsfonts,amsthm,mathtools, amssymb}
\usepackage{xcolor}
\usepackage{tikz-cd}
\usepackage{hyperref}
\usepackage{aliascnt}
\usepackage{xparse}
\usepackage{tensor}
\usepackage{caption}
\usepackage{ mathrsfs }
\usepackage{tabto}
\usepackage{imakeidx}
\usepackage[normalem]{ulem}

 \RequirePackage{filecontents}

\indexsetup{level=\subsection*,toclevel=\paragraph,headers={ \MakeUppercase{Euler sequence on strata}}{\MakeUppercase{Euler sequence on strata}}, noclearpage,firstpagestyle=headings}
\makeindex[name=graph,title=Graphs and levels,columns=1,columnseprule,options=-s Idx.ist]

\usepackage{tikz}
\usetikzlibrary{matrix,arrows,calc}
\usetikzlibrary{positioning}
\usetikzlibrary{decorations.pathreplacing,decorations.markings,decorations.pathmorphing}
\usetikzlibrary{positioning,arrows,patterns}
\usetikzlibrary{cd}
\usetikzlibrary{intersections}

\def \HoG{.6cm} %
\def \PotlI{.85} %
\def \PotrI{.85} %

\tikzset{
	every loop/.style={very thick},
	comp/.style={circle,fill,black,,inner sep=0pt,minimum size=5pt},
	order bottom left/.style={pos=.05,left,font=\tiny},
	order top left/.style={pos=.9,left,font=\tiny},
	order bottom right/.style={pos=.05,right,font=\tiny},
	order top right/.style={pos=.9,right,font=\tiny},
	order node dis/.style={text width=.75cm},
	circled number/.style={circle, draw, inner sep=0pt, minimum size=12pt},
	below left with distance/.style={below left,text height=10pt},
    below right with distance/.style={below right,text height=10pt}
	}

\makeatletter
\ifcsname phantomsection\endcsname
    \newcommand*{\@gobblenexttocentry}[9]{}
\else
    \newcommand*{\@gobblenexttocentry}[4]{}
\fi
\newcommand*{\addsubsection}{%
    \addtocontents{toc}{\protect\@gobblenexttocentry}%
    \subsection*}
\makeatother

\begin{document}

\def\subsectionautorefname{Section}
\def\subsubsectionautorefname{Section}
\def\sectionautorefname{Section}
\def\equationautorefname~#1\null{(#1)\null}

\newcommand{\mynewtheorem}[4]{
  \if\relax\detokenize{#3}\relax %
    \if\relax\detokenize{#4}\relax %
      \newtheorem{#1}{#2}
    \else
      \newtheorem{#1}{#2}[#4]
    \fi
  \else
    \newaliascnt{#1}{#3}
    \newtheorem{#1}[#1]{#2}
    \aliascntresetthe{#1}
  \fi
  \expandafter\def\csname #1autorefname\endcsname{#2}
}

\mynewtheorem{theorem}{Theorem}{}{section}
\mynewtheorem{lemma}{Lemma}{theorem}{}
\mynewtheorem{rem}{Remark}{lemma}{}
\mynewtheorem{prop}{Proposition}{lemma}{}
\mynewtheorem{cor}{Corollary}{lemma}{}
\mynewtheorem{definition}{Definition}{lemma}{}
\mynewtheorem{question}{Question}{lemma}{}
\mynewtheorem{assumption}{Assumption}{lemma}{}
\mynewtheorem{example}{Example}{lemma}{}

\def\defbb#1{\expandafter\def\csname b#1\endcsname{\mathbb{#1}}}
\def\defcal#1{\expandafter\def\csname c#1\endcsname{\mathcal{#1}}}
\def\deffrak#1{\expandafter\def\csname frak#1\endcsname{\mathfrak{#1}}}
\def\defop#1{\expandafter\def\csname#1\endcsname{\operatorname{#1}}}
\def\defbf#1{\expandafter\def\csname b#1\endcsname{\mathbf{#1}}}

\makeatletter
\def\defcals#1{\@defcals#1\@nil}
\def\@defcals#1{\ifx#1\@nil\else\defcal{#1}\expandafter\@defcals\fi}
\def\deffraks#1{\@deffraks#1\@nil}
\def\@deffraks#1{\ifx#1\@nil\else\deffrak{#1}\expandafter\@deffraks\fi}
\def\defbbs#1{\@defbbs#1\@nil}
\def\@defbbs#1{\ifx#1\@nil\else\defbb{#1}\expandafter\@defbbs\fi}
\def\defbfs#1{\@defbfs#1\@nil}
\def\@defbfs#1{\ifx#1\@nil\else\defbf{#1}\expandafter\@defbfs\fi}
\def\defops#1{\@defops#1,\@nil}
\def\@defops#1,#2\@nil{\if\relax#1\relax\else\defop{#1}\fi\if\relax#2\relax\else\expandafter\@defops#2\@nil\fi}
\makeatother

\defbbs{ZHQCNPALRVW}
\defcals{DOPQMNXYLTRAEHZKCFI}
\deffraks{apijklmnopqueR}
\defops{PGL,SL,Sp,mod,Spec,Re,Gal,Tr,End,GL,Hom,PSL,H,div,Aut,rk,Mod,R,T,Tr,Mat,Vol,MV,Res,vol,Z,diag,Hyp,ord,Im,ev,U,dev,c,CH,fin,pr,Pic,lcm,ch,td,LG,id,Sym,Aut}
\defbfs{kiuvzwp} %

\def\ep{\varepsilon}
\def\abs#1{\lvert#1\rvert}
\def\dd{\mathrm{d}}
\def\inj{\hookrightarrow}
\def\eq{=}
\newcommand{\hyp}{{\rm hyp}}
\newcommand{\odd}{{\rm odd}}

\def\i{\mathrm{i}}
\def\e{\mathrm{e}}
\def\st{\mathrm{st}}
\def\ct{\mathrm{ct}}

\def\uC{\underline{\bC}}
\def\ol{\overline}
  
\def\Vrel{\bV^{\mathrm{rel}}}
\def\Wrel{\bW^{\mathrm{rel}}}
\def\twolev{\mathrm{LG_1(B)}}

\def\be{\begin{equation}}   \def\ee{\end{equation}}     \def\bes{\begin{equation*}}    \def\ees{\end{equation*}}
\def\ba{\be\begin{aligned}} \def\ea{\end{aligned}\ee}   \def\bas{\bes\begin{aligned}}  \def\eas{\end{aligned}\ees}
\def\={\;=\;}  \def\+{\,+\,} \def\m{\,-\,}

\newcommand*{\proj}{\mathbb{P}}
\newcommand{\barmoduli}[1][g]{{\overline{\mathcal M}}_{#1}}
\newcommand{\moduli}[1][g]{{\mathcal M}_{#1}}
\newcommand{\omoduli}[1][g]{{\Omega\mathcal M}_{#1}}
\newcommand{\modulin}[1][g,n]{{\mathcal M}_{#1}}
\newcommand{\omodulin}[1][g,n]{{\Omega\mathcal M}_{#1}}
\newcommand{\zomoduli}[1][]{{\mathcal H}_{#1}}
\newcommand{\barzomoduli}[1][]{{\overline{\mathcal H}_{#1}}}
\newcommand{\pomoduli}[1][g]{{\proj\Omega\mathcal M}_{#1}}
\newcommand{\pomodulin}[1][g,n]{{\proj\Omega\mathcal M}_{#1}}
\newcommand{\pobarmoduli}[1][g]{{\proj\Omega\overline{\mathcal M}}_{#1}}
\newcommand{\pobarmodulin}[1][g,n]{{\proj\Omega\overline{\mathcal M}}_{#1}}
\newcommand{\potmoduli}[1][g]{\proj\Omega\tilde{\mathcal{M}}_{#1}}
\newcommand{\obarmoduli}[1][g]{{\Omega\overline{\mathcal M}}_{#1}}
\newcommand{\obarmodulio}[1][g]{{\Omega\overline{\mathcal M}}_{#1}^{0}}
\newcommand{\otmoduli}[1][g]{\Omega\tilde{\mathcal{M}}_{#1}}
\newcommand{\pom}[1][g]{\proj\Omega{\mathcal M}_{#1}}
\newcommand{\pobarm}[1][g]{\proj\Omega\overline{\mathcal M}_{#1}}
\newcommand{\pobarmn}[1][g,n]{\proj\Omega\overline{\mathcal M}_{#1}}
\newcommand{\princbound}{\partial\mathcal{H}}
\newcommand{\omoduliinc}[2][g,n]{{\Omega\mathcal M}_{#1}^{{\rm inc}}(#2)}
\newcommand{\obarmoduliinc}[2][g,n]{{\Omega\overline{\mathcal M}}_{#1}^{{\rm inc}}(#2)}
\newcommand{\pobarmoduliinc}[2][g,n]{{\proj\Omega\overline{\mathcal M}}_{#1}^{{\rm inc}}(#2)}
\newcommand{\otildemoduliinc}[2][g,n]{{\Omega\widetilde{\mathcal M}}_{#1}^{{\rm inc}}(#2)}
\newcommand{\potildemoduliinc}[2][g,n]{{\proj\Omega\widetilde{\mathcal M}}_{#1}^{{\rm inc}}(#2)}
\newcommand{\omoduliincp}[2][g,\lbrace n \rbrace]{{\Omega\mathcal M}_{#1}^{{\rm inc}}(#2)}
\newcommand{\obarmoduliincp}[2][g,\lbrace n \rbrace]{{\Omega\overline{\mathcal M}}_{#1}^{{\rm inc}}(#2)}
\newcommand{\obarmodulin}[1][g,n]{{\Omega\overline{\mathcal M}}_{#1}}
\newcommand{\LTH}[1][g,n]{{K \overline{\mathcal M}}_{#1}}
\newcommand{\PLS}[1][g,n]{{\bP\Xi \mathcal M}_{#1}}

\DeclareDocumentCommand{\LMS}{ O{\mu} O{g,n} O{}}{\Xi\overline{\mathcal{M}}^{#3}_{#2}(#1)}
\DeclareDocumentCommand{\Romod}{ O{\mu} O{g,n} O{}}{\Omega\mathcal{M}^{#3}_{#2}(#1)}

\newcommand*{\Tw}[1][\Lambda]{\mathrm{Tw}_{#1}}  %
\newcommand*{\sTw}[1][\Lambda]{\mathrm{Tw}_{#1}^s}  %

\newcommand{\bfa}{{\bf a}}
\newcommand{\bfb}{{\bf b}}
\newcommand{\bfd}{{\bf d}}
\newcommand{\bfe}{{\bf e}}
\newcommand{\bff}{{\bf f}}
\newcommand{\bfg}{{\bf g}}
\newcommand{\bfh}{{\bf h}}
\newcommand{\bfm}{{\bf m}}
\newcommand{\bfn}{{\bf n}}
\newcommand{\bfp}{{\bf p}}
\newcommand{\bfq}{{\bf q}}
\newcommand{\bfP}{{\bf P}}
\newcommand{\bfR}{{\bf R}}
\newcommand{\bfU}{{\bf U}}
\newcommand{\bfu}{{\bf u}}
\newcommand{\bfz}{{\bf z}}

\newcommand{\bfl}{{\boldsymbol{\ell}}}
\newcommand{\bfmu}{{\boldsymbol{\mu}}}
\newcommand{\bfeta}{{\boldsymbol{\eta}}}
\newcommand{\bfomega}{{\boldsymbol{\omega}}}

\newcommand{\wh}{\widehat}
\newcommand{\wt}{\widetilde}

\newcommand{\ps}{\mathrm{ps}}  

\newcommand{\tdpm}[1][{\Gamma}]{\mathfrak{W}_{\operatorname{pm}}(#1)}
\newcommand{\tdps}[1][{\Gamma}]{\mathfrak{W}_{\operatorname{ps}}(#1)}

\newlength{\halfbls}\setlength{\halfbls}{.5\baselineskip}
\newlength{\halbls}\setlength{\halfbls}{.5\baselineskip}

\newcommand*{\Hrel}{\cH_{\text{rel}}^1}
\newcommand*{\Hrelbar}{\overline{\cH}^1_{\text{rel}}}

\newcommand*\interior[1]{\mathring{#1}}

\newcommand{\prodt}[1][\lceil j \rceil]{t_{#1}}
\newcommand{\prodtL}[1][\lceil L \rceil]{t_{#1}}

\title[The Euler sequence]
{The Chern classes and the Euler characteristic \\
of the moduli spaces of abelian differentials }

\author{Matteo Costantini}
\email{costanti@math.uni-bonn.de}
\address{Institut f\"ur Mathematik, Universit\"at Bonn,
Endenicher Allee 60,
53115 Bonn, Germany}
\author{Martin M\"oller}
\author{Jonathan Zachhuber}
\email{zachhuber@math.uni-frankfurt.de}
\address{Institut f\"ur Mathematik, Goethe-Universit\"at Frankfurt,
Robert-Mayer-Str. 6-8,
60325 Frankfurt am Main, Germany}
\email{moeller@math.uni-frankfurt.de}
\thanks{Research of the second and third author is supported
by the DFG-project MO 1884/2-1 and by the LOEWE-Schwerpunkt
``Uniformisierte Strukturen in Arithmetik und Geometrie''}

\begin{abstract}
For the moduli spaces of Abelian differentials, the Euler characteristic is one of the most basic intrinsic topological invariants.
We give a formula for the Euler characteristic that relies on
intersection theory on the smooth compactification by multi-scale
differentials. It is a consequence of a formula for the full Chern
polynomial of the cotangent bundle of the compactification.
\par
The main new technical tools are an Euler sequence for the
cotangent bundle of the moduli space of Abelian differentials
and computational tools in the Chow ring, such as normal
bundles to boundary divisors.
\end{abstract}
\maketitle
\tableofcontents

\section{Introduction}

Only few aspects of the topology of the moduli spaces of holomorphic
or meromorphic Abelian differentials $\bP\omoduli[g,n](\mu)$ with
singularities of type~$\mu = (m_1,\ldots,m_n)$ are currently known, such
as the connected components (\cite{kozo1}, \cite{boissymero}), and partial
information about
(quotients of) the fundamental group. This paper provides an expression
for the Chern classes of the cotangent bundle of the compactified
moduli spaces of abelian differentials and a formula to compute the
Euler characteristic of these moduli spaces.
\par
The moduli spaces of Abelian differentials can be thought of as
relatives of the moduli space of curves~$\moduli[g,n]$, for which the Euler
characteristic was computed in \cite{HaZa} using a cellular decomposition
(given by the arc complex) and
counting of cells. Our strategy here is quite different. While the Euler
characteristic is an intrinsic quantity associated to
$\bP\omoduli[g,n](\mu)$, our strategy heavily uses the compactification
$\bP\LMS$ constructed in \cite{LMS} and all its properties that make it
quite similar to the Deligne-Mumford  compactification~$\barmoduli[g,n]$
of $\moduli[g,n]$. Moreover, our strategy is not available
to compute the Euler characteristic~$\barmoduli[g,n]$, as it rather
mimics the case of the projective space~$\bP^d$: The unprojectivized
moduli spaces $\omoduli[g,n](\mu)$ are linear manifolds and thus the
cotangent bundle of $\bP\omoduli[g,n](\mu)$ is governed by the Euler
sequence, as in the case of~$\bP^d$. 
\par
Using this strategy we obtain the complete information about the Chern
classes of the (logarithmic) canonical bundle of the compactified moduli
spaces of Abelian differentials, and thus e.g.\ the $\chi_y$-genus. A
special case, the formula for the canonical class, is particularly easy
to state. We recall that the boundary divisors in $\bP\LMS$ are the
divisor $D_{\text{h}}$ of irreducible curves
with one self-node and the divisors $D_\Gamma$ parameterized
by level graphs~$\Gamma \in \LG_1(\LMS)$ that have one level below the
zero level and no horizontal edges (joining vertices of the same level).
As for the moduli space of curves, the boundary divisors are nearly
(in a sense that we elucidate further down) a product of two lower-dimensional
moduli spaces, corresponding to top and bottom level. 
Those boundary divisors~$D_\Gamma$ come with the integer~$\ell_\Gamma$,
the least common multiple of the prongs~$\kappa_e$ along the edges, see
Section~\ref{sec:DegUndeg} for a review of these notions. We let
$\xi = c_1(\cO(-1))$ be the first Chern class of the tautological bundle
on $\bP\LMS$ (see Section~\ref{sec:ELG}). 
\par
\begin{theorem} \label{thm:c1cor}
The first Chern class of the logarithmic cotangent bundle of the
projectivized compactified moduli space $\overline{B}  = \bP\LMS$ is
\be \label{eq:firstChern}
\c_1(\Omega^1_{\overline{B}}(\log D)) \= N \cdot \xi + \sum_{\Gamma \in 
\twolev} (N-N_\Gamma^\top) \ell_\Gamma  [D_\Gamma] \qquad \in \CH^1(\overline{B})\,,
\ee
where $N = \dim(\LMS)$ and where $N_\Gamma^\top:=\dim(B_{\Gamma}^\top)$
is the dimension of the unprojectivized top level stratum in~$D_\Gamma$.
\end{theorem}
\par
To compute the Euler characteristic we need to understand the top
Chern class as we recall in Section~\ref{sec:backEuler} along with standard
terminology from intersection theory. To state a formula for the full
Chern character we need to recall a procedure that also determines adjacency
of boundary strata. It is given by undegeneration maps~$\delta_i$ that
contract all the edges except those that cross from level~$-i+1$ to level~$-i$, 
see Section~\ref{sec:DegUndeg} and Figure~\ref{cap:H2} in Section~\ref{sec:examples}.
This construction can obviously be generalized so that a larger
subset of levels remains, for example the complement of~$i$, denoted
by the  undegeneration map~$\delta_i^\complement$. We can now define for 
any graph $\Gamma \in \LG_L(\ol{B})$ with $L$ levels below zero and without 
horizontal edges the quantity $\ell_\Gamma = \prod_{i=1}^L \ell_{\delta_{i}(\Gamma)}$.
\par
\begin{theorem} \label{intro:Chern}
The Chern character of the logarithmic cotangent
bundle is 
\bes
\ch(\Omega^1_{\overline{B}}(\log D)) \= e^{\xi} \cdot \sum_{L=0}^{N-1}
\sum_{ \Gamma \in \LG_L(\ol{B})}
\ell_{\Gamma}\left(N-N_{\delta_{L}(\Gamma)}^T\right) \fraki_{\Gamma *}
\Bigl(\prod_{i=1}^{L} \td \left(\cN_{\Gamma/\delta_{i}^\complement(\Gamma)}^{\otimes -\ell_{\delta_i(\Gamma)}}
\right)^{-1}\Bigr) \,, 
\ees
where $\cN_{\Gamma/\delta_{i}^\complement(\Gamma)}$ denotes the normal bundle 
of $D_\Gamma$ in $D_{\delta_{i}^\complement(\Gamma)}$, where $\td$ is the
Todd class  and $\fraki_{\Gamma}: D_\Gamma\hookrightarrow \overline{B}$
is the inclusion map.
\end{theorem}
\par
We also give closed expressions for the Chern polynomial in Theorem~\ref{thm:cpoly},
both fully factored and as a sum over level graphs.
\par
To compute the Euler characteristics, we can simplify this expression
significantly.
Moduli spaces of Abelian differentials are not homogeneous spaces and we should
not expect a proportionality between the top Chern class and the Mazur-Veech
volume form (\cite{masur82}, \cite{veech82}). For comparison we
note however that Mazur-Veech volumes of holomorphic minimal strata (where
$\mu = (2g-2)$) in each genus are essentially given by the top $\xi$-power
(\cite{SauvagetMinimal}). For non-minimal holomorphic strata (that is,
if all $m_i \geq 0$) this top $\xi$-power is zero and the
Mazur-Veech volume is computed by a product of $\xi^{2g-1}$ and $\psi$-classes
(\cite{CMSZ}). The top $\xi$-powers of all levels of all strata -- and only
these -- are combined to give the Euler characteristic
of~$\bP\omoduli[g,n](\mu)$. One thus needs the top $\xi$-powers for
meromorphic moduli spaces, even if one might be only interested in the
holomorphic case. Let $K_\Gamma = \prod_{e} \kappa_e$ be the product
of the prongs over all edges of~$\Gamma$.
\par
\begin{theorem} \label{intro:ECformula}
The orbifold Euler characteristic of the moduli space $\omoduli[g,n](\mu)$
is the dimen\-sion-weighted sum over all level graphs~$\Gamma \in \LG_L(\ol{B})$
without horizontal nodes
\be
\chi(\omoduli[g,n](\mu)) \= (-1)^{d} \,\sum_{L=0}^d \,\sum_{\Gamma \in \LG_L(\ol{B})}  
\frac{K_\Gamma \cdot N_\Gamma^\top}{|\Aut(\Gamma)|} \cdot
 \prod_{i=0}^{-L}  \int_{B_\Gamma^{[i]}} 
\xi_{B_\Gamma^{[i]}}^{d_{\Gamma}^{[i]}}
\ee
of the product of the top power of the first Chern class~$\xi_{B_\Gamma^{[i]}}$ 
of the tautological bundle at each level, where $d_{\Gamma}^{[i]} =
\dim(B_\Gamma^{[i]})$ and $d = \dim(\ol{B}) = N-1$.
\end{theorem}
\par
The  stratum $B_\Gamma^{[i]}$ at the level~$i$ of a
graph~$\Gamma$ is defined in Section~\ref{sec:comgenstra}. 
\par
\begin{figure}[h]
$$ \begin{array}{|c|c|c|c|c|c|c|c|c|}
\hline  &&&&&&&& \\ [-\halfbls] 
\mu  & (0) & (2) & (1,1) & (4) & (3,1) & (2,2) & (2,1,1) & (1,1,1,1)\\
[-\halfbls] &&&&&&&& \\ 
\hline &&&&&&&& \\ [-\halfbls]
\chi(B) %
& -\frac{1}{12}  &  -\frac{1}{40} & \frac{1}{30} & - \frac{55}{504} & 
\frac{16}{63} & \frac{15}{56} & -\frac{6}{7} & \frac{11}{3} \\
[-\halfbls] &&&&&&&& \\
\hline &&&&&&&& \\ [-\halfbls]
\mu  & (6) & (5,1) & (4,2) & (3,3) & (4,1,1) & (3,2,1) & (2,2,2) & (8)\\
[-\halfbls] &&&&&&&& \\
\hline &&&&&&&& \\ [-\halfbls]
\chi(B) %
& -\frac{1169}{720}  &  \frac{27}{5} & \frac{76}{15} & \frac{188}{45} & 
-\frac{200}{9} & -\frac{96}{5} & -\frac{187}{10} & -\frac{4671}{88} \\
[-\halfbls] &&&&&&&& \\
\hline
\end{array}
$$
\captionof{table}[foo2]{Euler characteristics of some holomorphic strata}
\label{cap:EulerHolo}
\end{figure}
Table~\ref{cap:EulerHolo} gives the Euler characteristics of some strata of
holomorphic differentials. A table of values of top $\xi$-powers and more
examples are provided in Section~\ref{sec:examples}. The evaluation of these
formulas is performed by a sage package {\tt diffstrata} that builds on
the package {\tt admcycles} for computation in the moduli space of curves
(\cite{DSvZ}). 
Specifically, the evaluation of tautological classes below is performed
using the formula for fundamental classes of strata conjectured in
\cite{fapa} and \cite{schmittDimTh} and proven recently in \cite{BHPSS}
based on results from \cite{HolmesSchmitt}.
The algorithms in this package are explained in \cite{CoMoZadiffstrata}.
\par
\medskip
\paragraph {\bf The Euler sequence.} Next we outline the ingredients
needed to prove these theorems. 
Recall that for projective space the Euler sequence is the exact sequence
\be 
0\longrightarrow \Omega^1_{\bP(V)} \longrightarrow
\cO_{\bP(V)}(-1)^{\oplus \dim(V)}
\overset{\ev}{\longrightarrow}   \cO_{\bP(V)}\longrightarrow 0\,.
\ee
Over the moduli space $\overline{B}= \bP\LMS$ this admits the following
generalization, that combines Theorem~\ref{thm:EulerDE} and Theorem~\ref{thm:coker}.
It states roughly that, using the local projective structure induced by
period coordinates, in the interior of the stratum we indeed have an
Euler sequence, if we replace the direct sum in the middle of the sequence
by a local system. This local system naturally extends across the boundary,
but the Euler sequence needs a correction term that we determine explicitly
via a local computation using perturbed period coordinates.
\par
\begin{theorem} \label{intro:Euler}
The logarithmic cotangent bundle sits in an exact sequence
\begin{equation} \label{intro:maineq}
0\longrightarrow \Omega^1_{\overline{B}}(\log D) \Bigl(
-\sum_{\Gamma\in \twolev} \ell_\Gamma D_\Gamma \Bigr) \to \cK \to
\cC\longrightarrow 0\,,
\end{equation}
where $\cC$ is an explicitly computable sheaf (see Lemma~\ref{le:defcalC})
supported on the boundary and where the vector bundle~$\cK$ on $\overline{B}$
fits into the Euler exact sequence
\be \label{intro:EulerExt}
0\longrightarrow \cK \longrightarrow (\Hrelbar)^\vee\otimes \cO_{\overline{B}}(-1)
\overset{\ev}{\longrightarrow}   \cO_{\overline{B}}\longrightarrow 0\,.
\ee
Here $\Hrelbar$ is the Deligne extension of the local system
of relative cohomology.
\end{theorem}
\par
This theorem directly implies Theorem~\ref{thm:c1cor}. To deduce the
other two theorems, we need to exploit further information on the
Chow ring of the compactification.
\par
\medskip
\paragraph {\bf The tautological rings.} In Section~\ref{sec:tautring} 
we define a notion of a system of {\em tautological rings~$R^\bullet(\LMS)$}
inside the Chow rings of the compactifications $\bP\LMS$ of the projectivized
strata $\bP\omoduli[g,n](\mu)$ that have been constructed in \cite{LMS}.
This is the smallest system of $\bQ$-subalgebras
$R^\bullet(\bP\LMS) \subset \CH^\bullet(\bP\LMS)$ which
\begin{itemize}
\item contains the $\psi$-classes attached to the marked points,
\item is closed under the pushfoward of the map forgetting a
regular marked point (a zero of order zero), and
\item is closed under the clutching homomorphisms~$\zeta_{\Gamma,*}p^{[i],*}$,
defined in Section~\ref{sec:clutch}.
\end{itemize}
\par
For the moduli space of curves $\barmoduli$ the clutching homomorphisms build
a boundary divisor from a product of two smaller moduli spaces, or from just one
for the irreducible boundary divisor that plays the role of our~$D_{\text{h}}$.
For multi-scale differentials the situation is  more involved. First,
to relate $D_\Gamma$ to a product of moduli spaces, we need to allow
spaces of disconnected curves
and allow to impose residue conditions since the levels of~$\Gamma$
have that property. We define such {\em generalized strata} and their modular
compactification in Section~\ref{sec:clutch}. Second, the boundary
divisors~$D_\Gamma$ do not admit maps to such generalized strata, since the
levels are tied to one another by a datum of the multi-scale differential,
the prong-matchings. We need to construct a covering space
$c_\Gamma: D_\Gamma^s \to D_\Gamma$ that removes the stacky structure of $D_\Gamma$,
which has two properties. First, there are projection maps $p^{[i]}$
from~$D_\Gamma^s$ to generalized strata and second, there are clutching maps
$\zeta_\Gamma: D_\Gamma^s \to \bP\LMS$, that factor as $\zeta_\Gamma =
\fraki_\Gamma \circ c_\Gamma$ into the finite map~$c_\Gamma$ and a closed
embedding~$\fraki_\Gamma$. (The upper index of $D_\Gamma^s$
refers to the use of the simple twist group as in \cite{LMS} in the
construction of this covering.)
\par
\begin{theorem}	\label{thm:addgen}
For each~$\mu$, a finite set of additive generators of $R^\bullet(\bP\LMS)$
is given by the classes
\be \label{eq:addgenR}
\zeta_{\Gamma_*} \Bigl(\,\prod_{i=0}^{-L} p^{[i],*} \alpha_i\,\Bigr) 
\ee
where $\Gamma$ runs over all level graphs for all boundary strata of
$\bP\LMS$ including the trivial graph and where~$\alpha_i$ is a
monomial in the $\psi$-classes supported on level~$i$
of the graph~$\Gamma$. 
\par
The tautological ring contains the $\kappa$-classes and all level-wise
tautological line bundle classes $\zeta_{\Gamma_*} p^{[i],*} \xi_{B_\Gamma^{[i]}}$
of all level graphs~$\Gamma$.
\end{theorem}
\par
An algorithm to perform the multiplication of these generators is
given along with the proof of Theorem~\ref{thm:addgen} in
Section~\ref{sec:tautring}. An important technical tool in the proof
is the excess intersection formula (see Proposition~\ref{prop:pushpullcomm})
which, like the above formulation of the tautological ring, has large
structural similarities with the case of the Deligne-Mumford
compactification. It is useful only if the normal bundles to the
boundary divisors are known. Contrary to the
Deligne-Mumford compactification the normal bundles to the boundary
divisors defined by two-level graphs are indeed bundles,
i.e.\ those boundary divisors do not self-intersect (see
Section~\ref{sec:strucbd}). Along with the clutching morphisms
we define in Section~\ref{sec:clutmor} the tautological bundles on the top and
bottom level strata of divisors and their first Chern classes~$\xi^\top$
and~$\xi^\bot$.  In Section~\ref{sec:nb} we show:
\par
\begin{theorem} \label{thm:intro:nb}
  The  normal bundle $\cN_\Gamma$ of a divisor~$D_\Gamma \in \LG_1(\ol{B})$
  has first Chern class
\be \label{eq:nbintro}
c_1(\cN_\Gamma) \= \frac{1}{\ell_\Gamma} \bigl(-\xi_\Gamma^\top
- c_1(\cL_\Gamma^\top) + \xi_\Gamma^\bot \bigr)\quad \text{in} \quad
\CH^1(D_{\Gamma})\,,
\ee
where $\cL_\Gamma^\top$ defined in~\eqref{eq:defLtop} is a line bundle
supported on the boundary of~$D_\Gamma$ where the top-level stratum
degenerates further.
\end{theorem}
\par
We define tautological rings $R^\bullet(D_\Gamma)$ of strata using the
analogs of the additive generators~\eqref{eq:addgenR} and as a consequence
of the preceding theorem the normal bundle of each~$D_\Gamma$ belongs to
the tautological ring $R^\bullet(D_\Gamma)$.
\par
\medskip
\paragraph {\bf Organization and strategy of proof.} After recalling
some background on intersection theory in Section~\ref{sec:backEuler},
we provide the necessary details on the compactification $\bP\LMS$
in Section~\ref{sec:LMS}. Each of the levels of a level graph gives rise
to the notion of generalized strata, that are defined in
Section~\ref{sec:clutch}. There we also introduce the covering of
boundary strata that allows a decomposition into a product of levels.
Section~\ref{sec:strucbd} provides a dimension count argument that implies the
smoothness of all non-horizontal boundary strata and that is at the heart
of a formula for exponentials of sums of over boundary graphs. This formula
allows, together with Theorem~\ref{thm:intro:nb},  the passage from
Theorem~\ref{intro:Euler} to Theorem~\ref{intro:ECformula}.
Section~\ref{sec:eulerseq} proves the restriction of Theorem~\ref{intro:Euler}
to the interior of $\bP\LMS$ and Section~\ref{sec:nb} proves
Theorem~\ref{thm:intro:nb}. In Section~\ref{sec:tautring}
we prove the properties of the tautological ring announced above. In
Section~\ref{sec:Chern} a local calculation at
the boundary  completes the proof of Theorem~\ref{intro:Euler} and
computations in the tautological ring allows the passage from
Theorem~\ref{intro:Chern} to Theorem~\ref{intro:ECformula}.
\par
\medskip
The strategy used here applies to other linear manifolds for which
a compactification similar to that in \cite{LMS} has been constructed.
It is already available for meromorphic $k$-differentials for $k>0$
(see \cite{CoMoZa})
and expected to work for any affine invariant manifold. The proof of the
main theorems should carry over with very few adaptations. We hope to
address these cases in a sequel.
\par
\medskip
\subsection*{Acknowledgments} We thank Matt Bainbridge, Dawei Chen,
Vincent Delecroix, Quentin Gendron, Sam Grushevsky, 
and Johannes Schmitt for inspiring discussions and help with implementation
of the algorithmic part of the project.  We are also grateful to the
Mathematical Sciences Research Institute (MSRI, Berkeley)
and the Hausdorff Institute for Mathematics (HIM, Bonn),
where significant progress on this paper was made during their programs
and workshops. The authors thank the MPIM, Bonn, for hospitality
and support for \cite{sage} computations. 

\section{Euler characteristics via logarithmic differential forms}
\label{sec:backEuler}

This section connects Euler characteristic to integrals of characteristic
classes of the sheaf of logarithmic differential forms. The following
proposition is certainly known, but not easy to locate in the
literature. We use the occasion to give a self-contained proof,
see also \cite{fiori}, and recall some standard exact sequences. 
\par
\begin{prop} \label{prop:chiviaTlog}
Let $\overline{B}$ be a compact smooth $k$-dimensional
manifold, let~$D$ be a normal crossing divisor and
$B = \overline{B} \smallsetminus D$. Then the Euler characteristic
of~$B$ can be computed as integral
\be \label{eq:chiviaint}
\chi(B) \= (-1)^k \int_{\overline{B}} \,c_k(\Omega^1_B(\log D)) 
\ee
over the top Chern class of the logarithmic cotangent bundle.
\end{prop}
\par
In all our applications, $\overline{B}$ will be a compact orbifold
or proper smooth Deligne-Mumford stack. 
We work throughout with orbifold Euler characteristics, and since
then both sides of~\eqref{eq:chiviaint} are multiplicative in the
degree of a covering, we can apply Proposition~\ref{prop:chiviaTlog}
verbatim.

\subsection{The compact case and the Riemann-Roch theorem}

We start with the proof of the special case of the main
theorem.
\par
\begin{prop} \label{prop:chiviaTB}
If $B = \overline{B}$ is smooth, compact and $k$-dimensional, then
\be
\chi(B) \=  \int_{\overline{B}} \, c_k(T_B) \,.
\ee
\end{prop}
\par
We start by recalling some intersection theory. Let $\cE$ be a
holomorphic vector bundle on $B$. Denote by $c_i\coloneqq c_i(\cE)\in \CH^i(B)$
the $i$th Chern class of $E$. Recall that $c_0 = 1$ and $c_i = 0$
for $i > \rk E \eqqcolon r$. The \emph{total Chern class} of~$\cE$
is the formal sum $\c(\cE) = 1 + c_1 + \dotsc + c_r.$ in $\CH(B)$.
Splitting formally $\c(\cE) = \prod_{i=1}^r (1+\alpha_i)$ into the Chern roots,
the \emph{Chern character} is defined as the formal power series
\[\ch(E) = \sum_{i=1}^r \exp(\alpha_i) \= \sum_{s\geq 0}\frac{1}{s!} 
 \sum_{i=1}^r \alpha_i^s
\= {\rm rk}(E) + c_1 + \frac{1}{2}(c_1^2 - 2c_2) + \cdots .\]
Furthermore, the Todd class is defined as
\[\td(E) = \prod_{i=1}^r \frac{\alpha_i}{1-\exp(-\alpha_i)}
\= 1 + \frac{1}{2} c_1 + \frac{1}{12}(c_1^2 + c_2) + \frac{1}{24} c_1c_2 + \cdots
.\]
\par
The Grothendieck-Riemann-Roch theorem in the case of a map $f:X \to Y$ and
for the special case of that the  higher direct images $R^if_* \cE$
vanish, states that
\be \label{GRR}
\ch(f_* \cE)\cdot \td(T_Y) \= f_*(\ch(\cE)\cdot\td(T_X))\,.
\ee
\par
\begin{proof}[Proof of Proposition~\ref{prop:chiviaTB}]
For a topological proof, see e.g.\ \cite[Proposition~11.24]{BottTu}.
Using the notations already set up, we give a quick proof if
moreover $B$ is K\"ahler. From the Borel-Serre identity
(\cite[Example~3.2.5]{Fulton}) on a $k$-dimensional manifold
\bes
c_k(T_B) \= \ch\Bigl(\sum_{j=1}^k (-1)^j \Omega_B^j\Bigr) \cdot \td(T_B)
\ees
and the application
\bes
 \int_{\overline{B}} \ch((-1)^j \Omega_B^j) \cdot \td(T_B) \= 
\sum_{\ell \geq 0} (-1)^{\ell+j} h^\ell(B,\Omega^j_B)
\ees
of Grothendieck-Riemann-Roch theorem for the map from~$B$ to a point
we get
\bes
\int_{\overline{B}} \,c_k(T_B) \= \sum_{\ell,j \geq 0} (-1)^{\ell+j} h^\ell(B,\Omega^j_B)
\= \chi(B)
\ees
by the Hodge decomposition. 
\end{proof}

\subsection{The non-compact case and log differential forms}

We suppose throughout that $D = \cup_{j=1}^s D_j$ is a reduced normal
crossing divisor, i.e., with distinct irreducible components~$D_i$
intersecting each other transversally. In this situation
$\Omega^1_{\overline{B}}(\log D)$ is defined to be the vector bundle
of rank~$n$ with the following local generators. In a neighborhood~$U$ of
a point where (say) the first $r \leq s$ divisors meet
and where $x_1,\ldots,x_k$ is a local coordinate system with
$D_j = \{x_j = 0\}$, then logarithmic cotangent bundle is defined by 
\be
\Omega^1_{\overline{B}}(\log D)(U) \= \Bigl\langle \frac{dx_1}{x_1}, \ldots,
\frac{dx_r}{x_r}, x_{r+1}, \ldots, x_{k} \Bigr \rangle\,
\ee
as an $\cO_{\overline B}(U)$-module. There is a fundamental exact sequence for
log differential forms, namely
\ba \label{eq:LogDiffSeq}
0&\to \Omega^1_{\overline B} \to \Omega^1_{\overline B}(\log D) \to \oplus_{j=1}^s
(\fraki_{j})_* \cO_{D_j} \to 0\,,
\ea
where $\fraki_j: D_j \to \overline B$ is the inclusion map.
More details can be found e.g.\ in \cite[Proposition~2.3]{esvibook}.
\par
\begin{proof}[Proof of Proposition~\ref{prop:chiviaTlog}]
We first reduce to the case that $D$ has simple normal crossings, i.e.,
to the case the the $D_j$ are all smooth. This can always be achieved
by an \'etale covering. Since both sides of \eqref{eq:chiviaint} are
multiplied by the degree under such a covering, we can assume simple
normal crossings. Our goal is to prove 
\bes
\int_{\overline B} c_k\Bigl(\Omega^1_{\overline{B}} \Bigl(\log \sum_{i \geq 2} D_i\Bigr)\Bigr)
\= \int_{\overline B} c_k\Bigl(\Omega^1_{\overline{B}}(\log D)\Bigr) - \int_{D_1}
c_{k-1}\Bigl(\Omega^1_{D_1} \Bigl(\log (\sum_{i \geq 2} D_i \cap D_1 \Bigr)\Bigr)\,,
\ees
The claim follows then from the additivity $\chi(B) + \chi(D)
= \chi(\overline{B})$ of the Euler characteristic,
Proposition~\ref{prop:chiviaTB} and an application to the preceding
identity to $B_j = \overline{B} \smallsetminus \cup_{i=j}^s D_j$.
\par
We consider the inclusion of the boundary divisor~$D_1$
and deduce from the ideal sheaf sequence that $c((\fraki_1)_*\cO_{D_1})
=(1-[D_1])^{-1}$ and that $c(\cN_{D_1}) = 1 + \fraki_1^*[D_1]$.
Moreover the normal bundle sequence $0\to \cT_{D_1} \to
\fraki_1^*\cT_{\overline{B}} \to \cN_{D_1} \to 0$ implies
\be \label{eq:LogdiffPB}
c(\Omega^1_{D_1}) \= \fraki_1^* \Bigl( c(\Omega^1_{\overline{B}}) \cdot
\frac{1}{1-[D_1]} \Bigr)\,.
\ee
On the other hand, the sequence~\eqref{eq:LogDiffSeq}  gives 
\be \label{eq:OmlogMult}
c(\Omega^1_{\overline{B}}(\log D)) \=  c(\Omega^1_{\overline{B}})
\cdot \frac{1}{1-[D_1]} \cdot \prod_{j=2}^s  \frac{1}{1-[D_j]}
\ee
and also
\be
 \quad c(\Omega^1_{D_1}(\log \Bigl(\sum_{i \geq 2}
D_i \cap D_1 \Bigr))) \=  c(\Omega^1_{D_1})
\cdot \prod_{j=2}^s  \frac{1}{1-[ D_1\cap D_j]}.
\ee
Hence comparing with~\eqref{eq:LogdiffPB} we get
\be \label{eq:OmRelPullback}
c\Bigl(\Omega^1_{D_1} \bigr(\log \Bigl(\sum_{i \geq 2}
D_i \cap D_1 \Bigr)\big) \Bigr)\=\fraki_1^*c\bigl(\Omega^1_{\overline{B}}(\log D)\bigr)\,.
\ee
Finally from \eqref{eq:OmlogMult} and from the appropriate version of the sequence~\eqref{eq:LogDiffSeq} we also get
\be
c(\Omega^1_{\overline{B}}(\log D)) \=\frac{1}{1-[D_1]} c\Bigl(\Omega^1_{\overline{B}}
\Bigl(\log \sum_{i \geq 2} D_i\Bigr)\Bigr).
\ee
The claim now follows by multiplying this last expression with $1-[D_1]$, integrating and taking the $k$-th coefficient,
using that $\int_{\overline{B}} [D]  \cdot c_{k-1}
(\Omega^1_{\overline{B}} (\log D)) = \int_D \fraki_1^*c_{k-1}
(\Omega^1_{\overline{B}}(\log D))$.
\end{proof}

\section{The moduli space of multi-scale differentials}
\label{sec:LMS}

We recall here from \cite{LMS} basic properties of the moduli space
of multi-scale differentials $\LMS$ and its projectivization $\bP\LMS$
that compactifies the moduli space $\bP\omoduli[g,n](\mu)$ of projectivized
meromorphic differentials. Throughout we suppose that
$\mu=(m_1,\dots,m_n)\in\bZ^n$ is the type of a differential, i.e.,
that $\sum_{j=1}^n m_j=2g-2$. We usually abbreviate
$B = \proj\omoduli[g,n](\mu)$ and $\overline{B} = \proj \LMS$.

\subsection{Enhanced level graphs} \label{sec:ELG}

To define strata and the ambient space in the meromorphic case, 
we assume that there are $r$ positive $m$'s, $s$ zeroes, and $l$ negative $m$'s,
with $r+s+l=n$, i.e.,~that we have
$m_1\ge \dots\ge m_r>m_{r+1}=\dots=m_{r+s} = 0>m_{r+s+1}\ge\dots \ge m_{n}$.
Note that $m_j=0$ is allowed, representing an ordinary marked point. A
pointed flat surface is usually denoted by $(X,\omega,\bz)$ where
$\bz = (z_1,\ldots,z_n)$ are the marked points corresponding to the
zeros, ordinary marked points, and poles of~$\omega$. The sections
over $\obarmoduli[g,n](\mu)$ corresponding to those marked points are
denoted by $\cZ_i$.  
We denote the polar part of $\mu$ by $\tilde\mu=(m_{r+s+1},\dots,m_n)$.
The strata of meromorphic differentials are then naturally defined
inside the twisted Hodge bundle 
$$K\barmoduli[g,n](\tilde{\mu}) \= f_* \Bigl(\omega_{\cX/\barmoduli[g,n]}
\Bigl(-\sum_{j=r+s+1}^n m_j \cZ_j\Bigr)\Bigr)$$
The strata are smooth complex substacks $\omoduli[g,n](\mu)$
of dimension $N = 2g-1+n$ in the holomorphic case $r=n$ and $N=2g-2+n$ in the
meromorphic case.
\par
To each boundary point in $D = \LMS \setminus \omoduli[g,n](\mu)$ there
is an associated {\em enhanced level graph} and~$D$ is stratified by the
type of this associated graph. Here a {\em level graph} is defined to be
a stable graph $\Gamma = (V,E,H)$,
\index[graph]{b002@$\Gamma$!  stable graph $\Gamma = (V,E,H)$}
with half-edges in~$H$ that are either
paired to form edges~$E$ or correspond to the $n$~marked points, together
with a total order on the vertices (with equality permitted).
The graph~$\Gamma$ is supposed to be connected here, from
Section~\ref{sec:tautring} on its components are in bijection with the
components of the flat surfaces the generalized stratum parameterizes.
For convenience we usually define the total order using a {\em level function}
$\ell: V(\Gamma) \to \bZ$, 
usually normalized to take values in
$\{0,-1,\ldots,-L\}$. We usually write $H_m = H \smallsetminus E$ for
the half-edges corresponding to the marked points. Moreover, an {\em
enhancement} (in \cite{fapa} or \cite{CMSZ} this number is called a
{\em twist}) is an assignment of a number $\kappa_e \geq 0$ to each
edge~$e$, so that $\kappa_e = 0$ if and only if the edge is horizontal.
The triple $(\Gamma, \ell, \{\kappa_e\}_{e \in E(\Gamma)})$ is called and
{\em enhanced level graph}. 
\index[graph]{b004@$(\Gamma, \ell, \{\kappa_e\})$! enhanced level graph}
We denote the closure of the boundary stratum parametrizing multi-scaled
differentials (as defined below) compatible with $(\Gamma, \ell, \{\kappa_e\})$
by $D_{(\Gamma, \ell, \{\kappa_e\})}$ or usually simply by~$D_\Gamma$. 
\par
\begin{theorem}[\cite{LMS}] \label{thm:recallLMS}
There is a proper smooth Deligne-Mumford stack\footnote{This is not exactly
the statement of the current version of \cite{LMS}. There, the space is
introduced as a compact orbifold or proper Deligne-Mumford stack with
finite quotient singularities at some boundary points. We anticipate here
the forthcoming version that improves the structure by changing the definition
of families of multi-scale differentials that locally represents the structure
of a quotient stack instead of the underlying quotient space.} 
$\proj \LMS$ that contains the projectivized stratum $\proj\omoduli[g,n](\mu)$
as open dense substack with the following properties.
\begin{itemize}
\item[(i)] The boundary $\proj \LMS \smallsetminus \proj\omoduli[g,n](\mu)$
is a normal crossing divisor.
\item[(ii)] The codimension of a boundary stratum $D_\Gamma$ in $\bP\LMS$
is equal to the number of horizontal edges plus the number $L$ of
levels below zero. %
\end{itemize}
\end{theorem}
\par
In particular, the {\em boundary divisors} consist of the divisor $D_{\text{h}}$
(if $g \geq 1$) with  just one non-separating horizontal edge and
the ('vertical') boundary divisors indexed by two-level graphs without
horizontal edges. (Separating horizontal edges are impossible because of
the absence of a residue.) We give local coordinates near the boundary
divisors in Section~\ref{sec:bdperiod}.
\par
Note that the boundary strata $D_\Gamma$ may be empty for some enhanced level
graphs. Deciding non-emptyness is the same as the realizability question that
was addressed in \cite{MUWrealize} purely in terms of graphs. The general
version taking into account the residue conditions is stated in the
algorithmic part of \cite{CoMoZadiffstrata}.
Note that these boundary strata may also be non-connected, see the
discussion of prong-matching equivalence classes below.
\par
Recall that the construction of $\proj \LMS$ in \cite{LMS} gives a
morphism
$\proj \LMS = \overline{B} \to
\proj \left(f_*\omega_{\overline{\cX}/\overline{\cM}_{g,n}}
(-\sum_{j=r+s+1}^n \mu_j \cZ_j)\right)$
to the projectivised twisted Hodge bundle over the Deligne-Mumford
compactification. The line bundle  $\cO_{\overline{B}}(-1)$ is the pullback
of the tautological bundle from there.

\subsection{Twisted differentials and multi-scale differentials}
\label{sec:twdmsd}

The space $\LMS$ is  a moduli stack for families of  a certain collection of
differentials, called multi-scale differentials, and this modular
interpretation  will be used e.g.\ in the Section~\ref{sec:clutch}
to define clutching maps and projection maps at the boundary.
We will however refer to $\LMS$ as a moduli space to stick to the
commonly used terminology. We recall the definition of a single multi-scale
differentials, referring
for full details of the definition for families to \cite{LMS}. We will
recall further details where needed.
\par
When referring to prongs we fix a direction in $S^1$ throughout, say
the horizontal direction. Suppose that a differential~$\omega$ has
a zero of order $m \geq 0$ at~$q \in X$. The differential~$\omega$
selects inside the real projectivized tangent space $P_q =  T_pX/ \bR_{>0}$
a collection of $\kappa = m+1$ horizontal (outgoing) {\em prongs}
at~$q$, the tangent vectors $\bR_{>0} \cdot \zeta^i_\kappa \partial/\partial z$
in a chart where $\omega = z^m dz$ is in standard form and where $\zeta_\kappa$
is a primitive $\kappa$-th root of unity. We denote them by $P_q^{{\rm out}}
\subset P_q$. The prongs are equivalently the tangent vectors
to the outgoing horizontal rays. Dually, if $\omega$ has a pole of
order $m \leq -2$, then $\omega$ has $\kappa = -m-1$ horizontal
(incoming) {\em prongs} at~$q$, denoted by $P_q^{{\rm out}} \subset P_q$,
the tangent vectors $\bR_{>0} \cdot -\zeta^i_\kappa \partial/\partial z$
in a chart where $\omega = z^m dz$.
\par
We start with an auxiliary notion of differentials from \cite{BCGGM1}.
Given a pointed stable curve~$(X,\bfz)$, a {\em twisted differential} is
a collection of differentials $\eta_v$ on each component~$X_v$ of~$X$, that
is {\em compatible with a level structure} on the dual graph~$\Gamma$ of~$X$,
i.e.\ vanishes as prescribed by~$\mu$ at the marked points~$z$, satisfies
the matching order condition at vertical nodes, the matching residue
condition at horizontal nodes and global residue condition of \cite{BCGGM1}.
We usually group the differentials on the components of level~$i$ of~$X$
to form the collection~$\eta_{(i)}$ and refer to a
twisted differential by $\bfeta = (\eta_{(i)})$. 
\par
A {\em multi-scale differential of type $\mu$} on a stable curve~$X$
consists of an enhanced level structure~$(\Gamma,\ell,\{\kappa_e\})$
on the dual graph~$\Gamma$ of~$X$, a twisted differential of type~$\mu$
compatible with the enhanced level structure, and a prong-matching for
each node of~$X$ joining components of non-equal level. Here the
{\em compatibility with the enhanced level structure} requires that
at each of the two points~$q^\pm$ glued to form the node
corresponding to the edge~$e \in E(\Gamma)$ the number of prongs
of the differential $\bfeta$ is equal to $\kappa_e$. Moreover, a {\em
prong-matching} is an order-reversing isometry $\sigma_q: P_{q^-} \to
P_{q^+}$ that induces a cyclic order-reversing bijection
$\sigma_q: P^{\rm in}_{q^-} \to P^{\rm out}_{q^+}$ between the incoming
prongs at~$q^-$ and the outgoing prongs at~$q^+$.
\par
Finally we state the equivalence relation on multi-scale differentials
used to construct~$\LMS$. The space of isomorphism classes of
twisted differentials compactible with~$(\Gamma,\ell,\{\kappa_e\})$
and a prong-matching is a finite cover~$\tdpm$ of a product of strata
(set $k=1$ in \cite[Section~3.3]{CoMoZa} or see \cite[Section~5]{LMS}
for the viewpoint with an additional Teichm\"uller marking).
Multi-scale differentials only retain the information on lower level up
to projectivization. This rescaling of the lower levels is roughly
given by a multiplicative torus $T^{L(\Gamma)}$. More precisely, the
universal cover $\bC^{L(\Gamma)} \to T^{L(\Gamma)}$ acts by rescaling the
differentials on each level and simultaneously by fractional Dehn twists on the
prong-matching. In fact a subgroup acts trivially, the twist
group $\Tw[\Gamma]$ that we describe in detail in Section~\ref{sec:PM}.
So the action factors through the action of the quotient $T_\Gamma = \bC^L/\Tw[\Gamma]$, called the {\em level
rotation torus}, and two
multi-scale differentials are defined to be equivalent, if they
differ by the action of $T_\Gamma$.
\par
The projectivized space $\bP\LMS$ parametrizes projectivized multi-scale
differentials, where $\bC^*$ acts by simultaneously rescaling the differentials
on all levels and leaving the prong-matchings untouched.
\par

\subsection{Divisors, degeneration, undegeneration}  \label{sec:DegUndeg}

We let $\LG_L(B)$ be the set of all enhanced $(L+1)$-level graphs without
\index[graph]{b012@$\LG_L(B)$!  Level graphs of~$B$ with $L$ levels below zero}
horizontal edges. Recall that boundary divisors of $\overline{B}$ are~$D_{\text{h}}$
and~$D_\Gamma$ for $\Gamma \in \twolev$. For later use we define
\be 
D \= D_{\text{h}} + \sum_{\Gamma\in \twolev}  D_\Gamma
\ee
to be the total boundary divisor. The structure of the normal crossing
boundary of $\LMS$ is encoded by {\em undegenerations}. Given a non-horizontal
level graph~$\Gamma$ with $L+1$~levels, the associated boundary
stratum~$D_{\Gamma}$ is contained in the intersection of~$L$ boundary
divisors~$D_{\Gamma_i}$ for $i=1,\ldots,L$ and we can describe this inclusion
as follows. View the $i$-th level passage as a horizontal line just above
level~$-i$. Contract in $\Gamma$ all
edges that do not cross this horizontal line to obtain a contraction map
$\delta_i: \Gamma \to \Gamma_i$ of enhanced level graphs, where $\Gamma_i$
obtains a two-level structure with the top level corresponding to the
components above the horizontal line and the bottom level those below
that line. We call this the $i$-th undegeneration
of~$\Gamma$. This  can be generalized for any subset
$I = \{i_1,\dots,i_n\}  \subseteq \{1,\dots,L\}$ %
and results in the undegeneration map
\bes
\delta_{i_1,\dots,i_n} \colon \LG_L(B)\to \LG_{n}(B)\,,
\ees
which contracts all the passage levels of a non-horizontal level
graph $D_{\Gamma}$ except for the passages between levels $-i_{k}+1$ and $-i_{k}$,
for those $i_k \in I$. For notational convenience we define
$\delta_I^\complement = \delta_{I^\complement}$.
\par
A {\em degeneration} of level graphs is simply the inverse of an
undegeneration. It is convenient to have a symbol to express this dual
process and we write 
\be
\Gamma {\rightsquigarrow}
\wh{\Delta} \qquad \text{or} \qquad
\Gamma \overset{[i]}{\rightsquigarrow}
\wh{\Delta} 
\ee
for a general undegeneration resp.\ specifically for an undegeneration
where the $i$-th level is split into two levels.
\par
\begin{rem}\label{rem:conventionindex}
With the convention used here and in all of the rest, the  levels of a level graph with $L+1$ levels are indexed by negative integers $\{0,-1,\dots,-L\}$, while the level passages are indexed by positive integers $\{1,\dots,L\}$. This implies for examples that $\Gamma \overset{[i]}{\rightsquigarrow}\wh{\Delta} $ is equivalent to $\Gamma=\delta_{(-i+1)}^\complement(\wh{\Delta})$.
\end{rem}
\par
Note the map of graphs $\delta_I$ is only
well-defined up to post-composition by automorphism of the enhanced level
graph $\Gamma_i$. Taking this into account will be important for
intersection theory, see Proposition~\ref{prop:pushpullcomm}.

\par

\subsection{Prong-matchings and their equivalence classes}  \label{sec:PM}

In this section we illustrate the amount of combinatorial information
encoded in the notion of a prong-matching, given that we also have to
take into account the action of the level rotation torus. We start with
a recurrent example.
\par
{\em Case of a level graph $\Gamma \in \twolev$, i.e.\ a divisor $D_\Gamma$
different from $D_{\text{h}}$}. Such an enhanced level graph has  $|E(\Gamma)|$~edges
each of which  carries the information of the prongs, and consequently then
there are $K_\Gamma = \prod_{e \in E(\Gamma)} \kappa_e$ prong-matchings. However,
this does not imply that locally
$D_\Gamma$ is a degree $K_\Gamma$-cover of the product of the moduli spaces
corresponding to the upper and lower level. Instead the effect of
projectivization of the lower level on prong-matchings has to be taken into
account. This effect is given by the action of the  level rotation
group~$R_\Gamma \cong \bZ^L \subset \bC^L$ in the universal cover of the
level rotation torus. This group $R_\Gamma$ acts diagonally turning the
prong-matching at each edge by one (in a fixed direction). The stabilizer of
a prong-matching is the {\em twist group}~$\Tw[\Gamma]$ referred to above. It
is isomorphic to  $\ell_\Gamma\bZ$ as subgroup of~$R_\Gamma$ where
\be \label{eq:defellGamm}
\ell_\Gamma \= \lcm(\kappa_e \colon e \in E(\Gamma)) \,.
\ee
\index[graph]{b030@$\ell_\Gamma$! for a divisor $D_\Gamma$: the
$\lcm$ of the prongs of all edges}
Orbits of~$R_\Gamma$ are
also called {\em equivalence classes of prong-matchings}. For divisors
there are $g_{\Gamma} := K_{\Gamma}/\ell_{\Gamma}$ such equivalence classes.
\par
{\em For a general level graph~$\Delta$} the situation is more complicated and
the compactification $\LMS$ acquires a non-trivial quotient stack structure that
can be computed as follows. As above, there are $K_\Delta = \prod_{e \in E(\Delta)}
\kappa_e$ prong matchings. Now the level rotation group is~$R_\Delta
\cong \bZ^{L}$, where the $i$-th factor twists by one all prong-matchings
that cross the horizontal line above level~$-i$. The stabilizer of a
prong-matching is still called the twist group~$\Tw[\Delta]$. However, this
group is no longer a product of the
level-wise factors. In  fact, for each~$i \in \bN$ the twist group of the
level-undegeneration $D_{\delta_i(\Delta)}$ is a subgroup of~$\Tw[\Delta]$ and we
call the sum of these subgroups the {\em simple Twist group} $\sTw[\Delta]$.
The generic stack structure of $D_\Delta$ is given by the product of the action
of the group $\Aut(\Delta)$ of enhanced level graphs automorphisms and a
cyclic group of order
\[e_{\Delta}=[\Tw[\Delta]: \sTw[\Delta]].\]
 The number of prong-matching
equivalence classes is
\be \label{eq:defggen}
g_\Delta \,:= \,|R_\Delta-\text{orbits on the set $K_\Delta$}|
\= K_\Delta/[R_\Delta:\Tw[\Delta]]\,.
\ee
These indices can easily be computed using the elementar divisor theorem.

\begin{figure}[ht]
\begin{tikzpicture}[very thick]
\fill (0,0) coordinate (x0) circle (2.5pt); \node [above] at (x0) {$X_{(0)}$};
\fill (1,-1) coordinate (x1) circle (2.5pt); \node [below right] at (x1) {$X_{(-1)}$};
\fill (-1,-2) coordinate (x2) circle (2.5pt); \node [below right] at (x2) {$X_{(-2)}$};

 \draw[] (x0) -- node[right]{$a$} node[left]{$e_{1}$} (x1);
 \draw (x0) --node[left]{$b$} node[right]{$e_{2}$} (x2);
\draw (x1) -- node[below]{$c$} node[above]{$e_{3}$} (x2);
\end{tikzpicture}
\quad \quad 
\begin{tikzpicture} [very thick]
\begin{scope}
\node[comp] (T) {};
\node[comp] (B1) [below=of T] {}
edge [bend left] 
node [xshift=.3cm,yshift=.9cm] {$Y_{(0)}$} 
node [xshift=-0.2cm] {$b$} (T)
edge [bend right] 
node [xshift=.2cm] {$a$} 
node [xshift=.4cm,yshift=-.7cm] {$Y_{(-1)}$} (T);
\node[comp] (B2) [below=of B1] {}
edge [bend left] 
node [xshift=-0.2cm] {$b$} 
node [] {} (B1)
edge [bend right]
node [xshift=.2cm] {$c$}  
node [xshift=.4cm,yshift=-.7cm] {$Y_{(-2)}$} (B1);
\end{scope}
\end{tikzpicture}
\caption{The triangle level graph and a graph with the same
undegenerations}
\label{cap:triangle}
\end{figure}
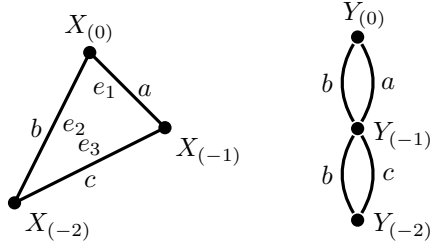
We discuss a simple case below where $\Tw[\Delta] \neq \sTw[\Delta]$ in
preparation for the examples in Section~\ref{sec:examples}. Finally,
we generalize for later use the lcm defined above. We define
$\ell_\Delta = \prod_{i=1}^L \ell_{\Delta,i}$ and where we use from 
now on the notation
\be \label{eq:defellGammai}
\ell_{\Delta,i} \= \lcm\Bigl(\kappa_e \colon e \in E(\Gamma)^{>-i}_{\leq -i}\Bigr)
\= \ell_{\delta_i(\Delta)}
\ee
as abbreviation of the one defined in the introduction, where $E(\Gamma)^{>-i}_{\leq -i}$
are the edges starting at level~$-i+1$ or above and ending at level~$-i$ 
or below. 
\par
\begin{example} \label{ex:triangle}
{\rm In the graph~$\Delta$ in Figure~\ref{cap:triangle} (left)
there are three edges $e_1$,$e_2$ and~$e_3$ with enhancements $a$, $b$ and~$c$.
The group $R_\Delta \cong \bZ^{2}$ acts on $\coprod_{i=1}^{k_1+k_2+k_3} \bZ/\kappa_i$
by mapping
\bes (1,0) \mapsto (\underbrace{1,\ldots\ldots,1}_{k_1+k_2},
\underbrace{0,\ldots,0}_{k_3}) \quad \text{and} \quad (1,0) \mapsto
(\underbrace{0,\ldots,0}_{k_1},\underbrace{1,\ldots\ldots,1}_{k_2+k_3})\,.
\ees
Consequently, there are $\gcd(a,b,c)$ orbits, i.e.\ that many equivalence
classes of prong-matchings near such a boundary point. The index of the
twist group~$\Tw[\Delta]$ in $R_\Delta$ is thus $abc/\gcd(a,b,c)$. On the other hand, as a
consequence of the discussion in the divisor case, the index of the simple
twist group~$\sTw[\Delta]$ in~$R_\Delta$ is $ab/\gcd(a,b)\cdot bc/\gcd(b,c)$.
Since $\Delta$ has no level graphs automorphisms, i.e., since $\Aut(\Delta)$ is
trivial, we conclude that in this case $D_\Delta$ is a quotient stack by
a group of order
\be \label{eq:eGamma}
e_\Delta \= \frac{\gcd(a,b,c)\,\lcm(a,b)\,\lcm(b,c)}{abc}\,.
\ee
}
\end{example}

\section{Clutching and projection to generalized strata} \label{sec:clutch}

In this section we define {\em generalized strata} where we allow disconnected
surfaces and residues constrained to a residue space~$\frakR$. This is similar
to a discussion in \cite{sauvagetstrata}. More precisely, we
show in Section~\ref{sec:comgenstra} how the construction of \cite{LMS}
carries over to this generalized context to give a compactification
$\proj\LMS[\bfmu][\bfg,\bfn][\frakR]$ of generalized strata.  
\par
The reason for dealing with generalized strata is to be able to work
with objects (like line bundles and Chow rings) on the individual levels
of a boundary stratum, and those might be disconnected and have with residue
conditions imposed by the GRC. We construct in
Section~\ref{sec:levelproj} for each boundary stratum~$D_\Gamma$  a finite
covering $D_\Gamma^s \to D_\Gamma$ that admits projections  $p_\Gamma^{[i]}:
D_\Gamma^s \to B_\Gamma^{[i]}$  where $B_\Gamma^{[i]}$ are the generalized strata
at level~$i$ of~$D_\Gamma$.

\subsection{The compactification of generalized strata}
\label{sec:comgenstra}

We start with the definition of strata in the generality that we need. First,
we allow for disconnected surfaces. Throughout
$\mu_i=(m_{i,1},\dots,m_{i,n_i})\in\bZ^{n_i}$ is the type of a differential,
i.e., we require that $\sum_{j=1}^{n_i} m_{i,j}=2g_i-2$ for some $g_i \in \bZ$
for $i=1,\ldots,k$. For a tuple $\bfg = (g_1,\ldots,g_k)$
of genera and a tuple $\bfn = (n_1,\ldots,n_k)$ together with
$\bfmu = (\mu_1,\ldots, \mu_k)$ we define the disconnected stratum
\be
\omoduli[\bfg,\bfn](\bfmu) \= \prod_{i=1}^k \omoduli[g_i,n_i](\mu_i)\,.
\ee
The projectivized stratum $\bP\omoduli[\bfg,\bfn](\bfmu)$ is the quotient
by the diagonal action of $\bC^*$, not the quotient by the action of $(\bC^*)^k$.
\par
Next, we prepare for global residue conditions.
Let $H_p\subseteq \cup_{i=1}^k \{(i,1),\cdots (i,n_i)\}$ be
the set of marked points such that $m_{i,j} < -1$. Now
consider vector spaces
$\frakR$ of the following special shape, modeled on the global residue condition
from \cite{BCGGM1}. Let $\lambda$ be a partition of~$H_p$ with parts denoted
by $\lambda^{(k)}$ and a subset~$\lambda_{\frakR}$ of the parts of $\lambda$
such that
\[\frakR\coloneqq \Bigl\{ r= (r_{i,j})_{(i,j) \in H_p} \in \bC^{H_p}
\quad \text{and} \sum_{(i,j) \in \lambda^{(k)}} r_{i,j}=0 \quad
\text{for all} \quad \lambda^{(k)} \in \lambda_{\frakR} \Bigr\}\,.\]
The subspace of surfaces with residues in~$\frakR$ will be denoted by
$\Romod[\bfmu][\bfg,\bfn][\frakR]$ and we will refer to
them as  {\em generalized strata}, too.
\par
To compute e.g.\ dimensions it is convenient to define the {\em residue subspace}
\be \label{eq:spaceR}
R \= \prod_{i=1}^{k} R_i \, \subseteq \,  \prod_{i=1}^{k} \bC^{l_i}
\ee
of differentials of the generalized stratum $\omoduli[\bfg,\bfn](\bfmu)$,
where $l_i$ is the number of negative entries in $\mu_i$. Here $R_i$
is the vector subspace cut out by the residue theorem in the $i$-th component
in the space generated by the vectors $r_{i,j}$ for each~$(i,j)$
with $m_{i,j} \leq -1$. When writing $\frakR \cap R$ we consider the intersection
inside~$ \prod_{i=1}^{k} \bC^{l_i}$.
\par
\begin{rem} \label{rem:gendim}
The dimension of the generalized stratum $\Romod[\bfmu][\bfg,\bfn][\frakR]$ is 
\be
\dim(\Romod[\bfmu][\bfg,\bfn][\frakR])=\left(\sum_{i=1}^k 2g_i+n_i-1\right)
-(l-\dim(\frakR \cap R)) \,,
\ee
where $l=\sum l_i$ is the total number of poles, i.e., marked points
with $m_{i,j}<0$.
\end{rem}
\par
We claim that the construction in \cite{LMS} can be carried out for
disconnected surfaces and for surfaces with an assigned residue subspace.
We only have to replace in the definition of the twisted differentials
$(X=(X_v)_{v \in V(G)}, \eta = (\eta_v)_{v \in V(G)})$ compatible with an enhanced
level graph~$\Gamma$ the global residue condition by the following condition.
We construct a new {\em auxiliary level graph $\widetilde{\Gamma}$}
\index[graph]{b006@$\widetilde{\Gamma}$!  auxiliary level graph}
by adding a new vertex $v_{\lambda^{(k)}}$ to $\Gamma$ at level~$\infty$
for each element
$\lambda^{(k)}\in \lambda_\frakR$ and converting a tuple $(i,j)\in \lambda^{(k)}$
into an edge from the marked point $(i,j)$ to  the vertex $v_{\lambda^{(k)}}$. 
\par
\begin{itemize}
\item {\bf $\frakR$-global residue condition ($\frakR$-GRC).} The tuple
of residues at the poles in~$H_p$ belongs to~$\frakR$ and for every
level $L<\infty$ of $\widetilde{\Gamma}$ and every connected component~$Y$
of the subgraph~$\widetilde{\Gamma}_{>L}$ one of
the following conditions holds.
\begin{itemize}
\item[i)] The component~$Y$ contains a marked point with a prescribed
pole that is  \emph{not} in $\lambda_\frakR$.
\item[ii)] The component~$Y$ contains a marked point with a prescribed
pole $(i,j) \in H_p$ and there is an $r \in \frakR$ with $r_{(i,j)} \neq 0$.
\item[iii)]  Let $q_1,\ldots,q_b$ denote the set of edges where~$Y$ 
intersects $\widetilde{\Gamma}_{=L}$. Then
$$ \sum_{j=1}^b\Res_{q_j^-}\eta_{v^-(q_j)}\=0\,,$$
where  $v^-(q_j)\in\widetilde{\Gamma}_{=L}$.
\end{itemize}
\end{itemize}
\par
This differs from the global residue condition in \cite{BCGGM1} only in
the subdivision of cases in~i) and~ii). As for the normal GRC (see
\cite{MUWrealize}),the $\frakR$-GRC also has an algorithmic graph theoretic
description, see \cite{CoMoZadiffstrata}.
\par
\begin{prop}
\label{prop:gencompactification}
There is a proper smooth Deligne-Mumford stack
$\proj\LMS[\bfmu][\bfg,\bfn][\frakR]$ containing
$\proj \Romod[\bfmu][\bfg,\bfn][\frakR]$ as an open dense substack
with the following properties:
\begin{itemize}
\item[(i)] The boundary $\proj \LMS[\bfmu][\bfg,\bfn][\frakR]
\smallsetminus \Romod[\bfmu][\bfg,\bfn][\frakR]$ is a normal crossing divisor.
\item[(ii)] A multi-scale differential defines a point in $\proj \LMS[\bfmu][\bfg,\bfn][\frakR]$ if and only if it is compatible with an enhanced level graph $\Gamma$ that satisfies the $\frakR$-GRC.
\item[(iii)] The codimension of a boundary stratum $D_\Gamma$ in
$\proj\LMS[\bfmu][\bfg,\bfn][\frakR]$ is equal to the number of horizontal
edges plus the number $L$ of levels below zero. %
\end{itemize}  
\end{prop}
\par
\begin{proof}
The residue spaces matters only for the existence of
modification differentials needed for gluing. In \cite[Lemma~4.6]{BCGGM1}
their existence for each component~$Y$ as in the global residue
condition was shown in case iii). This lemma also covers case i),
since we can impose a residue at that marked point to ensure that
the total sum equals zero. Since $\frakR$ is a vector space,
this can still be done in case ii).
\par
The smoothness and the normal crossing divisor properties follow from
the same reasoning as in \cite{LMS}. We leave the straightforward
verification of those many hidden claims of the proposition to
the reader.
\end{proof}
\par
Again, as in the usual situation, also the divisors~$D_\Gamma$ of generalized
strata may be disconnected or empty.
\par
\begin{example} {\rm To illustrate the $\frakR$-global residue condition
we consider the generalized stratum 
$\overline{B}  \= \bP\left( \omoduli[0,{3}](-2,-2,2)\times
\omoduli[0,{4}](-2,-2,1,1)\right)^\frakR $
where the special legs are given by the first two marked points of the first
component and the first two marked points of the second components, i.e.,
\[H_s \= \{(1,1),(2,1),(1,2),(2,2)\} \]
and the residue space is given by the the partition 
\[\lambda_{\frakR} \= \{\{(1,1),(1,2)\},\{(2,1),(2,2)\}\}\}.\]
This means that 
\[\frakR \=\{r_{(1,1)}+r_{(1,2)}\=0,\,\,\ r_{(2,1)}+r_{(2,2)}\=0\} \subset \bC^4,\]
and $R$ is the subspace defined by the residue theorem on each of the two
components, namely
\[R \=\{r_{(1,1)}+r_{(2,1)} \=0,\ r_{(1,2)}+r_{(2,2)} \=0\}. \]
By \autoref*{rem:gendim}, the above generalized stratum has dimension~$1$.
We want to show that the $\frakR$-GRC implies that there is only one $2$-level
boundary divisor in the compactification defined in
\autoref{prop:gencompactification}. This divisor
is given by the $2$-level graph with the 4-marked component on level~$0$
and the other component on level~$-1$. 
\par
The only two possible level graphs that could occur are the 2-level graph $\Gamma_1$ described above and the 2-level graph $\Gamma_2$ where the two components are inverted. Consider the auxiliary level graphs $\widetilde{\Gamma_1}$
and $\widetilde{\Gamma_2}$ needed in order to check the $\frakR$-GRC
given in Figure~\ref{cap:auxlevelgraphs}.
\begin{figure}[ht]
\begin{tikzpicture}
\draw[dashed] (0,0) rectangle (3,2);
\fill (1,1.5) coordinate (A) circle (2.5pt); 
\fill (1,3) coordinate (B) circle (2.5pt);
\fill (2.3,3) coordinate (C) circle (2.5pt);
\fill (2.3,.8) coordinate (D) circle (2.5pt);
\draw[very thick] (A)--(B)--(D)--(C)--(A);
\draw (A) -- +(240: .3) node [xshift=-3,yshift=-5] {$2$};
\draw (D) -- +(240: .3) node [xshift=-3,yshift=-5] {$1$};
\draw (D) -- +(300: .3) node [xshift=3,yshift=-5] {$1$};
\end{tikzpicture}
\quad \quad \quad 
\begin{tikzpicture}[]
\draw[dashed] (0,0) rectangle (3,2);
\fill (.9,.7) coordinate (A) circle (2.5pt); 
\fill (.9,3) coordinate (B) circle (2.5pt);
\fill (2.3,3) coordinate (C) circle (2.5pt);
\fill (2.3,1.5) coordinate (D) circle (2.5pt);
\draw[very thick] (A)--(B)--(D)--(C)--(A);
\draw (A) -- +(240: .3) node [xshift=-3,yshift=-5] {$2$};
\draw (D) -- +(240: .3) node [xshift=-3,yshift=-5] {$1$};
\draw (D) -- +(300: .3) node [xshift=3,yshift=-5] {$1$};
\end{tikzpicture}
\caption{Auxiliary level graphs $\widetilde{\Gamma_1}$ (left)
  and $\widetilde{\Gamma_2}$ (right) for the boundary strata
$\Gamma_1$ and $\Gamma_2$ (in the dashed boxes)}
\label{cap:auxlevelgraphs}
\end{figure}
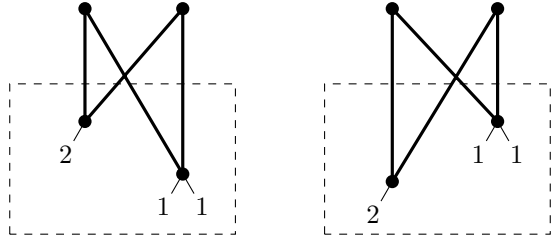
It is easy to see that condition (iii) of the  $\frakR$-GRC implies that the graph $\Gamma_1$ is illegal since both residues on the genus 0 component with the single zero of order 2 on the top level are zero, and this cannot happen.
}\end{example}

\subsection{Level projections and clutching} \label{sec:levelproj}

Consider a boundary stratum~$D_\Gamma$ given by an enhanced level graph~$\Gamma$.
It parameterizes multi-scale differentials, a differential on each level
together with a prong-matching. However, there are no well-defined
projection morphisms to the generalized strata on each level. 
E.g.\ $D_\Gamma$ might have generically trivial quotient stack structure and
the generalized strata on its levels might have everywhere trivial stack
structure, and yet special points of~$D_\Gamma$ have non-trivial quotient
structure. A graph~$\Gamma$ with two edges and two levels degenerating to
a triangle (Figure~\ref{cap:triangle}, left) provides an example. This is
due to the fact that the equivalence relation in the notion of multi-scale
differentials involves the twist group, which in the presence of edges
across multiple levels intertwines what happens at the levels. Our
goal here is to define a cover of~$D_\Gamma$ that has such projection maps.
\par
To define the generalized strata at the levels of~$D_\Gamma$
we let $(\bfg^{[i]}, \bfn^{[i]},\bfmu^{[i]})$ for  $i=0,\ldots,-L$ be the discrete
parameters genus, number of points and type at level~$i$ and let
$\frakR^{[i]}$ be the residue condition imposed at level~$i$. These residue
conditions are constructed via the $\frakR$-GRC  described before.
Our goal is:
\par
\begin{prop} \label{prop:prodcover}
There exists a stack $D_\Gamma^s$,  called the {\em simple boundary stratum of
type~$\Gamma$} that admits a finite map $c_\Gamma: D_\Gamma^s \to D_\Gamma$ and
finite forgetful maps
\be \label{eq:genproj}
 p_\Gamma^{[i]}: D_\Gamma^s \to  B_\Gamma^{[i]} :=
\proj\LMS[\bfmu^{[i]}][\bfg^{[i]},\bfn^{[i]}][\frakR^{[i]}]
\ee
for each $i=0,\ldots,-L$.
\end{prop}
\par
We denote by $p_\Gamma = \prod_{i=0}^{-L} p_\Gamma^{[i]}$ the product of all level
projections. In the case that~$D_\Gamma$ is a divisor we will also denote the
two projections by
\bes
p_\Gamma^\top: D_\Gamma^s \to B_\Gamma^\top
\= \proj\LMS[\bfmu^\top][\bfg^\top,\bfn^\top][\frakR^\top] 
\quad \text{and} \quad
p_\Gamma^\bot:  D_\Gamma^s \to B_\Gamma^\bot
\= \proj\LMS[\bfmu^\bot][\bfg^\bot,\bfn^\bot][\frakR^\bot]\,.
\ees
\par
With the help of the finite coverings $c_\Gamma$  and the inclusion of the
boundary strata $\fraki_\Gamma: D_\Gamma \to \ol{B}$,
we have now the {\em clutching maps} maps $\zeta_\Gamma = \fraki_\Gamma \circ
c_\Gamma$ at our disposal in order to define the generators of what we will
be defined as tautological ring, appearing in Theorem~\ref{thm:addgen}. 
\par
The strategy of proof of the proposition is to construct~$D_\Gamma^s$ as a
cover dominating the local covers of neighborhoods of more degenerate boundary
strata, following the strategy already used in \cite{MuEnum}. We do not
attempt to analyze whether $D_\Gamma^s$ is smooth, but the covering we
construct is branched at worst over the boundary divisors (hence locally
over the coordinate axis, since the boundary is normal crossing), so $D_\Gamma^s$
has at worst Cohen-Macaulay singularities (Proposition~2.2 in loc.~cit.), which allows
us to perform intersection theory as in loc.~cit. The objects of the
construction are summarized in the following diagram that we now explain.
\begin{center}
\begin{tikzcd}
& \widetilde{U}_\Delta \arrow{r}{\subset} \arrow[pos = 0.75]{dr}{q_\Delta}
& D_\Gamma^s
\arrow{ddrr} \arrow{ddll}{p_\Gamma} \arrow[bend left]{dddrrr}{c_\Gamma} \\
&& U_\Delta^s \arrow{dl}{p_\Gamma^\Delta}  \arrow{dr}
\arrow[bend right, pos = 0.6, swap]{ddrr}{c_\Gamma^\Delta} \\
B_\Gamma = \prod_i B_\Gamma^{[i]} & B_{\Gamma,\Delta} \arrow{l}{\supset}
&& U_\Delta^{\rm int} \arrow{r}{\subset} \arrow{dr} & D_\Gamma^{\rm gs}  \arrow{dr}\\
&&&& U_\Delta \arrow{r}{\subset} & D_\Gamma
	\end{tikzcd}
\end{center}
\par
For any level graph~$\Delta$ which is a degeneration of~$\Gamma$ we
let $U_\Delta \subset D_\Gamma$  be the open subset parametrizing multi-scale
differentials compatible with an undegeneration of~$\Delta$. In
particular $U_\Gamma \subset D_\Gamma$  is the complement of all boundary
strata where~$\Gamma$ degenerates further. In symbols (with notation as in
Section~\ref{sec:twdmsd})
\be \label{eq:defUgamma}
U_\Delta \= \Bigl(\coprod_{\Gamma {\rightsquigarrow} \Pi {\rightsquigarrow} {\Delta}}
\mathfrak{W}_{\operatorname{pm}}(\Pi)/T_{\Delta}\Bigr) /\bC^*\,,
\ee
with the complex structure and modular interpretation as in \cite[Section~12 and
Section~7]{LMS}, forgetting the Teichm\"uller marking there and intersecting
with~$D_\Gamma$. %
\par
Next, we define~$U_\Delta^s$. It will be the moduli stack of {\em simple multi-scale
differential compatible with an undegeneration of~$\Delta$} that we define now.
(This notion with additional Teichm\"uller marking is essentially also given
in \cite[Section~8]{LMS}.) Recall that we defined multi-scale differentials compatible
with~$\Delta$ as the quotient by the action of the level rotation torus
$T_\Delta = \bC^{L(\Delta)}/\Tw[\Delta]$. Simple multi-scale differentials compatible
(precisely) with~$\Delta$ are defined as the quotient by the simple level rotation torus
$T_\Delta^s =\bC^{L(\Delta)}/\sTw[\Delta]$.
(To define the notion in families, one should use the simple rescaling
ensemble instead of the rescaling ensemble, see \cite[Section~7]{LMS}.)
Multi-scale differentials compatible with an undegeneration~$\Pi$ of~$\Delta$,
i.e. for level graphs with $\Gamma {\rightsquigarrow} \Pi {\rightsquigarrow}
{\Delta}$, are also included in the stack we are about to define. We have to
specify the correct notion of equivalence $\mathfrak{W}_{\operatorname{pm}}(\Pi)$
so that these objects fit together in families. Similar to the definition of
the simple Dehn space in \cite{LMS}, this works if we declare two differentials in
$\mathfrak{W}_{\operatorname{pm}}(\Pi)$ to be declared equivalent if they
differ by the action of the image of $T_\Delta^s \to T_\Delta \to T_\Pi$.
To be able to define projection maps, we moreover mark all the half-edges
of~$\Gamma$ (i.e.\ the edges of~$\Gamma$, keeping the labels at the marked
points) in our notion of simple multi-scale differentials. The moduli stack of
simple multi-scale differential compatible with an undegeneration of~$\Delta$ is
denoted by $U^s_\Delta$ and comes with a covering map $c_\Gamma^\Delta: U^s_\Delta
\to U_\Delta$. In symbols
\be \label{eq:defDps}
U^s_\Delta \= \Bigl(\coprod_{\Gamma {\rightsquigarrow} \Pi {\rightsquigarrow} {\Delta}}
\mathfrak{W}^*_{\operatorname{pm}}(\Pi)/T_{\Delta}^s\Bigr) /\bC^*\,,
\ee
with topology and complex structure again as in \cite{LMS}, and where the
star should remind of the extra edge marking.
\par
Since the edges of~$\Gamma$ are labelled for points in  $U^s_\Delta$
and since the equivalence relation is defined by $T_\Delta^s$, hence
level by level, we may decompose the simple multi-scale differentials
parameterized by~$U^s_\Delta$ according to the levels of~$\Gamma$. In
this we we obtain maps $p_\Gamma^{\Delta,[i]}: U_\Delta^s \to B_\Gamma^{[i]}$
such that the product map $p_\Gamma^\Delta = \prod_i p_\Gamma^{\Delta,[i]}$
is a finite cover of an open subset $B_{\Gamma,\Delta}$  of the product
of level strata $B_\Gamma = \prod_i B_\Gamma^{[i]}$.
\par
The last step is to define a covering dominating all the $c_\Gamma^\Delta$.
For technical reasons we first define the 'generically simple'
intermediate space $D_\Gamma^{\rm gs}$, that removes the stack structure
over the open subset $U_\Gamma$ (if there is). The space $D_\Gamma^\ps$
contains $U_\Gamma^s$ as open dense subset. The covering
$D_\Gamma^{\rm gs} \to D_\Gamma$ is defined by marking all the edges of~$\Gamma$
and by using the covering of level rotation tori $T_\Gamma^s \to T_\Gamma$
over~$U_\Gamma$ as well as over the boundary strata~$D_\Gamma \setminus U_\Gamma$.
This is to say that for a degeneration~$\Gamma {\rightsquigarrow} \Delta$
the points in the intermediate space $U_\Delta^{\rm int}$ are multi-scale
differentials up to the equivalence relation given by the hybrid torus
$T_{\Delta} \times_{T_\Gamma} T_\Gamma^s$. Finally, we take $D_\Gamma^s$
to be the normalization of $D_\Gamma^{\rm gs}$ in a Galois field extension
of the function field of~$D_\Gamma^{\rm gs}$ that contains all the
extensions defined by $U_\Delta^s \to D_\Gamma^{\rm gs}$. (If $D_\Gamma$
happens to be reducible, we perform the construction on each connected
component. Actually, the $U_\Delta^s$ still have a stack structure due to
automorphism of the underlying stable curves. The details how to construct
the covering with this caveat are in \cite[Section~2b]{MuEnum}.)
This space comes with a forgetful map $c_\Gamma: D_\Gamma^s \to D_\Gamma$ that factors
as  $c_\Gamma  = c_\Gamma^\Delta \circ q_\Delta: \widetilde{U}_\Delta \to U_\Delta$ 
over the preimages of~$U_\Delta$. We may now define
$p_\Gamma^{[i]} = p_\Gamma^{\Delta,[i] } \circ q_\Delta$, since the 
$\widetilde{U}_\Delta$ for all degenerations~$\Gamma {\rightsquigarrow}
\Delta$ cover~$D_\Gamma^s$. This completes the {\em proof of
Proposition~\ref{prop:prodcover}.}
\par

\subsection{Push-pull comparison} \label{sec:clutmor}

Let $\Gamma \in \LG_L(\ol{B})$ be a level graph. Several recursive 
computations in the sequel are performed on the level strata~$B_\Gamma^{[i]}$
and we want to transfer the result via $p^{[i]}$-pullback and
$c_\Gamma$-pushforward to~$D_\Gamma$. This section provides the basic
relations in this push-pull procedure. The degree of~$c_\Gamma$ seems
difficult to compute. In applications we only the following relative statement.
\par
\begin{lemma} \label{le:degreeratio}
The ratios of the degrees of the projections in Proposition~\ref{prop:prodcover} is
\be
\frac{\deg(p_\Gamma)}{\deg(c_\Gamma)}
\= \frac{K_\Gamma}{|\Aut(\Gamma)|\,\ell_\Gamma}\,.
\ee
\end{lemma}
\par
\begin{proof} The degrees can be computed at the generic point, where
both maps factor through $q_\Gamma$. The degree of~$p_\Gamma^\Gamma$ is the number of
equivalence classes of prong-matchings, which is $K_\Gamma / [R_\Gamma : \Tw[\Gamma]]$.
The degree of $c_\Gamma^\Gamma$ is the index $[\Tw[\Gamma]: \sTw[\Gamma]] \cdot
|\Aut(\Gamma)|$. The claimed equality
\be
\frac{\deg(p_\Gamma)}{\deg(c_\Gamma)} \= 
\frac{\deg(p_\Gamma^\Gamma)}{\deg(c_\Gamma^\Gamma)}
\=  \frac{1}{|\Aut(\Gamma)} \frac{K_\Gamma} {[R_\Gamma: \sTw[\Gamma]]} 
\= \frac{K_\Gamma}{|\Aut(\Gamma)|\,\ell_\Gamma}
\ee
follows from the definition of the simple twist group.
\end{proof}
\par
Next we compare codimension 1 boundary classes on the strata~$D_\Gamma\in \LG_L(B)$
and on their level strata~$B_\Gamma^{[i]}$ in order to pull back tautological relations.
We use the symbol $[D_\Gamma]$ to denote the fundamental class of the
substack of $\overline{B}$ parameterizing multi-scale differentials compatible
with a degeneration of~$\Gamma$. Let $i \in \bZ_{\leq 0}$.
\par
Consider a graph $\Delta \in \LG_1(B_\Gamma^{[i]})$ 
defining a divisor in $B_\Gamma^{[i]}$. We aim to compute its pullback 
to $D_\Gamma^s$ and the push forward to~$D_\Gamma$ and to~$\ol{B}$. 
Recall that in $D_\Gamma^s$  the edges of~$\Gamma$ have been labeled once
and for all (we write $\Gamma^\dagger$ for this labeled graph) and that the
level strata~$B_\Gamma^{[i]}$ inherit these labels.
Consequently, there is unique graph $\widehat{\Delta}^\dagger$ which is a
degeneration of~$\Gamma^\dagger$ and such that extracting 
the levels~$i$ and~$i-1$ of~$\widehat{\Delta}^\dagger$ equals~$\Delta$.
The resulting unlabeled graph will simply be denoted by~$\widehat{\Delta}$.
(Recall from \autoref{rem:conventionindex} that
$\delta_{(-i+1)}^\complement(\widehat{\Delta})=\Gamma$.)
On the other hand, the procedure of gluing in and forgetting labels
is not injective. For a fixed labeled graph~$\Gamma^\dagger$
we denote by $J(\Gamma^\dagger,\widehat{\Delta})$ the set of 
$\Delta \in \LG_1(B_\Gamma^{[i]})$ such that $\widehat{\Delta}$ is the result of
that procedure. Obviously the graphs in $J(\Gamma^\dagger,\widehat{\Delta})$
differ only by the labeling of their half-edges.
\par
\begin{lemma} \label{le:autcancel}
The cardinality of $J(\Gamma^\dagger,\widehat{\Delta})$
is determined by
\bes
|J(\Gamma^\dagger,\widehat{\Delta})|\cdot |\Aut(\wh{\Delta})| \=
|\Aut(\Delta)|\cdot |\Aut(\Gamma)|\,.
\ees
\end{lemma}
\par
\begin{proof}
Consider the map $\varphi: \Aut(\wh{\Delta}) \to \Aut(\Gamma)$
induced by the undegeneration $\delta_{(-i+1)}^\complement$ of the $(-i+1)$-th
level passage of $\wh{\Delta}$. For an element in
the kernel, the graph~$\Gamma$ is fixed, so we may as well label
it. Thanks to these labels, extraction of the levels~$i$ and $i-1$
now defines a graph $\Delta \in \LG_1(B_\Gamma^{[i]})$ and the
restriction map ${\rm Ker}(\varphi) \to \Aut(\Delta)$ is an isomorphism.
To determine the cokernel of~$\varphi$ we use the labels given
by~$\Gamma^\dagger$ and a degeneration $\widehat{\Delta}^\dagger$ labeled
except for the edges interior to that pair of levels. 
After restriction to the levels~$i$ and $i-1$ the elements
in the image of~$\varphi$ act trivially. The resulting bijection
of ${\rm Coker}(\varphi)$ and $J(\Gamma^\dagger,\widehat{\Delta})$
proves the result.
\end{proof}
\par
We now determine the multiplicities of the push-pull procedure.
Recall from~\eqref{eq:defellGammai} the definition of $\ell_{\Gamma,j}$,
for $j\in \bZ_{\geq 1}$.
\par
\begin{prop} \label{prop:divpstar}
For a fixed $\Delta \in \LG_1(B_\Gamma^{[i]})$, the divisor classes of
$D_{\widehat{\Delta}}$ and the clutching of $D_\Delta$ are related by
\be \label{eq:Dcomparison}
\frac{|\Aut(\widehat{\Delta})|}
{|\Aut(\Delta)| |\Aut(\Gamma)|}
\cdot
c_\Gamma^* [D_{\widehat{\Delta}}]
\= \frac{\ell_{{\Delta}}}{\ell_{\widehat{\Delta},-i+1}}
 \cdot p_\Gamma^{[i],*} [D_\Delta]\,.
\ee
in $\CH^1(D_\Gamma^s)$ and consequently by
\be
\frac{|\Aut(\widehat{\Delta})|}{|\Aut(\Gamma)|} \cdot 
\ell_{\widehat{\Delta},-i+1}   \cdot [D_{\widehat{\Delta}}]
\=  \frac{|\Aut(\Delta)|}{\deg(c_\Gamma)} \cdot  \ell_\Delta \cdot
\,c_{\Gamma,*} \bigl(p_\Gamma^{[i],*} [D_\Delta]\bigr)
\ee
in $\CH^1(D_\Gamma)$.
\end{prop}
\par
\begin{proof} If suffices to show the first equation, the second
follows by taking~$c_{\Gamma,*}$. Since the two sides are supported on the
same set, it suffices to verify the multiplicities. Since near the divisors
under consideration both sides are pullback via $q_{\wh{\Delta}}$ this can be done
by computing the ramification orders of the finite maps $c_\Gamma^{\wh{\Delta}}$
and $p_\Gamma^{\wh{\Delta}}$ over the divisor $D_{\widehat{\Delta}}$ and
over $\tilde D_\Delta = D_\Delta \times \prod_{j \neq i} B_\Gamma^{[j]}$ respectively.
\par
We start with $c_\Gamma^{\wh{\Delta}}$. There, passing to the equivalence relation by
the torus $T_\Gamma^s$  gives a covering of degree $[\Tw[\Gamma]:\sTw[\Gamma]]$, both
at a generic point and over $D_{\widehat{\Delta}}$. Adding the markings
on the edges of~$\Gamma$ gives $|\Aut(\Gamma)|$ additional choices at
a generic point. Over $D_{\widehat{\Delta}}$ only the automorphism in
image of the map~$\varphi$ (as in the proof of Lemma~\ref{le:autcancel})
can be rigidified by adding the marking. This image has cardinality
$|\Aut(\widehat{\Delta})|/|\Aut(\Gamma)|$ and thus the ramification
order is the reciprocal of the factor on the left hand side
of~\eqref{eq:Dcomparison}.
\par
Next we consider the map~$p_\Gamma^{\wh\Delta}$. Since in $\prod_{j} B_\Gamma^{[j]}$ and
thus also on $\tilde D_\Delta$ the half-edges that form the edges of~$\Gamma$ are
labeled, graph automorphism do not contribute to branching. However, after
adding the prong matching for~$\Gamma$, the orbits of the $-i+1$-st component
of the integer subgroup $\bZ^{L+1} \subset \bC^{L+1}$ of the level rotation
torus change. In $\tilde D_\Delta$ (and in $D_\Delta$) the orbit has
size $\ell_\Delta$, while in $D_\Gamma^s$ the orbit  has size
$\ell_{\widehat{\Delta},-i+1}$ since the prongs of edges of $\widehat{\Delta}$
are acted on, too. Since this component of the level rotation torus
is not present at a generic point and since all other components
have the same effect at a generic point and over $\tilde D_\Delta$,
we conclude that the ramification order is the reciprocal of the factor on
the right hand side of~\eqref{eq:Dcomparison}.
\end{proof}
\par
Next we compare various versions of the $\xi$-class on boundary
strata. A first definition is by a local description. Consider
a level~$i\in \{0,\dots,-L\}$ of a boundary stratum~$D_\Gamma$ and recall that it is
a moduli space of multiscale differentials compatible with a degeneration
of~$\Gamma$. We define the line bundle $\cO_\Gamma^{[i]}(-1)$ on $D_\Gamma$ as
follows. On open sets where $\Gamma$ does not degenerate further,
it is generated by the $i$-th component $\eta_{(i)}$ of the multi-scale
differential. If $\Gamma$ degenerates to~$\Gamma_1$ the level~$i$
splits up into an interval $i$ to $i-k$ of levels, then the local generator
of $\cO_\Gamma^{[i]}(-1)$ is the multi-scale
components $\eta_{(i)}$ for the top of these levels. We let $\xi_\Gamma^{[i]}
= c_1(\cO_\Gamma^{[i]}(-1))$ and write $\xi^\top_\Gamma$ for the top
level contribution. 
\par
\begin{rem} {\rm 
  Since stable differentials on a boundary stratum are zero on all levels
  apart from the top one, we have $\xi^\top_\Gamma = \xi|_{D_\Gamma}$.
}\end{rem}
\par
\begin{prop} \label{prop:xiasppull}
The first Chern classes of the tautological bundles on the levels
of a boundary divisor are related by
\be
c_\Gamma^* \, \xi^{[i]}_\Gamma  \=  p_\Gamma^{[i],*} \xi_{B_\Gamma^{[i]}}
\qquad \text{in} \quad \CH^1(D_\Gamma^s)\,.
\ee
\end{prop}
\par
\begin{proof} Comparing local generators, we obtain a  collection
of isomorphisms
\bes c_\Gamma^{\Delta,*} \cO_{B_\Gamma^{[i]}(-1)}
\,\cong \, (p_\Gamma^{\Delta,[i]})^*\,\cO_\Gamma^{[i]}(-1)
\ees
compatible with restrictions
to undegenerations. The $q_\Delta$-pullbacks of this collection 
of maps gives the isomorphsm on $D_\Gamma^s$, and then we take the first
Chern class.
\end{proof}
\par
We will continue the study of the tautological ring in
Sections~\ref{sec:nb} and~\ref{sec:tautring}, using local descriptions near
the boundary introduced along with Section~\ref{sec:eulerseq}.

\section{The structure of the boundary} \label{sec:strucbd}

In this section we show that the non-horizontal boundary divisors $D_\Gamma$
are smooth.  More generally we show that if a collection of non-horizontal
divisors intersects, then there is a unique order on this collection such that
$i$-th divisors appear as the $i$-th $2$-level undegeneration of an
intersection point.
\par
In the sequel it will be convenient to assume that the $2$-level
graphs have been numbered once and for all, say as $\twolev =
\{\Gamma_1, \ldots, \Gamma_M\}$. Note that that the intersection of
two divisors, say $D_{\Gamma_1}$ and $D_{\Gamma_2}$, consists a priori
of the sublocus $D_{12}$ of unions of $D_{\Lambda}$, for $\Lambda\in \LG_2(B)$
with $\delta_1(\Lambda)=\Gamma_1$ and $\delta_2(\Lambda)=\Gamma_2,$ and
the sublocus $D_{21}$, which is the union  of $D_{\Lambda}$ for
$\Lambda\in \LG_2(B)$ with $\delta_1(\Lambda)=\Gamma_2$ and
$\delta_2(\Lambda)=\Gamma_1$. The notation generalizes to any
number of levels. We define the suborbifold 
\be
\label{eq:notationindex}
D_{i_1,\ldots,i_L} \, \subseteq \, \bigcap_{j=1}^L D_{i_j}  
\ee
consisting of all $D_{\Lambda}$, with  $\Lambda\in \LG_L(B)$  such that
$\delta_j(\Lambda) = \Gamma_{i_j}$ for all $j=1,\ldots,L$ and we refer to
this by the ordered set $[i_1,\dotsc,i_L]$, called the \emph{profile}
of the boundary stratum. We denote by $\mathscr{P} = \mathscr{P}(B)$
the set profiles of~$B$ and by $\mathscr{P}_L$ those of length~$L$.
The language of profiles is used mainly in this section and then
again in Theorem~\ref{thm:cpoly}, while elsewhere we usually work with
set of level graphs. The sage package {\tt diffstrata} makes fully
use of the notion of profiles and the following proposition.
\par
\begin{prop} 	\label{prop:ordering}
If $\cap_{j=1}^L D_{\Gamma_{i_j}}$ is not empty,  there is a unique
ordering $\sigma\in\Sym_L$ on the set $I=\{i_1,\dots,i_L\}$ of indices
such that
	\[D_{\sigma(I)}\=\bigcap_{j=1}^L D_{\Gamma_{i_j}}\,.\]
Moreover if $i_k=i_{k'}$ for a pair of indices $k\not =k'$,
then $D_{i_1,\dots,i_L}=\emptyset$.
\end{prop}
\par
\begin{rem} {\rm In general the intersection of boundary divisors
$D_{\sigma(I)}$ is not irreducible, i.e., it consists of boundary strata
associated to different enhanced level graphs, see for example the $3$-level
graphs in Figure~\ref{cap:triangle}.
}\end{rem}
\par
The preceding proposition also gives a useful relation. Suppose two divisors
$D_{\Gamma_1}$ and $D_{\Gamma_2}$ meet in a boundary stratum $D_\Delta$. Two
situations may occur. Either $\delta_1(\Delta) = \Gamma_1$ and
$\delta_2(\Delta) = \Gamma_2$ or vice versa. In the first situation,
$\Delta$ arises from degenerating the lower level of~$\Gamma_1$. We phrase
this by saying that $\Gamma_2$ {\em goes under} $\Gamma_1$
and write $\Gamma_2 \prec \Gamma_1$. A priori, this notion might depend
on the enhanced level graph~$\Delta$. But the preceding proposition implies
that it does in fact not depend on~$\Delta$.
\par
The proof of Proposition~\ref{prop:ordering} uses dimension estimates and the
following lemma. We define 
\[d_\Lambda^{[p]} \=\dim(B_\Lambda^{[p]}) \text{ for all } \Lambda\in\LG_L(B)\,, \]
where $B_\Lambda^{[p]}$ is the projectivized substratum at level~$p\in \{0,\dots,-L\}$ of~$D_\Lambda$ defined in Proposition~\ref{prop:prodcover}. Note that the
sum $\sum_{p=0}^{-L} (d_\Lambda^{[p]} + 1) = N = 1+\dim(\overline{B})$ is the
unprojectivized dimension of the stratum. 
\par
\begin{lemma} \label{lem:merging}
The dimensions of the levels of a boundary stratum~$D_\Lambda$ and the
boundary divisor $D_{\delta_k(\Lambda)}$ given by its $k$-th undegeneration
are related by
  \bes
  d_{\delta_k(\Lambda)}^{[0]} \= k-1+\sum_{p=0}^{k-1} d_{\Lambda}^{[-p]},\quad
  d_{\delta_k(\Lambda)}^{[-1]} \= L-1-k+\sum_{p=k}^{L-1} d_{\Lambda}^{[-p]}\,.
  \ees
\end{lemma}
\begin{proof}
The follows directly from the description of undegeneration, see \cite{LMS}.
\end{proof}
\par
\begin{proof}[Proof of Proposition~\ref{prop:ordering}]
Assume that, after reordering,  $\cap_{i=1}^{L} D_{i}$ is not empty, and
that $D_{\Lambda}$ is a component of $D_{1,\dots,L}$. Assume furthermore that
there is a permutation $\sigma\in S_j$, such that also $D_{\sigma(1),\dots,\sigma(L)}$
is non-empty, containing a component $D_{\Lambda^\sigma}$. Now, by definition
\[\delta_1(\Lambda^\sigma) \=\delta_{\sigma(1)}(\Lambda)\=D_{\Gamma_{\sigma(1)}}\,.\]
By \autoref{lem:merging} we can then write the dimension of the top component of $D_{\Gamma_1}$ and $D_{\Gamma_{\sigma(1)}}$ in two different ways, namely
\bas
d_{\Gamma_{1}}^{[0]} &\=d^{[0]}_{\Lambda}\=\sigma^{-1}(1)-1 +\sum_{p=0}^{\sigma^{-1}(1)-1} d_{\Lambda^\sigma}^{[-p]} \\
d_{\Gamma_{\sigma(1)}}^{[0]} &\=d^{[0]}_{\Lambda^\sigma}
\=\sigma(1)-1 +\sum_{p=0}^{\sigma(1)-1} d_{\Lambda}^{[-p]}.
\eas
By substituting the first expression into the second one we obtain
\[ d^{[0]}_{\Lambda^\sigma}=\sigma(1)-1 +\sigma^{-1}(1)-1 +\sum_{p=0}^{\sigma^{-1}(1)-1} d_{\Lambda^\sigma}^{[-p]}+\sum_{p=1}^{\sigma(1)-1} d_{\Lambda}^{[-p]}\]
which simplifies to
\[ 0=\sigma(1)-1 +\sigma^{-1}(1)-1 +\sum_{p=1}^{\sigma^{-1}(1)-1} d_{\Lambda^\sigma}^{[-p]}+\sum_{p=1}^{\sigma(1)-1} d_{\Lambda}^{[-p]}.\]
This implies that $\sigma(1)=1$.
By induction we get that $\sigma=\id$.
\par
In order to prove the second statement assume by contradiction that  the orbifold $D_{i_1,\dots,i_{L}}$ is non-empty, with $i_1=i_k$ for  $1<k\leq L$. Let $D_{\Lambda}$ be a component of  $D_{i_1,\dots,i_{L}}$. Then by \autoref{lem:merging} we get
\bas
d_{\delta_{1}(\Lambda)}^{[0]} \=d_{\Lambda}^{[0]}=k-1+\sum_{p=0}^{k-1} d_{\Lambda}^{[-p]}.
\eas
This implies that $k=1$, which is already a contradiction.
\end{proof}
\par

\section{Euler sequence for strata of abelian differentials}
\label{sec:eulerseq}

The characteristic classes of the tangent bundle to projective space~$\bP(V)$
of a vector space~$V$ are conveniently computed using the Euler
sequence
\be \label{eq:EulerStd}
0\longrightarrow \Omega^1_{\bP(V)} \longrightarrow
\cO_{\bP(V)}(-1)^{\oplus \dim(V)}
\overset{\ev}{\longrightarrow}   \cO_{\bP(V)}\longrightarrow 0\,.
\ee
Our main computational tool uses the affine structure of strata
to provide a similar Euler sequence on the compactified strata
$\overline{B} = \proj \LMS$.
\par
\begin{theorem} \label{thm:EulerDE}
There is a vector bundle~$\cK$ on $\overline{B}$ that fits into an exact
sequence
\be \label{eq:EulerExt}
0\longrightarrow \cK \longrightarrow (\Hrelbar)^\vee\otimes \cO_{\overline{B}}(-1)
\overset{\ev}{\longrightarrow}   \cO_{\overline{B}}\longrightarrow 0\,,
\ee
where $\Hrelbar$ is the Deligne extension of the relative cohomology,
such that the restriction of $\cK$ to the interior~$B$ is the cotangent bundle $\Omega^1_B$.
\end{theorem}
\par
An explicit description of local generators of~$\cK$ is part of the proof
in this section. We will have set up the tools to describe~$\cK$ intrinsically
in Theorem~\ref{thm:coker}.
\par
We will define the evaluation map $\ev$ in the course of the construction.
The construction happens first over the open part and then the finite
covering charts that exhibit $\bP\LMS$ locally as quotient stack.
\par

\subsection{Over the open stratum} \label{sec:openeulerseq}

Recall that moduli space of Abelian differential  have an affine structure
given by period coordinates. Concretely, for
a pointed flat surface  $(X,\omega,\bz)$ we denote by $Z = \{z_1,\ldots,
z_{r+s}\}$ the zeros and by $P =  \{z_{r+s+1},\ldots, z_n\}$ the poles among
the marked points, thus including marked ordinary points in~$Z$.
By \cite{HubbardMasur} or \cite{Veech} (see also \cite{kdiff}) integration of
the one-form along relative periods is a local biholomorphism and thus provides
local charts of $\omoduli[g,n](\mu)$ in the vector space 
\bes V \= V_{(X,\omega,\bz)} \coloneqq \H^1(X\setminus P, Z;\bC)\,.
\ees
The changes of charts are linear, in fact with $\bZ$-coefficients. This makes
the projectivization~$B$ into a $(\PGL_{N}, \bP^{N-1})$-manifold.
\par
We denote by $\Hrel$ the local system on~$B$ with fiber the relative
cohomology $V = \H^1(X\setminus P, Z;\bC)$ and recall that~$N = \dim(V) =
\dim(B) +1$.  Recall that the fiber of $\cO_B(-1)$ at the
point $(X,\omega,\bz)$ is the vector space
generated by~$\omega$. We thus obtain the evaluation map 
\[\ev\colon(\Hrel)^\vee\otimes \cO_B(-1)\to \cO_B,\quad  \gamma\otimes \omega\mapsto
\int_{\gamma} \omega\]
by integrating the one-form.
\par
\begin{prop}
There is a short exact sequence of vector bundles on $B$
\[0\longrightarrow \Omega^1_B \longrightarrow (\Hrel)^\vee\otimes \cO_B(-1) \overset{\ev}{\longrightarrow}   \cO_B\longrightarrow 0\]
that locally on a chart $\bP V$ is given by the standard Euler sequence.
\end{prop}
\par
\begin{proof}  Let $\pi:\tilde{B}\to B$ be the universal cover of $B$.
Consider the developing map $\dev\colon \tilde{B}\longrightarrow \bP(V)$,
which is a $\pi_1(B)$-equivariant local isomorphism. We use the sequence
on the standard charts of $\bP(V)$ and we claim that its $\dev$-pullback
descends to an exact sequence~$B$.
\par
To justify this, consider paths $\{\alpha_i\}_{i=1}^N$ that form a local
frame of $(\Hrel)^\vee$. Let $\{a_i\}_{i=1}^N$ be the corresponding local
coordinates and $\{\dd a_i\}$ the local frame of $\Omega^1_{\bP(V)}$. On the
open subset  $U_k=\{a_k\not =0\}\subseteq \bP(V)$
the monomorphism of the Euler sequence~\eqref{eq:EulerStd} is given by
\be \label{eq:eulerinbasis}
\dd a_i\mapsto \Bigl(\alpha_i-\frac{a_i}{a_k}\alpha_k\Bigr)\otimes \omega,\qquad
i=1,\dots,\hat{k},\dots,N\,,
\ee
where~$\omega$ is the representative of the line bundle with $\int_{\alpha_k}
\omega =1$. The pull-back sequence gives rise to an isomorphism of short exact
sequences
\begin{center}
	\begin{tikzcd}
0\arrow{r} &  \dev^*\left(\Omega^1_{\bP(V)}\right)  \arrow{r} \arrow{d}{\cong} &
\dev^*(V ^\vee\otimes \cO_{\bP(V)}(-1)) \arrow{r}{\ev}\arrow{d}{\cong} & \dev^*(\cO_{\bP(V)}) \arrow{r}\arrow{d}{\cong}& 0 \\
	0\arrow{r} & \pi^*(\Omega^1_B) \arrow{r} & \pi^*(\Hrel)^\vee\otimes \pi^*(\cO_B(-1)) \arrow{r}{\ev} & \pi^*(\cO_B)\arrow{r}& 0 
	\end{tikzcd}
\end{center}
Each vector bundle appearing is provided with a canonical $\pi_1(B)$ action and the vertical maps are isomorphisms of $\pi_1(B)$-vector bundles. The first vertical map is an isomorphism since the developing map is a local isomorphism and $ \pi^*(\Omega^i_B)\cong \Omega^i_{\tilde{B}}$ for every $i$.
Since the evaluation map is $\pi_1(B)$-equivariant, so is the kernel. Hence the short exact sequence passes to the quotient by the action of $\pi_1(B)$ and yields the claim.
\end{proof}

\subsection{Coordinates near the boundary}  \label{sec:bdperiod}

Coordinates near the boundary of the moduli space $\LMS$ are {\em perturbed
  period coordinates} (\cite[Section~11]{LMS} or \cite[Section 3]{CoMoZa})
that we now illustrate in typical
cases that exhibit all the relevant features. The reader is encouraged to
read this subsection in parallel with the subsequent one, where the
Euler sequence is extended step by step to these boundary strata.
\par
\medskip
\paragraph{\bf Case 1: only horizontal nodes} Suppose that the level graph
$\Gamma$ consists of $k \geq 1$ horizontal edges only, all of them must
necessarily be non-separating. At a smooth point
near $D_\Gamma$ the relative homology can be grouped into
\begin{itemize}
\item the vanishing cycles $\alpha_i$ for $i=1,\ldots,k$ around the
nodes,
\item loops $\beta_i$ symplectically dual to $\alpha_i$, and
\item paths $\gamma_1, \ldots, \gamma_{N-2k}$ completing the above
to a basis of relative homology.
\end{itemize}
Coordinates in a chart of~$\LMS$ near $D_\Gamma$ are given by
the periods $c_i = \int_{\gamma_i} \omega$, by $a_i = \int_{\alpha_i} \omega$ and
by the exponentiated period ratio $q_i = \exp(2\pi i b_i/a_i)$ where
$b_i = \int_{\beta_i} \omega$. To provide charts of the projectivization
$\overline{B}$ we fix~$a_1$ to be identically one.
\par
\medskip
\paragraph{\bf Case 2: two levels, only vertical nodes}
For concreteness, we suppose that in the $2$-level graph $\Gamma \in \LG_1(B)$
there is only one vertex on each level and for concreteness, say,  
with three edges $e_1,e_2,e_3$ joining the two vertices. Suppose moreover
that there is no marked zero on lower level. (If there is such a marked
point on each level, the loops $\beta_i$ below have to be replaced by
relative periods across the level, leading to similar constructions.) At
a point close to~$D_\Gamma$ the relative homology can be grouped into
\begin{itemize}
\item loops $\beta_1$ through $e_1$ and $e_3$ and $\beta_2$ through $e_2$
and $e_3$,
\item loops $\alpha_1$ and $\alpha_2$, the vanishing cycles corresponding
to $e_1$ and $e_2$,
\item paths $\gamma_1^{[0]}, \ldots, \gamma_{d_0}^{[0]}$ forming a basis of
the relative homology on top level, 
\item loops $\gamma_1^{[-1]}, \ldots, \gamma_{d_1}^{[-1]}$ forming a basis of
the  homology on bottom level,
\end{itemize}
for some $d_0,d_1 \in \bZ$, see also Figure~\ref{cap:cyclesCase2}.
\par
On the other hand, the surfaces on the boundary stratum~$D_\Gamma$
have a basis of relative homology that can be grouped into
\begin{itemize}
\item relative periods $\widetilde{\beta}_i$ joining the marked points at
the upper ends of the edge $e_i$ to the upper end of the edge $e_3$ for $i=1,2$,
\item loops $\widetilde{\alpha}_i$ around the poles at lower ends of $e_i$
for $i=1,2$,
\item paths $\widetilde{\gamma}_1^{[0]}, \ldots, \widetilde{\gamma}_{d_0}^{[0]}$
forming a basis of the relative homology on top level, 
\item loops $\widetilde{\gamma}_1^{[-1]}, \ldots, \widetilde{\gamma}_{d_1}^{[-1]}$
forming a basis of the  homology on bottom level.
\end{itemize}
\begin{figure}[htb]
	\centering
\begin{tikzpicture} [scale=.50] 
\draw (0,0) ellipse (.75cm and 1.5cm);
\draw (2.5,0) ellipse (.75cm and 1.5cm);
\draw[blue] (2.5,0) ellipse (1.05cm and 1.8cm)
	node [above,xshift=-.5cm,yshift=.75cm]{$\beta_2$};
\begin{scope}[red,xshift=-1.55cm, yshift=0cm]
	\draw[dotted, thick] (0,0) arc (180:0:.4cm and .15cm);
	\draw (0,0) arc (180:360:.4cm and .15cm) node[midway,below,xshift=-.07cm]{$\alpha_1$};
\end{scope}
\begin{scope}[red,xshift=.75cm, yshift=0cm]
	\draw[dotted, thick] (0,0) arc (180:0:.4cm and .15cm);
	\draw (0,0) arc (180:360:.4cm and .15cm) node[midway,below]{$\alpha_2$};
\end{scope}
\draw[blue,rounded corners=1.1cm] (-1,-2.5) rectangle (3.8,2.5)
	node [midway,yshift=1.6cm]{$\beta_1$};
\draw plot [smooth ,tension=.65] coordinates 
	{(-1.55,0) (-1.7,.5) (-2.95,1.5) (-2.5,3.8) (1.2,4.25)
	 (5,3.8) (5.5,1.5) (4.5,.5) (4.3,0)}; %
\begin{scope}[yscale=-1,xscale=1]
	\draw plot [smooth ,tension=.65] coordinates 
	   {(-1.55,0) (-1.7,.5) (-2.95,1.5) (-2.5,3.8) (1.2,4.25)
		(5,3.8) (5.5,1.5) (4.5,.5) (4.3,0)}; %
\end{scope}
\begin{scope}[xshift=-1.2cm, yshift=3.4cm]
	\draw (.4,-.32) arc (180:0:.5cm and .45cm);
	\draw[] (0,0) arc (190:350:.9cm and .5cm);
\end{scope}
\begin{scope}[xshift=2.1cm, yshift=3.6cm]
	\draw (.4,-.32) arc (180:0:.5cm and .45cm);
	\draw[] (0,0) arc (190:320:.9cm and .5cm);
	\begin{scope}[rotate around={65:(.95,.98)}]
		\draw (0,0) arc (180:0:.46cm and .2cm);
		\draw[dotted, thick] (0,0) arc (190:360:.46cm and .2cm) 
		node [midway, xshift=.3cm, yshift=-.2cm]{$\gamma_j^{[0]}$};
	\end{scope}
\end{scope}
\begin{scope}[xshift=.7cm, yshift=-2.9cm]
	\draw (.4,-.32) arc (180:0:.5cm and .45cm);
	\draw[] (0,0) arc (190:350:.9cm and .5cm);
	\begin{scope}[rotate around={-90:(0.02,-.37)}]
		\draw[dotted, thick] (0,0) arc (180:0:.5cm and .2cm);
		\draw (0,0) arc (190:360:.5cm and .2cm) 
		node [xshift=.6cm,yshift=.26cm]{$\gamma_j^{[-1]}$};
	\end{scope}
\end{scope}
\begin{scope}[xshift=11cm]
\draw (.2,.12) ellipse (1.2cm and 1.cm);
\draw (2.6,.12) ellipse (1.2cm and 1.cm);
\draw[blue]  (1.41,0) arc (190:110:1.2cm and 1.3cm) arc (110:10:1.2cm and 1.4cm)
	node [above,xshift=-1.2cm,yshift=.45cm]{$\widetilde\beta_2$};
\begin{scope}[red,xshift=-1.33cm, yshift=-.7cm]
	\draw[dotted, thick] (0,0) arc (180:0:.4cm and .12cm);
	\draw (0,0) arc (180:360:.4cm and .12cm) 
		node[midway,below,xshift=-.1cm]{$\widetilde\alpha_1$};
\end{scope}
\begin{scope}[red,xshift=1cm, yshift=-.65cm]
	\draw[dotted, thick,thick] (0,0) arc (180:0:.4cm and .12cm);
	\draw (0,0) arc (180:360:.4cm and .12cm) 
		node[midway,below]{$\widetilde\alpha_2$};
\end{scope}
\draw[blue] (-1,0) arc (180:0:2.4cm and 2.5cm)
	node [midway,above]{$\widetilde\beta_1$};
\draw plot [smooth ,tension=.65] coordinates  
   {(-1,0) (-1.35,.7) (-2.95,1.5) (-2.5,3.8) (1.2,4.25)
	(5,3.8) (5.5,1.5) (4.05,.7) (3.8,0)}; %
\begin{scope}[yscale=-1,xscale=1]
	\draw plot [smooth ,tension=.65] coordinates 
	   {(-1,0) (-1.35,.7) (-2.95,1.5) (-2.5,3.8) (1.2,4.25)
		(5,3.8) (5.5,1.5) (4.05,.7) (3.8,0)}; %
\end{scope}
\begin{scope}[xshift=-1.2cm, yshift=3.4cm]
	\draw (.4,-.32) arc (180:0:.5cm and .45cm);
	\draw[] (0,0) arc (190:350:.9cm and .5cm);
\end{scope}
\begin{scope}[xshift=2.1cm, yshift=3.6cm]
	\draw (.4,-.32) arc (180:0:.5cm and .45cm);
	\draw[] (0,0) arc (190:320:.9cm and .5cm);
	\begin{scope}[rotate around={65:(.95,.98)}]
		\draw (0,0) arc (180:0:.46cm and .2cm);
		\draw[dotted, thick] (0,0) arc (190:360:.46cm and .2cm) 
			node [midway, xshift=.3cm, yshift=-.2cm]{$\gamma_j^{[0]}$};
	\end{scope}
\end{scope}
\begin{scope}[xshift=.7cm, yshift=-2.9cm]
	\draw (.4,-.32) arc (180:0:.5cm and .45cm);
	\draw[] (0,0) arc (190:350:.9cm and .5cm);
	\begin{scope}[rotate around={-90:(0.02,-.37)}]
		\draw[dotted, thick] (0,0) arc (180:0:.5cm and .2cm);
		\draw (0,0) arc (190:360:.5cm and .2cm) 
			node [xshift=.6cm,yshift=.26cm]{$\gamma_j^{[-1]}$};
	\end{scope}
\end{scope}
\end{scope}
\end{tikzpicture}
\caption{Cycles in Case~2, near the boundary stratum and at the
boundary stratum}\label{cap:cyclesCase2}
\end{figure}
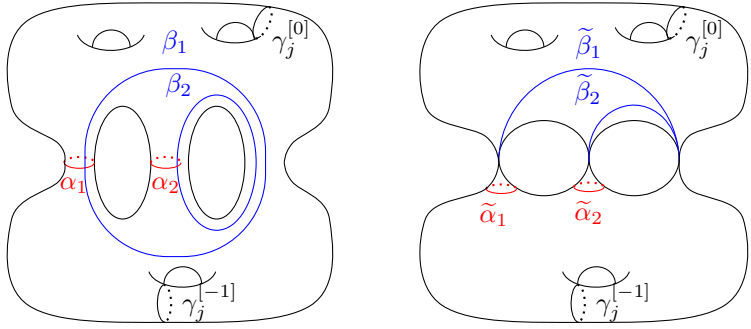

From this description it is apparent that $d_i$ is related to the
projectivized and unprojectivized dimensions of the level strata previously
introduced by $d_i = N_\Gamma^{[i]} -2 = d_\Gamma^{[i]}-1$. 
\par
The main statement about perturbed period coordinates  \cite[Section~11]{LMS}
is that on the one hand, coordinates near the boundary are given by the periods
on the boundary surfaces and on the other hand, periods with and without
tilde are nearly the same after appropriate rescaling. To make this statement
concrete, let $\kappa_i$ be the enhancements corresponding the edges~$e_i$
and let $\ell = \lcm(\kappa_1,\kappa_2,\kappa_3)$. Near our current
boundary divisor $D_\Gamma$ the universal family of curves has a 
(universal) family of differentials $\omega$ and~$\ell$ is chosen so
that rescaling $\eta_{(-1)} = t^{-\ell} \omega_{(-1)}$ is holomorphic and generically non-zero
for a coordinate with $D_\Gamma = \{t = 0\}$ locally (\cite[Section~12]{LMS}).
At each point $p \in D_\Gamma$ we find a non-zero $\eta$-period on lower
level, say the period along $\widetilde{\gamma}_1^{[-1]}$, and choose~$t$
and thus $\eta$ so that $\int_{\widetilde{\gamma}_1^{[-1]}} \eta = 1$.
\par
A chart of~$\LMS$ near~$p$ is then nearly the product of a neighborhood
of the irreducible components $(X_0,\omega)$ and $(X_1,\eta)$ of the
fiber over~$p$ in their respective strata of meromorphic differentials.
Here, 'nearly' refers to the fact that, because of prong-matchings,  it  is a $\ell$-fold cover fully
ramified over $t=0$, and moreover, because of enhanced level graph automorphisms, it is a quotient
stack by the subgroup~$G$ of $S_3$ that exchanges edges with the same
enhancement.
\par
Coordinates on this chart are then given by~$t$ and the periods 
\bas
\widetilde{b_i} &\=\int_{\widetilde{\beta}_i} \eta_{(0)} \quad (i=1,2),
& r_i &\= \int_{\widetilde{\alpha}_i} \eta_{(-1)} \quad (i=1,2), \\
\widetilde{c_i}^{[0]} &\=\int_{\widetilde{\gamma}_i^{[0]}} \eta_{(0)} \quad (i=1,\ldots,d_0), \quad
& \widetilde{c_i}^{[-1]} &\= \int_{\widetilde{\gamma}_i^{[-1]}}  \eta_{(-1)} \quad (i=2,\ldots,d_1)\,.
\eas
To provide charts of the projectivization $\overline{B}$ we simply fix one
of the periods on top level, say $\widetilde{c}_1^{[0]}$, to be identically
one. (If $d_0=0$ we take $\widetilde{b}_1 \equiv 1$ instead.)
\par
In each sector near the boundary, the perturbed period coordinates
are related to the $\omega$-periods by
\ba \label{eq:perid2lev}
b_i &:= \int_{{\beta}_i} \omega \, \sim \, \widetilde{b_i}
\qquad 
& a_i &:= \int_{\alpha_i} \omega  \=  t^\ell r_i \
\\ 
c_i^{[0]} &:= \int_{\gamma_i^{[0]}} \omega  \, \sim \, \widetilde{c_i}^{[0]}, \qquad
& c_i^{[-1]}&:= \int_{\gamma_i^{[-1]}} \omega \= t^\ell \widetilde{c_i}^{[-1]}.
\ea
where $\sim$ indicates that the difference is $O(t^\ell)$.
The difference stems (for $c_i^{[0]}$) from the fact that the $\omega$
in the universal family is not just the deformation of the twisted
differential  $(\eta_{(0)},\eta_{(-1)})$ %
in the fiber over~$p$ in its product moduli space, but blurred
by some modification differentials. For the $b_i$ there is an
additional error in the same order of magnitude due to a choice
of a nearby base point in the  plumbing construction.
\par
\medskip
\paragraph{\bf Case 3: two levels, additional horizontal nodes} This
is a mixture of the previous two cases. To see the effects, we assume
that we are in the situation of Case 2, with one horizontal node
and thus additionally a pair of cycles $\alpha^{[j]}$ and $\beta^{[j]}$
with $j=0$ or $j=-1$ depending on the level where the horizontal node
is attached.We may then uniformly write the periods
$a^{[j]} = \int_{\alpha^{[j]}} \eta_{(j)}$ and
$b^{[j]} = \int_{\beta^{[j]}} \eta_{(j)}$. The additional coordinates are~$a^{[j]}$
and the exponentiated period ratio $q^{[j]} = \exp(2\pi i b^{[j]}/a^{[j]})$.
\par
\medskip
\paragraph{\bf Case 4: three levels, three nodes}
This is the generalization of the triangle case (Figure~\ref{cap:triangle}
left),
with edges replaced possibly by multiple strands, say $k_i$ strands
for the edge $e_i$, including the case $k_i=0$ for missing edge (as
the long edge in Figure~\ref{cap:triangle} right). Let $\ell_1$
be the lcm of the enhancements on the edges starting at level~$0$
and $\ell_2$ the lcm of the edges ending at level~$-2$, as
defined in~\eqref{eq:defellGammai}.
\par
A point $p \in D_\Gamma$ on the corresponding divisor is given by
meromorphic differential forms $(X_{(0)},\eta_{(0)})$, $(X_{(-1)},\eta_{(-1)})$,
$(X_{(-2)},\eta_{(-2)})$ together with prong-matchings. We denote
by $\widetilde{\gamma}_i^{[j]}$ for $j=0,-1,-2$ 
and $i=1,\ldots, N_j$ paths of the relative homology of the surfaces.
(There are no global residue conditions in this example.)
We may suppose that $\int_{\widetilde{\gamma}_1^{[j]}} \eta_{(j)} = 1$ for $j=0,-1,-2$
to fix the scale of the $\eta_{(j)}$ on lower level and for $j=0$
to fix an open subset of the projectivization.
\par
A chart of~$\LMS$ near~$p$ is then nearly the product of a neighborhood
of the irreducible components $(X_{(j)},\eta_{(j)})$ where $j=0,-1,-2$ of the fiber
over~$p$ in their respective strata of meromorphic differentials. Slightly
abusing notation we call the universal differentials over these neighborhoods
also~$\eta_{(j)}$. A coordinate system for the neighborhood of $p \in \overline{B}$
is given by functions $t_1$ and  $t_2$ that correspond to rescalings of the two
levels together with the functions $\widetilde{c}_i^{[j]} =
\int_{\widetilde{\gamma}_i^{[j]}} \eta_{(j)}$ for $j=0,-1,-2$ and
$i=2,\ldots,N_j$. In particular $N_0 + N_{-1} + N_{-2} = N$.
\par
To give the relation of these coordinates to nearby periods, note that
the universal differential~$\omega$ over $\LMS$, has by construction
the property that the periods of $\omega$ on bottom level agree with
those of $t_1^{\ell_1} t_2^{\ell_2} \eta_{(-2)}$, the periods on level~$-1$ of
$\omega$ differ from those of $t_1^{\ell_1} \eta_{(-1)}$ by functions that decay
like $t_1^{\ell_1} t_2^{\ell_2}$ and periods on top level of $\omega$
differ from those of $\eta_{(0)}$ by functions that decay like $t_1^{\ell_1}$.
Here, as we have illustrated in Case~2, the loops around the nodes
corresponding to the $k_1+k_2+k_3$ edges can be treated as residues and thus as
periods on the level at the lower end of the edge, while the loops through
those edges (denote previously by $\beta_i$) can be treated as relative periods
on the highest level that the loop touches. 
\par

\subsection{The Euler sequence on the Deligne extension} \label{sec:exteulerseq}

Recall that the Deligne extension of a local system on~$B$ is a canonical
extension to a vector bundle on~$\ol{B}$ admitting an extension of the
Gauss-Manin connection to a connection with regular singular points
(\cite{DelEqDiff}). 
In this section we want to extend the Euler sequence across the boundary to
construct~\eqref{eq:EulerExt}. For this purpose we exhibit local generators
of the Deligne extension $\Hrelbar$ of $\Hrel$, extend the map $\ev$
and determine its kernel in each of the cases as we discussed perturbed period
coordinates in Section~\ref{sec:bdperiod}, adopting notation from there.
\par
\medskip 
\paragraph{\bf Case 1: only horizontal nodes} A basis of
$(\Hrelbar)^\vee$ consists of the cycles $\alpha_1,\ldots,\alpha_k$
and $\gamma_1,\ldots,\gamma_{N-2k}$  that extend across~$D_\Gamma$, together
with the linear combinations
\bes
\widehat{\beta}_i \= \beta_i - \frac1{2\pi i}\log(q_i)\alpha_i 
\ees
designed to be monodromy invariant. Since the family one-forms $\omega$
extends across $D_\Gamma$ to a family of stable differentials, the definition 
\bes
\ev(\widehat{\beta}_i \otimes \omega) \= \int_{\beta_i} \omega
- \frac1{2\pi i}\log(q_i)
\int_{\alpha_i} \omega \= b_i - \frac1{2\pi i}\log(q_i) a_i \= 0
\ees
extends the definition of $\ev$ in the interior and gives a well-defined
holomorphic function. To check the surjectivity of $\ev$ we can use any of the
periods that extend across~$D_\Gamma$. We claim that the kernel of $\ev$ is
on the chart~$U$ with $a_1 \equiv 1$
\be
\cK \= \bigl\langle dq_1/q_1, da_2, dq_2/q_2, \ldots, da_k, dq_k/q_k, dc_1,\ldots, dc_{N-2k} \bigr\rangle \,
\ee
as $\cO_U$-module. In fact, using the definition~\eqref{eq:eulerinbasis}
in the interior one checks that
\be
dq_i/q_i\=d\log(q_i)\=d\left(2\pi i \frac{b_i}{a_i}\right)\, \mapsto \,
\frac{2\pi i}{a_i} \Bigl(\beta_i - \frac{b_i}{a_i} \alpha_i\Bigr) \otimes \omega
\ee
is mapped to a local generator of $(\Hrelbar)^\vee\otimes \cO_{\overline{B}}(-1)$
since the functions~$a_i$ do not vanish near such a boundary point. Moreover
$dq_i/q_i$ is mapped to the kernel of~$\ev$ by the preceding calculation. For the other
elements these claims follow as in the interior.
\par
\medskip 
\paragraph{\bf Case 2: two levels, only vertical nodes} We first work in the
special case near a boundary divisor $D_\Gamma$ where $\Gamma$ has three edges
as in the case discussed in Section~\ref{sec:bdperiod}. A basis of
$(\Hrelbar)^\vee$ consists of the cycles $\alpha_1,\alpha_2,
\gamma_1^{[0]}, \ldots, \gamma_{d_0}^{[0]}, \gamma_1^{[-1]}, \ldots,
\gamma_{d_1}^{[-1]}$ that extend across~$D_\Gamma$, together with
the linear combinations 
\be \label{eq:defhatbeta}
\widehat{\beta}_i \= \beta_i - \log(t)(\alpha_i + (\alpha_1+\alpha_2))
\qquad(i=1,2)
\ee
that are monodromy invariant since turning once around the divisor acts
by simultaneous Dehn-twists around the core curves of the three plumbing
cylinders. Sending cycles that extend across $D_\Gamma$ to
their $\omega$-integrals and letting
\bes
\ev(\widehat{\beta}_i \otimes \omega) \= \int_{\beta_i} \omega
- \log(t)\Bigl(\int_{\alpha_i} \omega + \int_{\alpha_1+\alpha_1} \omega \Bigr)
\qquad(i=1,2)
\ees
extends the definition of $\ev$ in the interior and is well-defined since
the function $\log(t)\Bigl(2\int_{\alpha_i} \omega + \int_{\alpha_1+\alpha_1} \omega \Bigr)
= O(t^\ell \log(t))$ is bounded near $D_\Gamma$ and $\int_{\beta_i} \omega \to
\int_{\widetilde{\beta}_i} \omega$ is bounded as well.
\par
Obviously on the chart with $c_1^{[0]} = 1$ the kernel of $\ev$ is 
\ba \label{eq:cKCase2Var1}
{\rm Ker}(\ev) &\= \Bigl\langle
\gamma_i^{[0]} - {c_i^{[0]}} \gamma_1^{[0]}\,\, (i=2,\ldots,d_0);\,\,\,
\alpha_i - {a_i} \gamma_1^{[0]}\,\, (i=1,2);\,\, \\ 
& \qquad \quad \gamma_i^{[-1]} - {c_i^{[-1]}} \gamma_1^{[0]}\,\, (i=1,\ldots,d_1);\,
\widehat{\beta_i} - \widehat{b}_i \gamma_1^{[0]}\,\, (i=1,2)
\Bigr\rangle 
\ea
where $\widehat{b}_i$ is the integral of $\widehat{\beta_i}$. We claim that
via the identification of periods in~\eqref{eq:perid2lev} this kernel is
precisely the image of 
\bes
\cK \= \bigl\langle d\widetilde{c}_2^{[0]}, \ldots, d\widetilde{c}_{d_0}^{[0]},\,
d\widetilde{b}_1, d\widetilde{b}_2, \, t^\ell dt/t,  \, t^\ell
d\widetilde{c}_2^{[-1]}, \ldots, t^\ell d\widetilde{c}_{d_1}^{[-1]}, \,
t^\ell dr_1, \, t^\ell dr_2,
\bigr \rangle \,.
\ees
under the map~\eqref{eq:eulerinbasis}. First, since we used the
coordinate $\widetilde{c}_1^{[1]}$ to fix the scaling on the bottom level,
the differential form  $\ell t^\ell dt/t = d c_1^{[-1]}$ is mapped
to $\gamma_1^{[-1]} - {c_1^{[-1]}} \gamma_1^{[0]}$. Then from~\eqref{eq:perid2lev},
we see that $t^\ell d\widetilde{c}_i^{[-1]}$ is mapped to a linear combination
of $\gamma_i^{[-1]} - {c_i^{[-1]}} \gamma_1^{[0]}$ and the previous generator
for any $i \geq 2$. Similarly $t^\ell dr_i$ maps to $\alpha_i - {a_i}
\gamma_1^{[0]}$ and a linear combination of the previous generators.
In the second step we consider the generators that correspond to top level.
The form $d\widetilde{c}_i^{[0]}$ does not quite map to $\gamma_i^{[0]}
- {c_i^{[0]}} \gamma_1^{[0]}$ because of the presence of modification
differentials, but the difference is a linear combination of the differential
of some~$\eta$-periods that we have shown already in the first
step to belong to ${\rm Ker}(\ev)$. Similarly, the image of
$d\widetilde{b_i}$ and $\widehat{\beta_i} - \widehat{b}_i \gamma_1^{[0]}$
is differentials of periods on lower level (from~\eqref{eq:perid2lev},
to compare with $db_i$ and from~\eqref{eq:defhatbeta}).
\par
We now rename and regroup the generators of~$\cK$ in a form that generalizes
to other level graphs. Since the $\beta$-periods  become relative periods
and since the $\alpha$-periods for the edges joining the
levels are simply residues appearing on lower level, we may name
the set of all periods on top level by $\widetilde{c}_i^{[0]}$ for $1 \leq i
\leq N_0$ and those on bottom level by $\widetilde{c}_i^{[-1]}$ for $1 \leq i
\leq N_1$. Then the above argument gives that
\be \label{eq:cK2lev}
\cK \= \bigl\langle d\widetilde{c}_2^{[0]}, \ldots, d\widetilde{c}_{N_0}^{[0]},\, 
t^\ell dt/t,  \, t^\ell d\widetilde{c}_2^{[-1]}, \ldots, t^\ell
d\widetilde{c}_{N_1}^{[-1]}  \bigr\rangle \,.
\ee
\par
\medskip
\paragraph{\bf Case 3: two levels, additional horizontal nodes} We mix the
conclusion of the two previous cases. If the horizontal node is at the
top level, then
\be
\cK \= \langle d\widetilde{c}_2^{[0]}, \ldots, d\widetilde{c}_{N_0}^{[0]},\, da,\, 
dq/q, \, t^\ell dt/t,  \, t^\ell d\widetilde{c}_2^{[1]}, \ldots,
t^\ell d\widetilde{c}_{N_1}^{[1]} \rangle \,,
 \ee
while in the case of a horizontal node and the bottom level 
\be
\cK \= \langle d\widetilde{c}_2^{[0]}, \ldots, d\widetilde{c}_{N_0}^{[0]},\, 
t^\ell dt/t,  \, t^\ell d\widetilde{c}_2^{[-1]}, \ldots, t^\ell d\widetilde{c}_{N_1}^{[-1]},
\, t^\ell da,\,  t^\ell dq/q \rangle \,.
\ee
\par
\medskip
\paragraph{\bf Case 4: three levels, three nodes}
We can adopt here from Case~2 the argument the $\alpha_j$-periods corresponding
to the graphs become residues and the monodromy-invariant modifications
$\widehat{\beta_j}$ of the dual $\beta_j$-periods have $\ev$-images that
tend to the $\beta_j$-integrals. We claim that thus ${\rm Ker}(\ev)$ is the image of
\ba
\cK \= \langle d\widetilde{c}_2^{[0]}, \ldots, d\widetilde{c}_{N_0}^{[0]},\,\,
& t_1^{\ell_1} dt_1/t_1,\,
\, t_1^{\ell_1} d\widetilde{c}_2^{[-1]}, \ldots, \,t_1^{\ell_1}
d\widetilde{c}_{N_1}^{[-1]}  \\
&  t_1^{\ell_1} t_2^{\ell_2} dt_2/t_2,
\, t_1^{\ell_1}t_2^{\ell_2} d\widetilde{c}_2^{[-2]}, \ldots, t_1^{\ell_1}t_2^{\ell_2}
d\widetilde{c}_{N_2}^{[-2]} \rangle \,.
\ea
under the map~\eqref{eq:eulerinbasis}. We justify this, starting at bottom level.
The differential form
\bes d(\widetilde{c}_1^{[-2]}) \= d(t_1^{\ell_1}t_2^{\ell_2})
\=  \ell_2 t_1^{\ell_1}t_2^{\ell_2} dt_2/t_2 +
\ell_1 t_1^{\ell_1}t_2^{\ell_2} dt_1/t_1  \quad \in \cK\,,
\ees
since it is mapped to $\gamma_1^{[-2]} - {c_1^{[-2]}} \gamma_1^{[0]}$, which
in analogy with~\eqref{eq:cKCase2Var1} belongs to the natural basis
of~${\rm Ker}(\ev)$.  Next, the form
$dt_1^{\ell_1}t_2^{\ell_2} d\widetilde{c}_{i}^{[-2]}$ map to a linear
combination of the elements $\gamma_i^{[-2]} - {c_i^{[-2]}} \gamma_1^{[0]}$
in the natural basis of~${\rm Ker}(\ev)$ and the previous generator. 
\par
We next proceed to the middle level. There, the form $\ell_1t^{\ell_1} dt_1/t_1$
is not quite equal to $d(c_1^{[-1]})$ because of the presence of modification
differentials. It thus does not quite map to the basis element
$\gamma_i^{[-1]} - {c_i^{[-1]}} \gamma_1^{[0]}$ of ${\rm Ker}(\ev)$.
But the difference is a
combination of elements that we have already shown to belong to~$\cK$.
As a combination of this form and $d(\widetilde{c}_1^{[-2]})$ we now have
$\ell_2 t_1^{\ell_1}t_2^{\ell_2} dt_2/t_2 \in \cK$. Considering
the remaining form $dc_i^{[-1]}$ from periods on middle level, and then
all the form $d{c}_i^{[0]}$ for $i \geq 2$ on top level identifies the
remaining elements listed in~$\cK$ with elements of~${\rm Ker}(\ev)$,
up to the effect of modification differentials, which produce differentials
of periods already shown to belong to~$\cK$.
\par
The notation 
	\begin{equation}\label{eq:prodtnotation}
	\prodt  \= \prod_{i=1}^j t_i^{\ell_i}, \quad j\in \bN.
	\end{equation}
will be convenient here and in the sequel.
\par
\begin{proof}[Proof of Theorem~\ref{thm:EulerDE}] 
Continuing the argument as in the preceding cases, we see that near a
point $p \in D_\Gamma$ the elements
\begin{itemize}
\item $	\prodt dt_j/t_j,$ for every level $-j$,
\item the $	\prodt$-multiples of differential forms
associated to periods on level~$-j$
\item  $	\prodt  dq_k^{[-j]}/q_k^{[-j]}$ for every horizontal
node with parameter $q_k$ on level $-j$
\end{itemize}
freely generate~$\cK$.
\end{proof}

\section{The normal bundle to boundary strata} \label{sec:nb}

In this section we provide formulas to compute the first Chern class of
the normal bundle~$\cN_\Gamma = \cN_{D_\Gamma}$ to a boundary divisor~$D_\Gamma$.
We will encounter here and in the sequel frequently the top level correction
line bundle
\be \label{eq:defLtop}
\cL_{\Gamma}^\top \=  \cO_{D_\Gamma} \Bigl(\sum_{\wh{\Delta} \in \LG_2(\ol{B})
\atop \delta_2(\wh{\Delta}) = \Gamma} \ell_{\wh{\Delta},1}D_{\wh{\Delta}} \Bigr)
\ee
on~$D_\Gamma$ that records all the degenerations of the top level
of~$\Gamma$.
\par
\begin{theorem}	\label{thm:nb}
  Suppose that $D_\Gamma$ is a divisor in $\ol{B}$ corresponding to
a graph $\Gamma \in \LG_1(\ol{B}$). Then
\be \label{eq:nb}
c_1(\cN_\Gamma) \= \frac{1}{\ell_\Gamma} \bigl(-\xi_\Gamma^\top
- c_1(\cL_\Gamma^\top) + \xi_\Gamma^\bot \bigr)\quad \text{in} \quad
\CH^1(D_{\Gamma})\,.
\ee
\end{theorem}
\par
In case the graph~$\Gamma$ contains an edge~$e$ (which is automatic if
the ambient stratum parameterizes connected curves, but often not satisfied
in the generalization to higher codimension strata below) there is an
alternative expression for the Chern class of the normal bundle, that gives
a comparison to the situation in the moduli space of curves.
Let $e^\pm$ be the half edges that form the edge~$e$.
\par
\begin{prop} \label{prop:c1normal}
The first Chern class of the normal bundle $\cN_\Gamma$ of
a boundary divisor $D_{\Gamma}$ is
\be
\c_1(\cN_\Gamma)\=
-\frac{\kappa_e}{\ell_{\Gamma}}\biggl(\psi_{e^+} +
\psi_{e^-}\biggr)-\frac{1}{\ell_{\Gamma}}
\sum_{\wh{\Delta}\in \LG_{2,e}^{\Gamma}(B)} \!\!\!\!\!\!\ell_{\wh{\Delta},a_{\wh{\Delta},\Gamma}}
[D_{\wh{\Delta}}].
\ee
as an element of $\CH^1(D_{\Gamma})$, where $\LG_{2,e}^\Gamma(B)$ is the set
of $3$-level graphs in $\LG_2^\Gamma(B)$ where the edge~$e$ goes from
level zero to level~$-2$ and where $a_{\wh{\Delta},\Gamma} \in \{1,2\}$ is the
index such that the $a_{\wh{\Delta},\Gamma}$-th undegeneration of $\wh{\Delta}$ is not equal to~$\Gamma$.
\end{prop}
\par
We say that~$\LG_{2,e}^{\Gamma}(B)$ are the 3-levels graphs where the edge~$e$
{\em becomes long}.
\index[graph]{b016@$\LG_{2,e}^\Gamma(B)$!  $3$-level graphs in $\LG_2^\Gamma(B)$
  where~$e$ is long}
We give direct proofs of both expressions for the normal bundle. The
equivalence of the statements follows from an application of the relation
in Proposition~\ref{prop:Adrienrel} below.
\par
\begin{proof}[Proof of Theorem~\ref{thm:nb}]
We consider over the boundary stratum~$D_\Gamma$ the line bundles $\cL_1 =
\cO_\Gamma^{[0]}(-1) \otimes \cL_\Gamma^\top$ and $\cL_2 = \cO_\Gamma^{[-1]}(-1)$
where the tautological bundles on the levels have been introduced in
Section~\ref{sec:clutmor}. Roughly the content of the theorem is that
the ratio of local sections of these line bundles is the function
$t_1^{\ell_\Gamma}$, which is also the $\ell_\Gamma$-th power of a transversal
coordinate. For the precise statement we compare the cocycles defining
the line bundles $\cL_1^{-1} \otimes \cL_2$ and $\cN_{\Gamma}^{\ell_\Gamma}$.
\par
We start by considering the open subset of~$D_\Gamma$ where~$\Gamma$
does not degenerate further. A local section of $\cL_1^{-1} \otimes \cL_2$
is the ratio of two relative
differential forms, thus a function on the base, that we may compute
as $u = \int_{\alpha_1} \eta_{(-1)} / \int_{\alpha_0} \eta_{(0)}$ for some paths 
$\alpha_1$ at level~$-1$  and $\alpha_0$ at level~$0$. Here $\alpha_1$
can be taken as (usual) relative cycle, and for $\alpha_0$ we might have
to use a path starting and ending at points {\em near} the upper ends
of connecting nodes, as in the definition of perturbed period coordinates
in \cite[Section~11]{LMS}. We consider a nearby coordinate patch where
now the ratio is $\widetilde{u} = \int_{\widetilde{\alpha_1}} \eta_{(-1)}  /
\int_{\widetilde{\alpha_0}} \eta_{(0)} $ for some new cycles related to the
original ones by a base change $\widetilde{\alpha_1} = \alpha_1 + \gamma_1$
and $\widetilde{\alpha_0} = \alpha_0 + \gamma_0$
in the homology of the upper and lower level subsurfaces respectively.
One computes that
\bes
\widetilde{u} = u \cdot \frac{1+y}{1+x}\,, \quad \text{where} \quad
x\= \int_{\gamma_0} \eta_{(0)}\,/\,\int_{\alpha_0} \eta_{(0)} \quad \text{and} \quad
y \= \int_{\gamma_1} \eta_{(-1)} \,/\,\int_{\alpha_1} \eta_{(-1)}
\ees
In particular these $x,y$ are local functions on the upper and lower level
strata.
\par
On the other hand, by construction (of the perturbed period coordinates)
the $\ell_\Gamma$-th power of a transversal coordinate is given by 
\bes
t_1^{\ell_\Gamma} \= s \=
\int_{\alpha_1} \eta_{(-1)} \,/\, \int_{\alpha_0} (\eta_{(0)} + \xi_{(0)})\,,
\ees
where $\xi_{(0)}$ is the modification differential at level~$0$ constructed in
\cite[Section~11]{LMS} and where the~$\alpha_i$ are as above. Again
a nearby coordinate patch is given by $\widetilde{s} =
\int_{\widetilde{\alpha_1}} \eta_{(-1)} /, \int_{\widetilde{\alpha_0}} (\eta_{(0)} + \xi_{(0)})$
with cycles as above. The main point now is that $\xi_0$ is divisible by~$s$
by construction, and so its contribution vanishes in after $s$-derivation
and setting $s=0$, so
\be
\frac{\partial \widetilde{s}}{\partial s}\Bigr|_{s = 0} \= \frac{1+y}{1+x}
\= \frac{\widetilde{u}}{u}\,,
\ee
showing that the cocycles from  $\cL_1^{-1} \otimes \cL_2$ and
$\cN_{\Gamma}^{\ell_\Gamma}$ agree on the subset under consideration.
\par
If the bottom level degenerates or in case of horizontal degenerations
of~$\Gamma$, the above claims remain valid without modification, if
we take~$\alpha_1$ to be a period that does not go to lower level.
If the top level degenerates into two levels (without loss of generality,
higher codimension degenerations do not affect the first Chern classes),
the above cocycle comparison is valid verbatim, if all
pairs of level indices are shifted from $(0,-1)$ to~$(-1,-2)$,
that is, if we compare the periods of a form on the middle level with
the periods of a form at bottom level.
Since the multi-scale differential on the middle level is
$t_1^{\ell_{\wh\Delta,1}}$ times a top level differential at the intersection
with $D_{\wh\Delta}$, the sections of we are locally comparing with
are sections of $\cL_1 = \cO_\Gamma^{[0]}(-1) \otimes \cL_\Gamma^\top$ as
we claimed. 
\end{proof}
\par
\begin{proof}[Sketch of proof of Proposition~\ref{prop:c1normal}]
We let $m_e = \ell(\Gamma)/\kappa_e$. In $\barmoduli[g,n]$ consider
the divisor $D_e$ corresponding to the single edge~$e$ and denote
by~$\cN_e$ its normal bundle. With the same symbol we denote also
the pullback of this normal bundle under the forgetful map $D_\Gamma \to D_e$.
We claim that (at least outside a subvariety of codimension two)
there is a short exact sequence of quasi-coherent $\cO_{D_\Gamma}$-modules
\begin{equation} \label{eq:oldnb}
0\longrightarrow \cN_\Gamma^{m_e} \to \cN_e \to \cQ_\Gamma \to 0
\end{equation}
where the  coherent sheaf~$\cQ_\Gamma$ is supported on the set $\LG_{2,e}^{\Gamma}(B)$
and this sheaf is given by
\be \label{eq:nbQ}
\cQ_\Gamma \= \bigoplus_{\Delta \in \LG_{2,e}^{\Gamma}(B)} 
\cO_{D_\Gamma} /  I_{D_{\wh{\Delta}}}^{\ell_{\wh{\Delta},a}/\kappa_e} \,,
\ee
where $a = a_{\wh{\Delta},\Gamma}$ as above and where $I_{D_{\wh{\Delta}}}$ is
the ideal sheaf of the divisor~$D_{\wh{\Delta}} \subseteq D_{\Gamma}$. This claim
obviously implies the proposition.
\par
To prove it, we use the local description of the universal family over $\LMS$
given by the plumbing construction described in Section~12 of \cite{LMS}.
At a boundary point that is precisely in the intersection of divisors
$D_{\Gamma_i}$ we let $m_{e,i} = \ell_{\Gamma_i}/\kappa_e$. Then the construction
states in particular that the universal family is constructed using the
plumbing fixture
\bes \bV_e \= \Bigl\{ (u,v) \in \Delta^2\,\colon\,
uv = \prod_{i= L(e^-)}^{L(e^+)} t_i^{m_{e,i}} \Bigr\}
\ees
at the node corresponding to the edge~$e$, where $u$ and $v$ are coordinates
on the surfaces at the upper and lower end of the edge and where $L(e^\pm)$
denotes the levels at the edges of the edge. A local generator of~$\cN_e$
is $\partial/\partial f$ if $uv=f$ is a local equation of the node.
On the other hand, a local generator of~$\cN_\Gamma^{m_e}$ is
$\partial/\partial (t_i^{m_{e,i}})$ if~$\Gamma$ is the undegeneration of
the $i$-th level at the point under consideration. (In particular,
$m_e = m_{e,i}$ in this situation). This follows from the form
of perturbed period coordinates. This implies that at a generic point
of~$D_\Gamma$ (and more generally whenever the edge does not become long)
the natural map $\cN_\Gamma^{m_e} \to \cN_\Gamma$ is an isomorphism. At
the remaining points,
\bes
\frac{\partial}{\partial f} \= \prod_{j= L(e^-) \atop j \neq i}^{L(e^+)}
t_j^{m_{e_j}}\,\frac{\partial}{\partial (t_i^{m_{e_i}})}  \,+\, \cdots
\ees
where the suppressed tangent vectors vanish when  restricted to $D_\Gamma$.
Since $t_j$ are the defining equations
of divisors $D_\Delta$ where the edge becomes long, this implies~\eqref{eq:nbQ}.
\end{proof}
\par
\begin{example} {\rm Consider the stratum $\bP\Omega\cM_{0,5}(a_1,a_2,a_3,a_4,-b)$
with $a_i \geq 0$ and $b = +2+\sum a_i \geq 0$. We study the 'cherry'
divisor~$\Gamma$ (see also \cite[Section~14.4]{LMS})
with one vertex on top level, carrying the unique pole, and two
vertices on lower level, carrying the first two and the third plus forth point,
respectively. The vertices on lower level are each connected to the top level
by a single edge, denoted by~$e_1$ and~$e_2$ respectively. The enhancements are
given by $\kappa_1 = a_1+a_2+1$ and $\kappa_2 = a_3+a_4+1$.
Hence $\ell_\Gamma = \lcm(\kappa_1,\kappa_2)$. 
\par
We compute the degree of the normal bundle using either of the two edges.
Note that the boundary divisor $D_\Gamma$ has two intersection points with
other boundary strata, where~$e_1$ and $e_2$ become long edges. Neighborhoods
of these points are quotient stacks by a cyclic group of order
$m_i = \ell_\Gamma/\kappa_i$. To see this, say where~$e_1$ becomes long, we check
that $\sTw = \ell \bZ \oplus \kappa_1\bZ$ and $\Tw = \langle (0,\kappa_1),
(\kappa_2,-\kappa_2)\rangle$, hence the index is $m_1$, as claimed.
\par
In this example, the bundle $\cN_e$ has degree zero when pulled back
to~$D_\Gamma$, since $D_\Gamma$ is contracted when mapped to $\barmoduli[0,5]$.
Applying the theorem, we get
\bes
\deg(\cN^{m_1}) \= 0 - \frac{1}{m_2} \frac{\kappa_2}{\kappa_2},
\quad \text{hence} \quad \deg(\cN) = \frac{1}{m_1m_2}\,
\ees
and using~$e_2$ we arrive at the same conclusion.
}\end{example}
\par
\medskip
Our next task is to identify the normal bundle as sum of two contributions
from the top an bottom level via push-pull through the level projections
and clutching maps. For this purpose we define 
\be
\cL_{B_\Gamma^\top} \= \cO_{B_\Gamma^\top}
\Bigl( \sum_{\Delta \in \LG_1(B_\Gamma^\top)} {\ell_{\Delta}} D_{\Delta}\Bigr)
\ee
\par
\begin{lemma} \label{le:calLpushpull}
  We have $
p_\Gamma^{\top,*} \cL_{B_\Gamma^\top} \= c_{\Gamma}^* \cL_\Gamma^\top\,.$
\end{lemma}
\par
\begin{proof} We sum the first equation in
Proposition~\ref{prop:divpstar} over all $\Delta \in \LG_1(B_\Gamma^\top)$.
Each~$\wh\Delta$ will appear for all graphs in
$J(\Gamma^\dagger,\widehat{\Delta})$ as discussed at the beginning
of Section~\ref{sec:clutmor}. However thanks to Lemma~\ref{le:autcancel} 
this factor cancels with all the automorphism factors in that
proposition to give the statement we claim.
\end{proof}
\par
The lemma obviously implies
\bes
c_1(\cL_\Gamma^\top) \= \frac{1}{\deg(c_\Gamma)}\cdot c_{\Gamma,*}\, 
p_\Gamma^{\top *} c_1(\cL_{B_\Gamma^\top})\,.
\ees
Since the tautological bundles on top and on bottom level have a pullback 
description by Proposition~\ref{prop:xiasppull}, we have shown that
there exist $\nu^\top \in \CH^1(B_\Gamma^\top)$ and $\nu^\bot \in \CH^1(B_\Gamma^\bot)$
such hat
\be \label{eq:nudecomp}
\nu_\Gamma \,:=\,\c_1(\cN_\Gamma) \= \c_{\Gamma,*}\,(p^{\top})^* \nu_{\Gamma}^\top
\,+\, c_{\Gamma,*}\, (p^{\bot})^*\nu_{\Gamma}^\bot\,.
\ee
\par
\medskip
The normal bundle computation has a generalization to an inclusion
$\frakj_{\Gamma,\Pi}\colon D_{\Gamma} \hookrightarrow D_{\Pi}$ between non-horizontal
boundary strata of relative codimension one, say defined by the $L$-level
graph $\Pi$ and one of its $(L+1)$-level graph degenerations $\Gamma$.
This generalization is needed in Section~\ref{sec:tautring} for recursive
evaluations. Such an inclusion is obtained by splitting one of the
levels of~$\Pi$, say the level~$i\in \{0,-1,\dots,-L\}$. We
define 
\be \label{eq:defLithlev}
\cL_{\Gamma}^{[i]} \=  \cO_{D_\Gamma}
\Bigl(\sum_{\Gamma \overset{[i]}{\rightsquigarrow} %
  \wh{\Delta} } \ell_{\wh{\Delta},-i+1}D_{\wh{\Delta}} \Bigr)
\quad \text{for any}  \quad i\in \{0,-1,\dots,-L\}\,,
\ee
where the sum is over all graphs $\wh{\Delta} \in \LG_{L+2}(\ol{B})$
that yield divisors in~$D_\Gamma$ by splitting the $i$-th level, which in
terms of undegenerations means $\delta_{-i+1}^\complement
(\wh{\Delta}) = \Gamma$.
With the same proof as above, simply shifting
attention to level~$i$ of $\Pi$,  we obtain:
\begin{prop}\label{prop:generalnormalbundle}
For $\Pi \overset{[i]}{\rightsquigarrow} \Gamma$ (or equivalently
$\delta_{-i+1}^\complement(\Gamma)=\Pi$) the Chern class of the normal
bundle $\cN_{\Gamma,\Pi} = \cN_{D_\Gamma/D_\Pi}$  is given by
\be \label{eq:nbinPi}
c_1(\cN_{\Gamma,\Pi}) \= \frac{1}{\ell_{\Gamma,(-i+1)}} \bigl(-\xi_\Gamma^{[i]}
- c_1(\cL_\Gamma^{[i]}) + \xi_\Gamma^{[i-1]} \bigr)\quad \text{in} \quad
\CH^1(D_{\Gamma})\,.
\ee
\end{prop}
\par
With the same proof as in Lemma~\ref{le:calLpushpull} we obtain
\be \label{eq:calLpullbackgeneral}
p_\Gamma^{[i]*} \cL_{B_\Gamma^{[i]}} \= c_{\Gamma}^* \cL_\Gamma^{[i]}
\quad \text{where} \quad
\cL_{B_\Gamma^{[i]}} \= \cO_{B_\Gamma^{[i]}}
\Bigl( \sum_{\Delta \in \LG_1(B_\Gamma^{[i]})} {\ell_{\Delta}} D_{\Delta}\Bigr)\,.
\ee
We can thus write the normal bundle as a sum of bundles that are
$c_\Gamma$-pushforwards of pullbacks from $B_\Gamma^{[i]}$ and from
$B_\Gamma^{[i-1]}$. We express this by saying that the normal bundle
is {\em supported on the levels~$i$ and $i-1$} (for $i \in \bZ_{\leq 0}$).
\par
We need some compatibility statements for pullbacks of normal bundles 
to more degenerate graphs. We start with auxiliary bundles, whose
pullback we need, too.
\par
\begin{lemma}
Let $\Gamma\in \LG_L(B)$ and let $\Gamma \overset{[i]}{\rightsquigarrow}
\wh{\Delta}$ be a codimension one degeneration of $\Gamma$ obtained by
splitting the level~$i\in \{0,\dots,-L\}$. Then for every $j\in \{0,\dots,-L\}$
\bes
\frakj_{\wh{\Delta},\Gamma}^*(\xi_{\Gamma}^{[j]})
\= \begin{cases} \xi_{\wh{\Delta}}^{[j]},&
\text{if $j\geq i$}\\ \xi_{\wh{\Delta}}^{[j-1]}& \text{if $j<i$}
\end{cases}
\ees
and %
\bes
\frakj_{\wh{\Delta},\Gamma}^*\left(c_1\left(\cL_{\Gamma}^{[j]}\right)\right)
\= \begin{cases} c_1\left(\cL_{\wh{\Delta}}^{[j]}\right),&
\text{ if $j> i$} \\
c_1\left(\cL_{\wh{\Delta}}^{[j-1]}\right)& \text{ if $j< i$}\\
c_1\left(\cL_{\wh{\Delta}}^{[j-1]}\right)+\xi_{\wh{\Delta}}^{[j-1]}
-\xi_{\wh{\Delta}}^{[j]} \!\!&
\text{ if $j= i$.}
\end{cases}
\ees
\end{lemma}
\begin{proof}
For the cases $j\not = i$, the claim is obvious since level $i$ is untouched in
the degeneration from $\Gamma$ to $\wh\Delta$. If $i=j$ then the second claim follows from
\bas \ 
\frakj_{\wh{\Delta},\Gamma}^*\left(c_1\left(\cL_{\Gamma}^{[j]}\right)\right)
&\= \frakj_{\wh{\Delta},\Gamma}^*
\Biggl(\,\,\sum_{\Gamma \overset{[j]}{\rightsquigarrow} \Lambda,
\,\, \Lambda \neq \wh{\Delta}}
\ell_{\Lambda,-j+1}[D_{\Lambda}] +\ell_{\wh{\Delta},-j+1}[D_{\wh{\Delta}}] \Biggr )\\
&\=c_1\left(\cL_{\wh{\Delta}}^{[j]}\right)+c_1\left(\cL_{\wh{\Delta}}^{[j-1]}\right)+\ell_{\wh{\Delta},-j+1}
\c_1(\cN_{D_{\wh{\Delta}}/D_{\Gamma} }) \\
 &\=c_1\left(\cL_{\wh{\Delta}}^{[j]}\right)+c_1\left(\cL_{\wh{\Delta}}^{[j-1]}\right)+\left(-\xi_{\wh{\Delta}}^{[j]}+\xi_{\wh{\Delta}}^{[j+1]}-c_1\left(\cL_{\wh{\Delta}}^{[j]}\right)\right)\\
&\=c_1\left(\cL_{\wh{\Delta}}^{[j-1]}\right)+\xi_{\wh{\Delta}}^{[j-1]}
-\xi_{\wh{\Delta}}^{[j]}\,.
\eas 
The case $j=i$ for the first claim about  pulling back $\xi_{\Gamma}^{[j]}$
follows directly from the definition of $\cO_\Gamma^{[j]}(-1)$ by local
generators. Alternatively
one can compute it by applying the relation~\eqref{eq:xirel}. If the chosen
marked point is supported on the $j$-th level of $\wh{\Delta}$, the calculation
is straightforward. If the marked point~$h$ is supported on the $(j-1)$st
level of~$\wh\Delta$, then $\wh\Delta$ appears among the boundary terms
of~\eqref{eq:xirel}. Pulling back makes the normal bundle appear, and
thus $\xi^{[i]}_\Delta$ in the formula from Theorem~\ref{thm:nb}. 
The remaining boundary terms of~\eqref{eq:xirel} can be grouped into
those where~$h$ ends up at level~$j-1$ or~$j-2$ after pulling
back to~$\wh\Delta$. These groups cancel with the remaining two terms of
the normal bundle.
\end{proof}
\par
As a consequence of the preceding lemma and Theorem~\ref{thm:nb} we
obtain:
\par
\begin{cor} \label{cor:pullbacknormal}
Let $\Gamma\in\LG_L(B)$ and let $\wh{\Delta}$ be a codimension one degeneration
of the $(-i+1)$-level of $\Gamma$, i.e. such that $\Gamma=
\delta_i^\complement(\wh{\Delta})$, for some $i\in \{1,\dots,L+1\}$. Then  
\[\frakj_{\wh{\Delta},\Gamma}^*\left(\ell_{\Gamma,j}
\c_1\bigl(\cN_{\Gamma/\delta_{j}^\complement(\Gamma)}\bigr)\right)
\= \begin{cases}
\ell_{\wh{\Delta},j}\,\,\c_1\left(\cN_{\wh{\Delta}/\delta_{j}^\complement(\wh{\Delta})}\right)
,&  \text{ for } j< i \\
\ell_{\wh{\Delta},j+1}\c_1\left(\cN_{\wh{\Delta}/\delta_{(j+1)}^\complement(\wh{\Delta})}\right)
& \text{ otherwise.}
\end{cases}\]
\end{cor}
\par

\section{The tautological ring} \label{sec:tautring}

In this section we give the precise definition of the tautological ring and
prove Theorem~\ref{thm:addgen}. We define the {\em tautological rings of strata}
as  the smallest set of $\bQ$-subalgebras $R^\bullet(\LMS) \subset \CH^\bullet(\LMS)$
which
\begin{itemize}
\item contains the $\psi$-classes attached to the marked points,
\item is closed under the pushfoward of the map forgetting a
regular marked point (a zero of order zero), and
\item is closed under the maps $ \zeta_{\Gamma,*}p^{[i],*}$
  defined in Proposition~\ref{prop:prodcover} for all level
  graphs~$\Gamma$\,.
\end{itemize}
Our goal is to provide additive generators of this ring and show that
the main players, normal bundles and the logarithmic cotangent bundle
have Chern classes in this ring. The main tool
is the excess intersection formula that allows to compute the intersection
product of boundary strata, possibly decorated with $\psi$-classes.
\par
In fact, there are two definitions of other (refined) tautological rings. One
option is the refined ring $R_{\rm ref}^\bullet(\LMS)$ that is closed under
all the clutching morphisms $\zeta^{\rm ref}_*p^{[i],*} $ that distinguish
the components of boundary strata that are reducible due to inequivalent
prong-matchings. Obviously,  $R^\bullet(\LMS) \subseteq R_{\rm ref}^\bullet(\LMS)
\subset \CH^\bullet(\LMS)$. There is an analog of Theorem~\ref{thm:addgen},
replacing in the additive generators the inclusion maps $\fraki_\Gamma$ of
reducible boundary strata by the inclusion maps of irreducible components.
The proofs below can be adapted to that setting. 
\par
The second option is to include $D_{\text{h}}$ or equivalently clutching morphism
for horizontal nodes into the definition of the tautological ring
$R_{\text{h}}^\bullet(\LMS)$ (and not distinguishing inequivalent prong-matchings,
although one could obviously do both). Obviously $R^\bullet(\LMS)
\subseteq R_{\text{h}}^\bullet(\LMS)$.
\par
In order to express $c_1(\Omega_{\overline{B}})$ we need $D_{\text{h}}$, so we need to
work in $R_{\text{h}}^\bullet(\LMS)$. However, one of the main points of this section
is that the Chern polynomial of the logarithmic cotangent bundle belongs to
the smallest of the natural candidates for a tautological ring. It seems
interesting to decide which of the two inclusions of tautological rings
defined above are strict.

\subsection{Excess intersection formula}

Suppose we are given two level graphs $\Lambda_1$ and $\Lambda_2$ without horizontal
nodes and the corresponding inclusion maps $\fraki_{\Lambda_j} \colon D_{\Lambda_j} \to \LMS$
into a compactified stratum. For a class $\alpha \in \CH^\bullet(D_{\Lambda_2})$, we want
to compute $\fraki_{\Lambda_1}^* \fraki_{\Lambda_2,*} \alpha$ as the push-forward from the
maximal-dimensional boundary strata in the support of $D_{\Lambda_1}\cap D_{\Lambda_2}$,
in terms of an $\alpha$-pullback and normal bundle classes encoding the excess
intersection of $D_{\Lambda_1}$ and $D_{\Lambda_2}$. 
We say that a level graph $\Pi$ is {\em a $({\Lambda_1,\Lambda_2})$-graph} if there
are undegeneration morphisms $\rho_i \colon \Pi \to \Lambda_i$, i.e.\ edge
contraction morphisms with the property that there are subsets~$I_{\Lambda_1}$ and~$I_{\Lambda_2}$
of level passages of $\Pi$ such that $\delta_{I_{\Lambda_1}}(\Pi) = \Lambda_1$ and $\delta_{I_{\Lambda_2}}(\Pi) = \Lambda_2$.
(Automorphisms of $\Lambda_i$, i.e.\ the stack structure of $D_{\Lambda_i}$
stemming from permuting the edges requires the distinction between~$\delta$'s
and the~$\rho_i$'s.)
We call $\Pi$ {\em a generic $(\Lambda_1, \Lambda_2)$-graph}, if $I_{\Lambda_1}^\complement \cap I_{\Lambda_2}^\complement
= \emptyset$. The intersection formula will 
use the inclusion maps as indicated in the diagram 
\begin{center}
	\begin{tikzcd}
D_\Pi  \arrow{r}{\frakj_{\Pi,\Lambda_2}} 
\arrow{d}{\frakj_{\Pi,\Lambda_1}}  & D_{\Lambda_2}  \arrow{d}{\fraki_{\Lambda_2}}\\
D_{\Lambda_1} \arrow{r}{\fraki_{\Lambda_1}} & \bP\LMS
	\end{tikzcd}
\end{center}
\par
\begin{prop} \label{prop:pushpullcomm}
For any $\alpha \in \CH^\bullet(D_{\Lambda_2})$ we can express its push-forward 
pulled back to $\Lambda_1$ as
\be
\fraki_{\Lambda_1}^* \fraki_{\Lambda_2,*} \alpha \= \sum_{\Pi}
 \frakj_{\Pi, \Lambda_1,*} \Bigl(\nu^\Pi_{\Lambda_1\cap \Lambda_2} \cdot \frakj_{\Pi, \Lambda_2}^* 
\alpha \Bigr)\,,
\ee
where the sum is over all generic $(\Lambda_1,\Lambda_2)$-graphs~$\Pi$. In
this expression
\[\nu^\Pi_{\Lambda_1\cap \Lambda_2} = \prod_{k \in I_{\Lambda_1}\cap I_{\Lambda_2}}
\frakj_{\Pi,\delta_k(\Pi)}^*\left(\nu_{\delta_k(\Pi)}\right)\]
is the product of the pull-back to $D_\Pi$ of the first Chern classes of the normal
bundles of the divisors containing  both~$D_{\Lambda_1}$ and~$D_{\Lambda_2}$.
\end{prop}
\par
\begin{proof} By the excess intersection formula (\cite[Proposition~17.4.1]{Fulton})
we have to show that the fiber product $\cF_{\Lambda_1,\Lambda_2} = D_{\Lambda_1}
\times_{\bP\LMS} D_{\Lambda_1}$ is the coproduct  $\cD = \coprod D_\Pi$ over all
generic $(\Lambda_1,\Lambda_2)$-graphs~$\Pi$ and to identify the excess normal bundle.
\par
First we define a map $\varphi: \cD \to \cF_{\Lambda_1,\Lambda_2}$ via the universal
properties of the coproduct and  the fiber product.  It is  the map induced by the
inclusions $\frakj_{\Pi,\lambda_i}: D_\Pi \to D_{\Lambda_i}$, for each generic
$(\Lambda_1,\Lambda_2)$-graph~$\Pi$.
\par
To give a converse natural transformation on objects we take a family parameterized
by~$\cF_{\Lambda_1,\Lambda_2}$, i.e.\ a pair of a family $(\cX_1,\bfeta_1)$ of multi-scale
differentials compatible with an undegeneration of~$\Lambda_1$ and a family
$(\cX_2,\bfeta_2)$ compatible with an undegeneration of~$\Lambda_2$. If we forget
the differentials, we can construct a family of pointed stable curves~$(\cX,\bz)$
over some stable graph~$\Pi$, which is generic as a $(\Lambda_1,\Lambda_2)$-stable
graph (see \cite{GP03} or \cite{SvZ}). We make~$\Pi$ into a level graph by declaring
a vertex~$v_1$ to be on top of~$v_2$ if this holds for either of their images
in~$\Lambda_1$ or in~$\Lambda_2$. Compatibility of the fiber product ensures that
this definition is consistent. This definition moreover ensures that $\Pi$ is
$(\Lambda_1,\Lambda_2)$-generic in our sense of enhanced level graphs. The construction
of~$\cX$ moreover exhibits a bijection of its $f$-relative components (relative
to the structure morphism~$f$ to the base) with the $f$-relative components of~$\cX_1$
(and also those of~$\cX_2$). We can thus pull back the differential~$\eta_1$ on each
of those components of~$\cX$ (or we could pull back~$\eta_2$) to a collection of
differentials~$\bfeta$ on~$\cX$. To see that this indeed defines a twisted differential compatible
with~$\Pi$, only the global residue conditions requires a non-trivial verification.
By definition of $(\Lambda_1,\Lambda_2)$-genericity and because of the unique ordering
of profiles shown in \autoref{prop:ordering}, for each level~$-i$ of~$\Pi$  there is
an index~$j \in \{1,2\}$ and a level $-i'$ of $\Lambda_j$ such that the connected
components of the subgraph of~$\Pi$ above level~$-i$ are in natural bijection with
the connected components of the subgraph of~$\Lambda_j$ above level~$-i'$. This implies
the global residue condition.
The enhancements of the edges~$\Pi$ are given by the identification of the edges with
those of  $\Lambda_1$ and $\Lambda_2$ in the first step of the converse construction. In
the same way we provide~$(\cX,\bfz,\bfeta)$ with a collection of prong-matchings and
pull back the rescaling ensembles as in \cite[Section~7]{LMS} to complete the
construction of a family of multi-scale differentials compatible with an
undegeneration of~$\Pi$. The converse natural transformation on morphisms is simply
the map constructed for families of pointed stable curves.
\par
The excess normal bundle is in general given by $E = \frakj_{\Pi,\Lambda_1}^*
\cN_{\Lambda_1}/\cN_{\Pi,\Lambda_2}$, where the normal sheaves appearing are the normal
sheaves of the morphisms $\fraki_{\Lambda_1}$ and $\frakj_{\Pi,\Lambda_2}$. Since
by \autoref{prop:ordering} the non-horizontal boundary strata are smooth and simple
normal crossing, the previous normal sheaves are vector bundles and  
they are given as the direct sum of the pull-back of the normal bundles of
appropriate divisors. More specifically $\cN_{\Lambda_1}=\oplus_{i=1}^{L(\Lambda_1)}
\cN_{\delta_i(\Lambda_1)}$ and $\cN_{\Pi,\Lambda_2}=\oplus_{i\in I_{\Lambda_2}^\complement}
\cN_{\delta_i(\Pi)}$. This implies that~$E$ is the direct sum of the the normal bundles
of the levels common to both~$\Lambda_1$ and~$\Lambda_2$ (pulled back to~$D_\Pi$)
and thus its  top Chern class is as claimed in the proposition.
\end{proof}
\par 
At the expense of introducing more notation, the excess intersection formula
can be generalized in two ways. First, the ambient space might be a boundary
stratum associated to a codimension $L$-level graph~$\Gamma$, as summarized in the diagram
\begin{center}
	\begin{tikzcd}
	D_\Pi  \arrow{r}{\frakj_{\Pi,\Lambda_2}} 
	\arrow{d}{\frakj_{\Pi,\Lambda_1}}  & D_{\Lambda_2}  \arrow{d}{\frakj_{\Lambda_2,\Gamma}}\\
	D_{\Lambda_1} \arrow{r}{\frakj_{\Lambda_1,\Gamma}} & D_{\Gamma}
	\end{tikzcd}
\end{center}
of inclusions. In this situation we define $\nu^\Pi_{(\Lambda_1\cap \Lambda_2)/\Gamma}$
to be the product of the pull-back to $\Pi$ of the Chern classes of the normal bundles $\cN_{\Gamma'/\Gamma}$, where~$\Gamma'$ ranges
over all codimension 1 non-horizontal degenerations~$\Gamma'$ of $\Gamma$ 
that are  common to $\Lambda_1$ and $\Lambda_2$. As above, we denote
appropriate pullbacks of this product by the same letter. The excess intersection
formula then reads
\begin{equation}
\label{eq:generalizedpush-pull}
\frakj_{\Lambda_1,\Gamma}^* \frakj_{\Lambda_2,\Gamma*} \alpha \= \sum_{\Pi}
\frakj_{\Pi, \Lambda_1,*} \Bigl(\nu^\Pi_{(\Lambda_1\cap \Lambda_2)/\Gamma} \cdot \frakj_{\Pi, \Lambda_2}^* 
\alpha \Bigr)\,,
\end{equation}
where the sum ranges over all $(\Lambda_1,\Lambda_2)$-graphs~$\Pi$.
\par
In the more general case that the level graphs $\Lambda_i$ also have horizontal nodes, 
there is an obvious generalization of this proposition. A general undegeneration
of boundary graphs is given by a pair $\delta = (\delta_{\rm ver},\delta_{\rm hor})$
consisting of a level undegeneration~$\delta_{\rm ver}$ as in Section~\ref{sec:DegUndeg}
and an undegeneration of horizontal nodes~$\delta_{\rm hor}$. One defines~$\Pi$ to be
a $(\Lambda_1,\Lambda_2)$-graph if there are undegenerations $\delta_i$ such that
$\delta_i(\Pi) = \Lambda_i$, for $i=1,2$. Such a a graph is generic if the vertical
undegenerations are generic as above and, moreover, if the horizontal contractions
are generic in the usual sense of $\barmoduli$ (see \cite{GP03} or
\cite[Chapter XVII]{acgh2}). We leave it to the reader to adapt the previous
proposition and the subsequent argument to the general context.
\par

\subsection{Relations in the tautological ring and the proof of Theorem~\ref{thm:addgen}}

Before concluding the proof of Theorem~\ref{thm:addgen} we need
some relations in the tautological ring. These relations are essentially
known, but we restate them here for convenience and to justify a version
for the spaces~$\proj\LMS[\bfmu][\bfg,\bfn][\frakR]$, i.e.\ possibly disconnected,
with residue conditions, and for multi-scale differentials rather than on the
incidence variety compactification. Recall the notation of \autoref{sec:comgenstra}
for generalized strata, where the $(i,j)$-th marked point is the $j$-th marked
point of the $i$-th surface and has order $m_{i,j}\in \bZ$.
\par
\begin{prop}[{\cite[Theorem~6(1)]{SauvagetMinimal}}] \label{prop:Adrienrel}
The class $\xi$ on $\overline{B} = \proj\LMS[\bfmu][\bfg,\bfn][\frakR]$ can be
expressed using the $\psi$-class at the $(i,j)$-the marked point as 
\be \label{eq:xirel}
\xi \= (m_{i,j}+1) \psi_{(i,j)} \,-\, \sum_{\Gamma \in \tensor[_{(i,j)}]{\LG}{_1}(\overline{B})}
\ell_\Gamma [D_\Gamma]\,
\ee
where  $\tensor[_{(i,j)}]{\LG}{_1}(\overline{B})$ are two-level graphs with
the leg~$(i,j)$ on lower level.
\end{prop}
\par
The fact that our $D_\Gamma$ record prong-matching equivalence classes
makes up for the difference between our formula and the one appearing in
\cite{SauvagetMinimal}, since the $\kappa_\Gamma$-factor appearing in loc.~cit.\
become $\ell_\Gamma=\kappa_\Gamma/g_\Gamma$ in our formula.
\par
\begin{proof} We expand the argument given in \cite[Proposition~2.1]{ChenTauto}
including the boundary terms. Let $\pi: \cX \to \overline{B}$ be the universal family
and $S_i$ be the image of the section given by the $i$-th marked point. The evaluation
map gives an isomorphism of $\pi^* \cO(-1)$ and $\omega_{\cX/\overline{B}}$ outside
the locus~$S_i$ and the lower level components of the boundary divisors. Consider
the construction of the universal differential over $\LMS$ in  \cite[Section 12]{LMS},
in particular in the plumbing fixture~(12.6) of loc.\ cit. The difference of $t$-powers
at the two branches is just $\ell_\Gamma$ in our notation, and all this is unchanged
in the presence of a GRC~$\frakR$. We we deduce that
\be \label{eq:pullbxi}
\pi^* \xi \=  c_1(\omega_{\cX/\overline{B}}) - \sum_{i=1}^n m_i S_i
- \sum_{\Gamma \in \twolev} \ell_\Gamma [\cX_\Gamma^\bot]\,,
\ee
where $\cX_\Gamma^\bot$ is the lower level component in the universal family over the
divisor $D_\Gamma$. We intersect both sides with~$S_i$ and apply $\pi_*$. Using
$\pi_*(S_i^2) = -\psi_i$ and $\pi_*( \omega_{\cX/\overline{B}} \cdot S_i)=\psi_i$,
this gives the claim.
\end{proof}
\par
We need a similar generalization of another relation of Sauvaget to our framework
that will be needed for the final evaluation of top degree classes (see
the end of Section~\ref{sec:Chern} and \cite{CoMoZadiffstrata})).
Consider a generalized stratum defined by a residue
condition~$\frakR$ as defined in Section~\ref{sec:comgenstra}. Suppose we remove
one element from the set~$\lambda_\frakR$ constraining the residues in the definition
of~$\frakR$. We denote this new set by $\lambda_{\frakR_0}$ and $\frakR_0$ the new
set of residue conditions. Two cases might occur. Either 
$\proj\LMS[\bfmu][\bfg,\bfn][\frakR] = \proj\LMS[\bfmu][\bfg,\bfn][\frakR_0]$
or $\proj\LMS[\bfmu][\bfg,\bfn][\frakR] \subsetneq \proj\LMS[\bfmu][\bfg,\bfn][\frakR_0]$
is a divisor. We consider the second case here and note that this condition is
equivalent to $S:=R \cap \frakR \subset S_0:=R \cap \frakR_0$ is codimension
one (rather than the two being equal), where~$R$ is the space of residues
defined in~\eqref{eq:spaceR}.
Consider now a boundary stratum~$D_\Gamma$ in~$\proj\LMS[\bfmu][\bfg,\bfn][\frakR_0]$.
For each level~$i$ of~$D_\Gamma$ and any GRC~$\frakR$ containing~$\frakR_0$,
we define the {\em residue condition $\frakR^{[i]}$ induced by~$\frakR$} to be the residue
condition given at level~$i$ by the auxiliary level graph~$\widetilde{\Gamma}_\frakR$
as defined in Section~\ref{sec:comgenstra}, created with the help of the auxiliary
vertices of~$\frakR$. For the top level we write $\frakR^\top$ for the induced residue
condition on top level. It can
be simply computed by discarding from the parts $\lambda_\frakR$ all indices of
edges that go to lower level in $D_\Gamma$.
\par
\begin{prop}[{\cite[Proposition~7.6]{SauvagetMinimal}}] \label{prop:AdrienR}
The class of the stratum $\proj\LMS[\bfmu][\bfg,\bfn][\frakR]$ with residue
condition~$\frakR$ compares inside Chow ring of the generalized stratum $\overline{B}
= \proj\LMS[\bfmu][\bfg,\bfn][\frakR_0]$ to the class $\xi$ by the formula
\be \label{eq:GRCremove}
    [\proj\LMS[\bfmu][\bfg,\bfn][\frakR]] \= - \xi \,\,-
    \sum_{\Gamma \in {\LG}^{\frakR}{_1}(\overline{B})}
\ell_\Gamma [D_\Gamma] \,\, - \sum_{\Gamma \in {\LG}_{1,\frakR}(\overline{B})}
\ell_\Gamma [D_\Gamma]\,,
\ee
where $\mathrm{LG}^{\frakR}_1(\overline{B})$ are two-level graphs with
$R_\Gamma \cap \frakR^\top = R_\Gamma \cap \frakR_0^\top$, i.e., where the
GRC on top level induced by~$\frakR$ does no longer introduce an extra condition
and where ${\LG}_{1,\frakR}(\overline{B})$ are two-level graphs where all the
legs  involved in the condition forming $\frakR \setminus \frakR_0$ go to lower
level.
\end{prop}
\par
\begin{proof} Consider that map $s: \cO_{\overline{B}}(-1) \to S_0/S$ to the
constant rank one vector bundle, mapping a points $(X,\omega)$ to the
(equivalence class mod~$S$ of the) tuple of residues of~$\omega$, which
defines point in~$S_0$. The vanishing locus in the interior of~$\overline{B}$
is by definition $\proj\LMS[\bfmu][\bfg,\bfn][\frakR]$, and in usual period
coordinates we see that the vanishing order is one there. To understand the
boundary contribution, consider first boundary divisors neither in
$\mathrm{LG}^{\frakR}_1(\overline{B})$ nor in  ${\LG}_{1,\frakR}(\overline{B})$.
For those, being in the vanishing locus of~$s$ is a non-trivial (divisorial)
condition, and thus this locus is of codimension two and irrelevant for the equation.
It remains to justify the vanishing statement and the vanishing order for
the other divisors. Any section of $\cO_{\overline{B}}(-1)$ decays like $t_1^{\ell_\Gamma}$ near
lower level components by construction of the compactification in
\cite[Section~12]{LMS}, where $t_1$ is a transversal coordinate. Consequently,
any $D_\Gamma \in {\LG}_{1,\frakR}(\overline{B})$
is in the support of the cokernel of the map~$s$, with multiplicity~$\ell_\Gamma$.
For $D_\Gamma$ in $\Gamma \in \mathrm{LG}^{\frakR}_1(\overline{B})$ the residues at
the poles going to level zero are zero (mod~$S$) all along~$D_\Gamma$ by definition.
Transversally, they become non-zero with the growth of the modification
differential (see the construction in \cite[Section~11]{LMS}), since the
modification differential must be generically non-zero on~$D_\Gamma$
if $\frakR$ imposes a non-trivial condition generically on the stratum, but
none along $D_\Gamma$. Since the modification differential scales with $t_1^{\ell_\Gamma}$
this proves the claim on the multiplicity of~$D_\Gamma$ in this case, too.
\end{proof}
\par
\medskip
We are now ready to prove that the tautological ring is finitely generated by the
additive generators displayed in Theorem~\ref{thm:addgen}.
\par
\begin{proof}[Proof of Theorem~\ref{thm:addgen}]
We let $R_{fg}^\bullet(\overline{B})$ be the vector space spanned by the classes
$\zeta_{\Gamma_*} (\,\prod_{i=0}^{-L(\Gamma)} p_\Gamma^{[i],*} \alpha_i\,) $,
where $\alpha_i$ is a	monomial in the $\psi$-classes supported on level~$i$
of the graph~$\Gamma$, where $\Gamma\in \LG(\overline{B})$ ranges among all
level graphs without horizontal nodes. Obviously this is a finite dimensional v
ector space since for any stratum~$\mu$ there are only finitely many level graphs,
and for each of them there is a finite number of monomials that give a non-zero class.
\par
By our definition of the tautological ring of the moduli space of multi-scaled
differentials, clearly of  $R_{fg}^\bullet(\overline{B})\subseteq R^\bullet(\overline{B})$.
\par
We show now that $R_{fg}^\bullet(\overline{B})$ is actually a subring of the
tautological ring, i.e., that it is closed under the intersection product. We prove
this by iteratively applying the projection formula and the excess intersection
formula \eqref{eq:generalizedpush-pull}. In the first step, for any two
classes $\alpha_{j}\in \CH^*(D_{\Lambda_j})$ the projection formula and
Proposition~\ref{prop:pushpullcomm} imply 
\ba
 \fraki_{\Lambda_1\,*}(\alpha_{1})\cdot\fraki_{\Lambda_2\,*}(\alpha_{2})
&\=\sum_{\Pi} \fraki_{\Lambda_1\, *}\Bigl( \alpha_{1} \cdot \frakj_{\Pi, \Lambda_1,*}
\Bigl(\nu^\Pi_{\Lambda_1\cap \Lambda_2} \cdot \frakj_{\Pi, \Lambda_2}^*  \alpha_2 \Bigr)\Bigr)\\
&\= \sum_{\Pi} \fraki_{\Pi,*} \bigl(\nu^\Pi_{\Lambda_1\cap \Lambda_2} \cdot \frakj_{\Pi\Lambda_1 }^*(\alpha_{1})
\cdot \frakj_{\Pi,\Lambda_2}^*(\alpha_{2})\bigr) %
\ea
where the sums are over all generic $(\Lambda_1,\Lambda_2)$-graphs~$\Pi$.
The excess intersection class $\nu^\Pi_{\Lambda_1\cap \Lambda_2}$ is given by pull-backs of normal bundles of divisors. By repeatedly applying \autoref{cor:pullbacknormal}, we see that the pull-back of the class of the normal bundle of a divisor is given by the class of the normal bundle of $D_\Pi$ in a codimension one undegeneration. The shape of such a class was computed in \eqref{eq:nbinPi}. By using the compatibility expressed in \eqref{prop:xiasppull} between level-wise tautological line classes and the tautological line classes on the level strata, together with \autoref{prop:Adrienrel}, we see that the classes of these normal bundles are given by $\psi$-class contributions and boundary contributions given by codimension one degenerations of~$\Pi$. If there are no boundary contributions, then we are done since we obtained an expression in terms of elements of  $R_{fg}^\bullet(\overline{B})$ supported on $\Pi$. If this is not the case, we can apply the same projection formula and excess intersection formula argument as before to these boundary contributions. (Now we have to use the more general excess intersection formula \eqref{eq:generalizedpush-pull} with ambient $\Pi$). This process has to terminate since the dimension of the boundary strata appearing in the excess intersection factor is decreasing, so at some point the excess class contribution will be trivial.   Hence we have shown that $R_{fg}^\bullet(\overline{B})$ is a subring of the tautological ring.
\par
In order to show that $R_{fg}^\bullet(\overline{B})$ is equal to
$R^\bullet(\overline{B})$, we need to show that $R_{fg}^\bullet(\overline{B})$
is closed under push-forward of clutching morphism and under $\pi$-pushforward.
The first statement is clear. For the second we argue inductively on the
dimension of~$\ol{B}$, starting with the obvious case $\dim(\ol{B})=0$. 
We may assume by induction hypothesis that the $\pi$-pushforwards of elements
in $R_{fg}^\bullet(\overline{B})$ are in $R_{fg}^\bullet(\pi(\overline{B}))$  for
any stratum of dimension less than the dimension of~$\ol{B}$.
\par
We first  show that $\pi_*( \zeta_{\Gamma,*} \psi_{n+1}^{\ell+1}) \in
R_{fg}^\bullet(\ol{B})$ for any graph $\Gamma$ with at least  two levels.
Let~$i$ be the level of~$\Gamma$ that contains the $n+1$-st marked point.
For $\psi_{n+1}^{\ell+1}$ to be non-zero we need the component containing
the $n+1$-st marked point to be positive-dimensional (taking GRC into account).
Let~$\Gamma'$ be the level graph obtained from~$\Gamma$ by forgetting this
point. There is thus a well-defined projection map $\pi^{[i]} : B_\Gamma^{[i]}
\to B_{\Gamma'}^{[i]}$ of generalized strata. Recalling that $\psi_{n+1} =
p_{\Gamma}^* \psi_{n+1}$ by our general abuse of notation we find
$\pi_*( \zeta_{\Gamma,*} \psi_{n+1}^{\ell+1}) = \zeta_{\Gamma',*} p_{\Gamma'}^*
\pi^{[i]}_* \psi_{n+1}^{\ell+1}$. By induction we know that
$\pi^{[i]}_* \psi_{n+1}^{\ell+1} \in R_{fg}^\bullet(\pi^{[i]}(B_\Gamma^{[i]}))$
and since the collection of rings $R_{fg}^\bullet(\cdot)$ is already known to
be stable under $\zeta_{\Gamma',*} p_{\Gamma'}^*$, we conclude that
$\pi_*( \zeta_{\Gamma,*} \psi_{n+1}^{\ell+1}) \in R_{fg}^\bullet(\ol{B})$.
\par
Second, in order to treat the case when $\Gamma$ is the trivial graph,
we consider $\ol{\cX}=\LMS[\mu,0][g,n+1]$, $\ol{B}=\LMS[\mu,0][g,n]$ and
the commutative diagram
\begin{center}
	\begin{tikzcd}
	\LMS[\mu,0][g,n+1]  \arrow{r}{f_{n+1}} 
	\arrow{d}{\pi}  & \barmoduli[g,n+1]  \arrow{d}{\pi_{n+1}}\\
	\LMS[\mu] \arrow{r}{f_{n}} & \barmoduli[g,n]
	\end{tikzcd}
\end{center}
where $\pi$ and $\pi_{n+1}$ are the maps forgetting the last point
and $f_{n+1}$ and $f_{n}$ are the maps forgetting the twisted differential.
These vertical maps are the universal families over their images respectively.
Consequently,
\bes
f_n^* \kappa_\ell \= f_n^* (\pi_{n+1})_*(\psi_{n+1}^{\ell+1}) \= \pi_*(f_{n+1}^*(\psi_{n+1}^{\ell+1})).
\ees
Recall that we abuse notation and identify $\psi$ and $\kappa$-classes in
$\CH^*(\overline{B})$ with their pull-back from $\barmoduli[g,n]$. We have
thus shown that $\pi_*(\psi_{n+1}^{\ell+1}) = \kappa_\ell$ also holds in
$\CH^*(	\overline{B})$. We thus only need to show that $\kappa_\ell\in
R_{fg}^\bullet(\overline{B})$. As before, the  special case of the dilaton
equation  $\kappa_\ell = \pi_* (\omega_{\cX/\overline{B}}^{\ell+1})$ holds also
in $\CH^*(\ol{B})$. Recall that $[\cX_\Gamma^\bot]$ is the lower level
component in the universal family over the divisor $D_\Gamma$.
From~\eqref{eq:pullbxi} we deduce that
\bes
\kappa_\ell \= \pi_*\left(\Bigl(\pi^* \xi \,+\,\sum_{i=1}^n m_i S_i
\,+ \sum_{\Gamma \in \twolev} \ell_\Gamma [\cX_\Gamma^\bot]\Bigr)^{\ell +1}\right)
\ees
is a linear combination of terms of the form $\xi^a \pi_*(S_i^{b_i} \prod
[\cX_\Gamma^\bot]^{c_\Gamma})$ with $a + b_i + \sum_\Gamma c_\Gamma = \ell+1$,
since the sections~$S_i$ are disjoint. The $\xi$-powers are tautological by
Proposition~\ref{prop:Adrienrel} and so we only need to study the $\pi_*$-term.
Let $\fraki: D_0 := \bigcap_{\Gamma: c_\Gamma>0, i\in \Gamma^\bot} D_\Gamma \to
\overline{B}$ be the inclusion of the intersection of boundary divisors
where the $i$-th marked point is on the bottom level, which the image of
the support $S_i^{b_i} \prod [\cX_\Gamma^\bot]^{c_\Gamma}$ under $\pi$.
Let $\widetilde{\fraki} : 
\cX_0 := \bigcap_{\Gamma: c_\Gamma>0, i\in \Gamma^\bot} \cX_\Gamma^\bot \to \ol{\cX}$ be
the corresponding inclusion in the total space of the family. Let
$\frakj_{0,\Gamma}: D_0 \to D_\Gamma$ and $\widetilde{\frakj}_{0,\Gamma} : \cX_0
\to \cX_\Gamma$ the inclusions into codimension one divisors. Finally let
$\sigma_i$ be the section of the $i$-th marked point and abusively also its
restriction to~$D_\Gamma$ and to~$D_0$.
\par
Suppose that $b_i >0$. Then using $\sigma_i^* S_i^k = (-\psi_i) \sigma_i^*(S_i^{k-1})$ 
we find 
\bas
\pi_*\Bigl(S_i^{b_i} \prod_\Gamma [\cX_\Gamma^\bot]^{c_\Gamma}\Bigr)
&\=
\pi_*\sigma_{i,*} \sigma_i^*\Bigl( S_i^{b_i-1}\cdot \tilde{\fraki}_* \bigl(\prod_\Gamma
\widetilde{\frakj}_{0,\Gamma}^* \cN_{\cX^\bot_\Gamma}^{c_\Gamma-1} \bigr)\Bigr) \\
&\= (-\psi_i)^{b_i-1} \cdot \sigma_i^* \Bigr( \tilde{\fraki}_* \bigl(\prod_\Gamma
\widetilde{\frakj}_{0,\Gamma}^* \cN_{\cX^\bot_\Gamma}^{c_\Gamma-1} \bigr)\Bigr) \\
&\= (-\psi_i)^{b_i-1}  \cdot  {\fraki}_* \Bigl(\prod_\Gamma \sigma_i^*
\bigr( \widetilde{\frakj}_{0,\Gamma}^* \cN_{\cX^\bot_\Gamma}^{c_\Gamma-1} \bigr)\Bigr) \\ 
& \= (-\psi_i)^{b_i-1}  \cdot  {\fraki}_* \Bigl(\prod_{\Gamma}
{\frakj}_{0,\Gamma}^* \cN_{\Gamma}^{c_\Gamma-1} \Bigr) \,, 
\eas
which is in $R_{fg}^\bullet(\overline{B})$  by Theorem~\ref{thm:nb}\,. If
$b_i=0$ the expression  $\pi_*(\prod_{\Gamma}[\cX_\Gamma^\bot]^{c_\Gamma})$ is
the $\pi_*$-pushforward of a sum of tautological generators supported on
non-trivial boundary strata and we have already shown before that they belong
to $R_{fg}^\bullet(\overline{B})$. 
\par
Since we have shown that $R_{fg}^\bullet(\overline{B})$ is a subring of the
tautological ring closed under clutching and $\pi$-pushforward, it has to be
the same as the tautological ring by minimality.
\par
We finally show the last statement of the theorem, namely  that the $\fraki_{\Gamma *}$
of the level-wise tautological classes $\xi_\Gamma^{[i]}$ and the $\kappa$-classes are
tautological.  For the $\xi$-classes, it is enough to notice that by
Proposition~\ref{prop:Adrienrel} the class~$\xi_{B_\Gamma^{[i]}}$ can be
expressed as a linear combination of a $\psi$-class and boundary classes, so it is
tautological on $B_\Gamma^{[i]}$ by the main statement of the theorem that we just
proved. Since the tautological rings are closed under clutching morphisms, also the
class $\zeta_{\Gamma_*}p_\Gamma^{[i],*}\xi_{B_\Gamma^{[i]}}$ is tautological. Notice that
this is, up to constant, the same as $\fraki_{\Gamma *}(\xi_\Gamma^{[i]})$.
Finally, the $\kappa$-classes are tautological since we have previously shown that
they belong to $R_{fg}^\bullet(\overline{B})$, which we have proven to be the same
as the tautological ring.
\end{proof}

\section{The Chern classes of the logarithmic cotangent bundle}
\label{sec:Chern}

In this section we relate the logarithmic cotangent bundle to
bundles whose Chern classes can be expressed in standard generators.
We will first prove in Theorem~\ref{thm:coker}, a restatement of
Theorem~\ref{intro:Euler}.
We will then complete the proofs of the remaining main theorems of the
introduction, Theorem~\ref{intro:Chern} and Theorem~\ref{intro:ECformula}.
\par
The first step is a direct consequence of the Euler
sequence~\eqref{eq:EulerExt}.
\par
\begin{cor} \label{cor:cKclasses}
The Chern character and the Chern polynomial of the kernel~$\cK$ of
the Euler sequence are  given by
\[\ch(\cK)\=Ne^{\xi}-1 \quad\text{ and } \quad
\c(\cK) \= \sum_{i=0}^{N-1}\binom{N}{i}\xi^i\,.\]
\end{cor}
\begin{proof}
The result follows from the properties of the Chern character and the
Chern polynomial, together with the fact that all higher Chern classes
of the Deligne extension~$\Hrelbar$ vanish. Indeed the Chern classes of a
logarithmic sheaf are given in terms of symmetric polynomials of residues
of the logarithmic connection (see \cite[B3]{esnaultvielog}) and the
Deligne extension is defined such that all these terms are zero,
since the residues are given by nilpotent matrices. (See also
the discussion around \cite[Theorem~17.5.21]{acgh2}.)
\end{proof}
\par
The second step relates the kernel of the Euler sequence to the
vector bundle we are actually interested in. We will use the abbreviations
\be \label{eq:ELgeneral}
\cE_B \= \Omega^1_{\overline{B}}(\log D) \quad \text{and} \quad
\cL_B \= \cO_{\overline{B}}\Bigl( \sum_{\Gamma \in \twolev} \ell_\Gamma D_\Gamma\Bigr)
\ee
throughout in the sequel. 
\par
\begin{theorem}	\label{thm:coker}
There is a short exact sequence of quasi-coherent $\cO_{\overline{B}}$-modules    
\begin{equation} \label{eq:maineq}
0\longrightarrow \cE_B \otimes \cL_B^{-1} \to \cK \to \cC\longrightarrow 0
\end{equation}
where $\cC \= \bigoplus_{\Gamma \in \twolev} \cC_\Gamma$ is a coherent sheaf
supported on the non-horizontal boundary divisors, whose precise form
is given in Lemma~\ref{le:defcalC} below.
\end{theorem}
\par
\begin{proof} %
We start analyzing the injection claimed in~\eqref{eq:maineq}. As in \autoref{sec:eulerseq}, all local calculations happen on the finite covering
charts of~$\bP\LMS$. At a generic
point of a divisor~$D_\Gamma$ the vector bundle~$\cE_B \otimes \cL_B^{-1}$
is generated (using the notation of Case~2 of Section~\ref{sec:bdperiod})
by $\langle t^\ell d\widetilde{c}_2^{[0]}, \ldots,
t^\ell d\widetilde{c}_{N_0}^{[0]},\, t^\ell dt/t, \, t^\ell d\widetilde{c}_2^{[-1]},
\ldots, t^\ell d\widetilde{c}_{N_1}^{[-1]} \rangle$. It is hence obviously
a subbundle of the kernel~$\cK$ as given in~\eqref{eq:cK2lev}. Similarly,
at the intersection point of $L$ divisors different from $D_{\text{h}}$,
the vector bundle~$\cE_B \otimes \cL_B^{-1}$ is generated by the elements
$\prodtL d\widetilde{c}_j^{[-i]}$ and $\prodtL dt_i/t_i$ for $j=2,\ldots,N_i$
and for $i=0,\cdots,L$, where recall that $\prodtL=\prod_{i=1}^L t_i^{\ell_i}$
was introduced in \eqref{eq:prodtnotation}.
This is obviously a subbundle of~$\cK$ as given in proof of
Theorem~\ref{thm:EulerDE}. In the presence of a horizontal edge, this
argument still works, see the form of the cokernel in
Case~1 and Case~3 above. The precise form of~$\cC$ is isolated
in several lemmas  below.
\end{proof}
\par
To start with the computation of~$\cC$, we will also need an infinitesimal
thickening the of the boundary divisor~$D_\Gamma$, namely we define
$D_{\Gamma,\bullet}$ to be its $\ell_\Gamma$-th thickening, the non-reduced
substack of $\LMS$ defined by the ideal $\cI_{D_\Gamma}^{\ell_\Gamma}$. We will
factor the above inclusion using the notation
\[\fraki_\Gamma = \fraki_{\Gamma,\bullet} \circ j_{\Gamma,\bullet} \colon
D_\Gamma \, \overset{j_{\Gamma,\bullet}}{\hookrightarrow} \, D_{\Gamma, \bullet} \,
\overset{\fraki_{\Gamma,\bullet}}{\hookrightarrow}\, \overline{B}\,. \]
\par
We need three more bundles. First, we recall from~\eqref{eq:defLtop} the
definition of the line bundle~$\cL_\Gamma^\top$ and we define $\cL_{\Gamma,\bullet}^\top
= (j_{\Gamma,\bullet})_*\cL_\Gamma^\top$. Second we need the analog
of $\cE_{B}$, but as a bundle on~$D_\Gamma$. Since the projections are defined
only on $D^s_{\Gamma}$ rather than on $D_{\Gamma}$ we cannot define this bundle
as a $p^\top$-pullback, but we need to define it by local generators.
That is, we define $\cE_\Gamma^\top$ to be the vector bundle of rank
$N_\Gamma^\top-1$ on~$D_\Gamma$ with generators $dc_j^{[0]}$ as
$\cO_{D_{\Gamma}}$-module at a generic point of~$\Gamma$ with the usual
coordinates from~\eqref{eq:perid2lev}.
At a point where the top level degenerates, into say $k$ levels, it
is generated as $\cO_{D_{\Gamma}}$-module by the differentials 
$dc_j^{[-i]}$ of level-wise periods and by $dt_i/t_i$ for $i=0,\ldots,k-1$.
Third, we define $\cE_{\Gamma,\bullet}^\top=(j_{\Gamma,\bullet})_*(\cE_\Gamma^\top)$.
\par
\begin{lemma} \label{le:EvsEbul} There is an equality of Chern characters
\[
\ch\Bigl((\fraki_{\Gamma,\bullet})_*(\cE_{\Gamma,\bullet}^\top \otimes
(\cL_{\Gamma,\bullet}^\top)^{-1}) \Bigr) \= \ch \Bigl((\fraki_\Gamma)_*
\bigl(\bigoplus_{j=0}^{\ell_\Gamma-1} \cN_{\Gamma}^{\otimes -j} \otimes
\cE_{\Gamma}^\top \otimes (\cL_{\Gamma}^\top)^{-1}\bigl)\Bigr)\,.
\]
\end{lemma}
\par
\begin{proof} If $\cF_\Gamma$ is a vector bundle on $D_\Gamma$ and $\cF_{\Gamma,\bullet}=(\fraki_{\Gamma,\bullet})_*(\cF_\Gamma)$ is its push-forward to the $\ell_\Gamma$-thickening, we consider the exact sequences
	\[0\to \cI_{D_{\Gamma}}^{k+1}\cF_{\Gamma,\bullet} \to\cI_{D_{\Gamma}}^{k}
	\cF_{\Gamma,\bullet} \to (j_{\Gamma,\bullet})_* \left(\frac{\cI_{D_{\Gamma}}^{k}}
	{\cI_{D_{\Gamma}}^{k+1}}\otimes_{\cO_D}\cF_\Gamma
	\right)\to 0,\quad k=0,\dots,\ell_\Gamma-1.  \]
Notice that $\cI^{\ell_\Gamma}\cF_{\Gamma,\bullet}=0$.
\par
We specialize to $\cF_\Gamma=\cE_{\Gamma}^\top \otimes (\cL_{\Gamma}^\top)^{-1}$
and compute the Chern character of its push-forward to the thickening via the
previous sequences. The statement then follows from the identification
${\cI_{D_{\Gamma}}^{k}} / {\cI_{D_{\Gamma}}^{k+1}} = \cN_\Gamma^{\otimes -k}$ and from
the fact that $(\fraki_{\Gamma,\bullet})_*$ is exact, since $\fraki_{\Gamma,\bullet}$
a closed embedding. 
\end{proof}	
\par
The cokernel of~\eqref{eq:maineq} can be described using the bundles
we just introduced.
\par
\begin{lemma} \label{le:defcalC}
The cokernel of~\eqref{eq:maineq} is given by 
\be \label{eq:defcalC}
\cC \= \bigoplus_{\Gamma \in \twolev} \cC_\Gamma \quad \text{where}
\quad \cC_\Gamma \= (\fraki_{\Gamma,\bullet})_* (\cE_{\Gamma,\bullet}^\top
\otimes (\cL_{\Gamma,\bullet}^\top)^{-1})\,.
\ee
\end{lemma}
\par
\begin{proof} Recall that local generators of $\cK$ had been given
in the proof of \autoref{thm:EulerDE}. At a generic point of the boundary
divisor~$D_\Gamma$, there is a map of coherent sheaves $\cK \to
(\fraki_{\Gamma,\bullet})_* (\cE_{\Gamma,\bullet}^\top \otimes
(\cL_{\Gamma,\bullet}^\top)^{-1})$ which is given in terms of the
generators~\eqref{eq:cK2lev}
by $t^\ell dt/t \mapsto 0$, by $t^\ell d\widetilde{c}_j^{[-1]} \mapsto 0$,
and by $d\widetilde{c}_j^{[0]} \mapsto d\widetilde{c}_j^{[0]} \mod t^\ell$
for all~$j$. The kernel of this map is obviously $\cE_B \otimes \cL_B^{-1}$.
\par
In a neighborhood~$U$ of  the intersection of~$L$ boundary divisors
$D_{\Gamma_i}$, labeled so that $\Gamma_i$ is the $i$-th undegeneration, we recall the shorthand notation $t_{\lceil s \rceil} = \prod_{i=1}^{s} t_i^{\ell_i}$
and we assign 
for every level $-i\in\{0,\dots,L\}$ 
\ba
t_{\lceil i \rceil} dt_{s}/t_{s} &\mapsto t_{\lceil i \rceil} dt_{s}/t_{s}
&\mod t_{i+1}^{\ell_{i+1}},\quad \in \cC_{\Gamma_{i+1}}\\
t_{\lceil i \rceil} d\widetilde{c}_j^{[-s]} &\mapsto t_{\lceil i \rceil} d\widetilde{c}_j^{[-s]}
&\mod t_{i+1}^{\ell_{i+1}},\quad \in \cC_{\Gamma_{i+1}}\\
t_{\lceil i \rceil} dq_k^{[-s]}/q_k^{[-s]}&\mapsto t_{\lceil i \rceil}
dq_k^{[-s]}/q_k^{[-s]}
&\mod t_{i+1}^{\ell_{i+1}},\quad \in \cC_{\Gamma_{i+1}}
\ea
for all $s=0,\dots,i$ and all $j$ and $k$. Again, this map is designed
so that the kernel is $\cE_B \otimes \cL_B^{-1}|_U$. A local computation
of transition functions shows that these maps glue together.
\end{proof}
\par
\begin{proof} [The proof of Theorem~\ref{thm:coker}]
is completed by the two preceding lemmas.
\end{proof}
\par
\begin{prop} \label{prop:prerecursion}
The Chern character of the twisted logarithmic cotangent bundle~$\cE_B \otimes \cL_B^{-1}$ can be
expressed in terms of the twisted logarithmic cotangent bundles of the top levels of non-horizontal divisors as
\bas
\ch(\cE_B \otimes \cL_B^{-1}) \=  N e^{\xi} -1  \,-\,  \sum_{\Gamma \in \twolev}
{\fraki_\Gamma}_* \left(
\ch(\cE_{\Gamma}^\top) \cdot \ch(\cL_{\Gamma}^\top)^{-1}
\cdot  \frac{(1-e^{-\ell_\Gamma\c_1(\cN_{\Gamma})})}{\c_1(\cN_{\Gamma})}\right)\,.
\eas
\end{prop}
\par
\begin{proof} First, by Corollary~\ref{cor:cKclasses} we have
$\ch(\cK)=N e^\xi-1$. Second, from the sequence~\eqref{eq:maineq} we get
\begin{equation}\label{eq:chpartial}
\ch(\cE_B\otimes \cL_B^{-1}) \=\ch(\cK) -\ch(\cC).
\end{equation}
From the additivity of the Chern character we get $\ch(\cC_\Gamma) = \oplus_{\Gamma \in \twolev}\ch( \cC_\Gamma)$. We  now aim to apply \autoref{le:EvsEbul} and the Grothendieck-Riemann-Roch Theorem~\eqref{GRR} to the
map $f=\fraki_\Gamma$, a smooth embedding. The contribution of the Todd classes
simplifies, since the normal bundle exact sequence
\[0\to \cT_{D_\Gamma} \to \fraki_\Gamma^*\cT_{\overline{B}} \to \cN_{\Gamma} \to 0\]
implies $\td(T_{D_\Gamma})\cdot\td(\cN_{\Gamma}) = \td(\fraki_\Gamma^*\cT_{\overline{B}})
= \fraki_\Gamma^*\td(\cT_{\overline{B}})$. If $\cF_{\Gamma}$ is a vector bundle on $D_\Gamma$, we can thus  simplify \eqref{GRR} and get
\bas
\ch(\fraki_{\Gamma,*} \cF_\Gamma) &\= \fraki_{\Gamma,*}(\ch(\cF_\Gamma)\cdot\td(\cT_{D_\Gamma})) \cdot
\td(\cT_{\overline{B}})^{-1}
\= \fraki_{\Gamma,*}(\ch(\cC)\cdot\td(\cT_{D_\Gamma})\cdot \fraki_\Gamma^*\td(\cT_{\overline{B}})^{-1})\\
&\= \fraki_{\Gamma,*}(\ch(\cF_\Gamma)\cdot\td(\cN_\Gamma)^{-1})\,.
\eas
Using the previous remark and Lemma~\ref{le:EvsEbul} we get
\begin{equation*}
\begin{split}
\ch(\cC_\Gamma) &\= (\fraki_\Gamma)_* \Bigl(\sum_{j=0}^{\ell_\Gamma-1}
\ch(\cE_{\Gamma}^\top) \cdot \ch(\cL_{\Gamma}^\top)^{-1}
\cdot \ch(\cN_{\Gamma})^{-j} \td\left([\cN_{\Gamma}]\right)^{-1}
\Bigr)\,.\\
&\= 
\sum_{j=0}^{\ell_\Gamma-1}{\fraki_\Gamma}_* \left( \ch(\cE_{\Gamma}^\top)
\cdot \ch(\cL_{\Gamma}^\top)^{-1} %
\cdot \frac{e^{-j \c_1(\cN_{\Gamma}) }(1-e^{-\c_1(\cN_{\Gamma})})}{\c_1(\cN_{\Gamma})}\right).
\end{split}
\end{equation*}
Canceling terms in the telescoping sum and substituting back the previous expression in \eqref{eq:chpartial} gives the proposition.
\end{proof}	
\par
From this proposition we get some concrete expansions.
\par
\begin{proof}[Proof of Theorem~\ref{thm:c1cor}]
Since the first Chern character is the same as the first Chern class, by
extracting the first degree parts from the expression given
in \autoref{prop:prerecursion} we compute the left hand side to be
\[\ch_1(\cE_B \otimes \cL_B^{-1}) \=\c_1(\cE_B)+(N-1)\sum_{\Gamma\in \twolev} \ell_{\Gamma}[D_\Gamma]\]
and the right hand side to be
\bes
N\xi-\sum_{\Gamma\in \twolev} \ell_{\Gamma}\fraki_{\Gamma,*}((N_\Gamma^\top-1)
[1_{D_\Gamma}])\=N\xi-\sum_{\Gamma\in \twolev} \ell_{\Gamma}(N_\Gamma^\top-1)[D_\Gamma]
\,.\ees
By comparing the two expressions, we get the claim.
\end{proof}
\par
In order to translate Proposition~\ref{prop:prerecursion} into a formula
that can be recursively evaluated, we compare the bundle $\cE_\Gamma^\top$ to
the analogous object 
\bes%
\cE_{B_\Gamma^\top} \= \Omega^1_{B_\Gamma^\top}(\log D_{B_\Gamma^\top}) %
\ees
on the top level of the divisor $D_\Gamma$ for $\Gamma \in \twolev$, where
$D_{B_\Gamma^\top}$ is the total boundary of the generalized stratum
$B_\Gamma^\top$, including the horizontal divisor.
\par
\begin{lemma} \label{le:pandcpull}
  We have
\be
p_\Gamma^{\top,*}\, \cE_{B_\Gamma^\top}  \=  c_\Gamma^* \, \cE_\Gamma^\top \,.
\ee
\end{lemma}
\par
\begin{proof}
The statement can be checked on the local generators. Indeed recall that
the generators of $\cE_{\Gamma}^\top$ as introduced before Lemma~\ref{le:EvsEbul}
are $dc_j^{[0]}$ at a generic point of~$D_\Gamma$, and  $dc_j^{[-i]}$
and $dt_i/t_i$ for $i=0,\ldots,k-1$. Note that even though the map $c_\Gamma$
is branched at the preimage of $\{t_i=0\}$, say given by ~$\{\tilde{t}_i=0\}$,
the pullback of the standard generators $dt_i/t_i$ of the log cotangent
bundle are  proportional to the standard generators $d\tilde{t}_i/\tilde{t}_i$.
We can apply the same argument for the finite degree map $p^\top \times p^\bot$, and check that the pull-back of the local generators of $\cE_{B_\Gamma^\top}$
coincide with the previous ones.
\end{proof}
\par
For the inductive proof we introduce the following shorthand notation. Let 
\bes
P_{B} \=\ch(\cE_B) \prod_{\Gamma \in \twolev} e^{-\ell_\Gamma[D_\Gamma]}
\quad \text{and} \quad 
P_{B_\Gamma^\top} \=\ch(\cE_{B_\Gamma^\top}) \prod_{\Delta \in \LG_1(B_\Gamma^\top)} e^{-\ell_\Delta [D_\Delta]}
\ees
be the Chern characters of the logarithmic cotangent bundles twisted by
a boundary contribution and let
\be P_\Gamma^\top \= \ch(\cE_{\Gamma}^\top) \cdot \ch(\cL_{\Gamma}^\top)^{-1}
\= \ch(\cE_{\Gamma}^\top) \prod_{\Gamma \overset{[0]}{\rightsquigarrow} 
  \wh{\Delta}}  e^{-\ell_{\Delta,1} [D_\Delta]}\,.
\ee
In these terms, Proposition~\ref{prop:prerecursion} reads
\be \label{eq:PBbase}
P_B \= (Ne^\xi -1)  \,-\,  \sum_{\Gamma \in \twolev} {\fraki_\Gamma}_*
\left( \ell_\Gamma  P_\Gamma^\top \td(\cN_\Gamma^{\otimes \ell_\Gamma})^{-1} \right)\,.
\ee
\par
We set $\delta_{L+1}(\Gamma)=\{\cdot\}$, the only graph with one level corresponding to the open stratum $B$, for $\Gamma\in \LG_L(B)$, to make boundary
terms well-defined in the sequel. In particular $N^\top_{\delta_{L+1}}(\Gamma) = N$.
\par
\begin{prop} The twisted Chern character~$P_B$ is given by
\be \label{eq:closedPb} 
P_B \=\sum_{L=0}^{N-1} \sum_{ \Gamma \in \LG_L(B) }\!\!\!\!\!
\bigl(N_{\delta_1(\Gamma)}^T e^{\xi_B}-1\bigr)\,\fraki_{\Gamma *}
\left(\prod_{i=1}^{L} -\ell_{\Gamma,i}
\td\left(\cN_{\Gamma/\delta_{i}^\complement(\Gamma)}^{\otimes \ell_{\Gamma,i}}\right)^{-1}
\right)\,.
\ee
\end{prop}
\par
\begin{proof} We prove the formula by induction. For one-dimensional strata
($N=2$) the formula is~\eqref{eq:PBbase}, since $P_\Gamma^\top$ is trivial then.
We claim that by induction hypothesis
\be \label{eq:PGamma} 
P_\Gamma^\top \= \sum_{L=0}^{N-2}  \!\!\sum_{ \wh\Delta \in \LG_{L+1}(B) \atop
\delta_{L+1}(\wh\Delta) = \Gamma} \!\!\!\!\!\bigl(N_{\delta_1(\Gamma)}^T
e^{\xi_B|_{D_{\Gamma}}}-1\bigr)\, \frakj_{\wh\Delta,\Gamma *}\left(\prod_{i=1}^{L}
- \ell_{\wh\Delta,i} \td\left(\cN_{\wh\Delta/\delta_{i}^\complement(\wh\Delta)}
^{\otimes \ell_{\wh\Delta,i}}\right)^{-1}
\right)
\ee
holds in $\CH^*(D_\Gamma)$. We insert this formula into~\eqref{eq:PBbase}.
Note that for the  degeneration of arbitrary codimension appearing
in~\ref{eq:PGamma} we have
\be \label{eq:jpbnormal}
j_{\wh\Delta, \Gamma}^* c_1(\cN_\Gamma^{\otimes \ell_\Gamma}) \=
c_1\Bigl(\cN_{\wh\Delta/\delta_{L+1}^\complement(\wh\Delta)}^{\otimes \ell_{\wh\Delta,L+1}}\Bigr)
\ee
by splitting the degeneration into codimension one degenerations
and applying successively Corollary~\ref{cor:pullbacknormal} in the case
$\delta_{L+1}(\wh\Delta) = \Gamma$. An application of the push-pull formula
now gives the expression in the proposition.
\par
To prove the claim, note that the induction hypothesis directly implies
that
\be \label{eq:PBGamma} 
P_{B_\Gamma^\top} \= \sum_{L=0}^{N-2} \sum_{ \Delta \in \LG_L(B_\Gamma^\top) }
\!\!\! \bigl(N_{\delta_1(\Delta)}^T e^{\xi_{B_\Gamma^\top}}-1\bigr)\,\fraki_{\Delta *}
\left(\prod_{i=1}^{L}  - \ell_{\Delta,i} \td\left(
\cN_{\Delta/\delta_{i}^\complement(\Delta)}^{\otimes \ell_{\Delta,i}}\right)^{-1}\right)
\ee
in $\CH^*(B_\Gamma^\top)$.
We now pull back this equation and our claimed equation to $D_\Gamma^s$
and compare. Agreement in $D_\Gamma^s$ implies the claim, since we are
working with rational Chow groups throughout. The agreement follows from
the comparison of the normal bundles in the argument of the Todd classes,
which in turn is a consequence of the comparison results
\autoref{prop:xiasppull} and \eqref{eq:calLpullbackgeneral}.
\end{proof}
\par
\begin{cor} \label{cor:chprefinal}
The Chern character of the logarithmic cotangent bundle is 
\bas 
\ch(\cE_{B}) &\= \sum_{L=0}^{N-1}  \sum_{ \Gamma \in \LG_L(B) }
\left(N_{\delta_1(\Gamma)}^T e^{\xi_B}-1\right)\fraki_{\Gamma *} \left(e^{\cL_{\Gamma}}
\prod_{i=1}^{L} - \ell_{\Gamma,i}\td \left(\cN_{\Gamma/\delta_{i}^\complement(\Gamma)}^{\otimes -\ell_{\Gamma,i)}}\right)^{-1}\right) %
\eas
where $\cL_{\Gamma}=\sum_{i=0}^{-L}\cL_{\Gamma}^{[i]}$. 
\end{cor}
\par
The subsequent simplifications of this formula are based on the following
observation. Suppose that $\Gamma \mapsto a_\Gamma$ is an assignment of a rational
number to every level graph $\Gamma \in \LG_L(B)$ for every~$L$ with
the property that if $L>1$ then
\be \label{eq:aismultipl}
a_\Gamma \= \prod_{i=1}^{L} a_{\delta_i(\Gamma)} \ee
is the product of those numbers over all undegenerations to two-level graphs.
We use the abbreviation $a_{\Gamma,i} = a_{\delta_i(\Gamma)}$. 
\par
\begin{lemma} \label{le:expprototype}
For a collection of $a_\Gamma$ satisfying~\eqref{eq:aismultipl} the identity
\bes
\exp\Biggl(\sum_{\Gamma \in \LG_1(\ol{B})} a_\Gamma [D_\Gamma] \Biggr) \=
1\+\sum_{L=1}^{N-1} \sum_{\Gamma \in \LG_L({\ol B})} a_\Gamma \, \fraki_{\Gamma,*} \Bigl(
\prod_{i=1}^{L} \td\bigl(
\cN^{\otimes -a_{\Gamma,i}}_{\Delta/ \delta_i^\complement(\Delta)})\bigr)^{-1} \Bigr)
\ees
holds in $\CH^*(\ol{B})$.
\end{lemma}
\par
\begin{proof} The proof shows that this equality holds in fact if we
restrict to any subset $S \subset \LG_1(\ol{B})$ on the left hand side
and if we restrict  on the right hand side  to the sum of
those $\Gamma \in \LG_L({\ol B})$ such that all their two-level undegenerations
belong to~$S$. The proof now proceeds by induction over~$|S|$.
\par
For $|S| = 1$ this is the identity $ \exp(a_\Gamma [D_\Gamma])=1+a_\Gamma\fraki_{\Gamma,*}(\td
(\cN^{\otimes -a_\Gamma}_\Gamma)^{-1}) $ that follows from the adjunction formula
$\fraki_\Gamma^* \fraki_{\Gamma,*} \alpha = c_1(\cN_\Gamma) \cdot \alpha$ and the relation between the generating series of the exponential and the Todd class.
\par
For $|S| > 2$ this follows from the uniqueness of the intersection
orders shown in Proposition~\ref{prop:ordering} and induction. We
give details for $|S|=2$, leaving it to the reader to set up the notation
for the general case. Let $\Gamma_k \in \LG_1(\ol{B})$ for $k=1,2$
and abbreviate $D_k = D_{\Gamma_k}$, $\cN_k = c_1(\cN_{\Gamma_k})$,
$\fraki_k = \fraki_{\Gamma_k}$ and $\frakj_{k}
= \frakj_{\Delta,\Gamma_k}$ for any graph~$\Delta$ with $\delta_k(\Delta)
=\Gamma_k$ for $k=1,2$. We denote by $[1,2]$ this set of $3$-level graphs.
Then by~\eqref{eq:jpbnormal}
\bas &\phantom{\=}
\sum_{\Delta \in [1,2]} \fraki_{\Delta,*} \Bigl( c_1\bigl(\cN_{\Delta/ \delta_1^\complement
(\Delta)}\bigr)^{x-1} c_1\bigl(\cN_{\Delta/\delta_2^\complement(\Delta)}\bigr)^{y-1}
\Bigr) \= \sum_{\Delta \in [1,2]} \fraki_{\Delta,*} \bigl(\frakj_{1}^* \cN_1^{x-1}
\frakj_{2}^* \cN_2^{y-1} \bigr) \\
& \= \sum_{\Delta \in [1,2]} \fraki_{1,*} \bigl(\fraki_1^*([D_1])^{x-1}
\frakj_{1,*} \frakj_{2}^* \cN_2^{y-1} \bigr)
\= [D_1]^{x-1} \cdot  \sum_{\Delta \in [1,2]} \fraki_{2,*} \frakj_{2,*}
\frakj_{2}^* \cN_2^{y-1} \\
&\= [D_1]^x \cdot \fraki_{2,*} \cN_2^{y-1} \= [D_1]^x \cdot [D_2]^y\,.
\eas
Taking the generating series over this expression proves the claim.
\end{proof}
\par
\begin{proof}[Proof of Theorem~\ref{intro:Chern}] In order to deduce
this theorem from Corollary~\ref{cor:chprefinal}, we introduce shorthand
notations for the products of inverse Todd classes, namely for any
$\Gamma \in \LG_L(\ol{B})$ we let 
\be \label{eq:defXGi}
X_{\Gamma,i}\=  \td\left(\cN_{\Gamma/\delta_{i}^\complement(\Gamma)}^{\otimes
	-\ell_{\Gamma,i}}\right)^{-1}  \quad \text{and} \quad
X_{\Gamma}\= \prod_{i=1}^{L} X_{\Gamma,i},
\ee
and
\[X_{\Delta \setminus \Gamma } \= \prod_{i\in I^\complement}\td
\left(\cN_{\Gamma/\delta_{i}^\complement(\Gamma)}^{\otimes -\ell_{\Gamma,i}}\right)^{-1}\]
if $\Gamma = \delta_I(\Delta)$ is the undegeneration keeping only the level passages
in~$I$ of~$\Delta$. Now the argument of Lemma~\ref{le:expprototype} with
$\ell_\Gamma$ playing the role of $a_\Gamma$ and with both sides restricted
to degenerations of a fixed $\Gamma \in \LG_L(\ol{B})$ gives
\bes
\exp(\cL_\Gamma) \= \exp\Biggl( \sum_{\Gamma \in\LG_{L+1}^\Gamma(\ol{B})}
\ell_\Gamma[D_\Gamma] \Biggr)
\= 1 \+ \sum_{L'=L+1}^{N-1} \sum_{\Delta \in \LG_{L'}^\Gamma(\ol{B})} \ell_\Delta\,
\frakj_{\Delta,\Gamma,*} (X_{\Delta \setminus \Gamma })\,,
\ees
where $\LG_{L'}^\Gamma(\ol{B})$ are the graphs with $L'$ levels below zero
that are degenerations of~$\Gamma$. We inject this formula into the right
hand side of Corollary~\ref{cor:chprefinal}. Since
\bes
\fraki_{\Gamma,*} \bigl( \frakj_{\Delta,\Gamma,*} (X_{\Delta \setminus \Gamma })
\cdot X_\Gamma) \= \fraki_{\Delta,*} (X_\Delta)
\ees
by the projection formula and equation~\eqref{eq:jpbnormal}, we obtain
\bes
\ch(\cE_B) \= \sum_{L=0}^{N-1} (-1)^L \sum_{ \Gamma \in \LG_L(B) }
\left(N_{\delta_1(\Gamma)}^T e^{\xi_B}-1\right)
\sum_{L'=L}^{N-1} \sum_{\Delta \in \LG_{L'}^\Gamma(\ol{B})} \ell_\Delta\,
\fraki_{\Delta,*} (X_\Delta)\,.
\ees
It remains to sort this expression as sum over $\ell_\Delta \fraki_{\Delta,*}
(X_\Delta)$. Since each $\Delta \in \LG_{L'}(B)$ appears  in the expression
of each $\Gamma$ with $\delta_I(\Delta)=\Gamma$, its coefficient in the
final expression of $\ch(\cE_B)$ is  (defining $\min(\{\emptyset\})=L'+1$ )
\bas
\sum_{I \subseteq \{1,\cdots,L'\}} \!\!\!\!(-1)^{|I|} \cdot 
\left(N_{\delta_{\min(I)}(\Delta)}^T e^{\xi_B}-1\right)
& \= e^{\xi_B} \cdot \!\!\!\!\!\!\sum_{I \subseteq \{1,\cdots,L'\}} (-1)^{|I|}
N_{\delta_{\min(I)}(\Delta)}\\
& \=  e^{\xi_B} \cdot \left(N-N_{\delta_{L'}(\Gamma)}^T\right)\,,
\eas
where the disappearance of $(-1)^{|I|+1}$ in the first equality  and the
cancellation in the second equality stem from canceling the contributions
of pairs under the involution $I \mapsto I\cup \{L'\}$ if $L'\not \in I$
and $I \mapsto I\setminus\{L'\}$, if $L'\in I$. 
\end{proof}
\par
\par
\medskip
In preparation for the next theorem we switch to the language of profiles
introduced in Section~\ref{sec:strucbd} and  recall that the notation depends
on the choice of the numbering of $\twolev = \{\Gamma_1, \ldots, \Gamma_M\}$.
We claim that Theorem~\ref{intro:Chern} can equivalently be restated as
\ba \label{eq:chviaprofiles}
\ch(\cE_{B}) 
&\= e^{\xi_B} \cdot \sum_{L=0}^{N-1}  \sum_{[j_1,\ldots,j_L] \in \mathscr{P}_L }
\bigl(N-N_{{j_L}}^T\bigr)
\prod_{i=1}^{L} \left(e^{\ell_{{j_i}} [D_{{j_i}}]} -1 \right) \,. 
\ea
To see the equivalence, it suffices to expand the product~\eqref{eq:chviaprofiles}
and to use Proposition~\ref{prop:ordering} about the uniqueness of the
order of letters in a profile. Note that we cannot replace $\mathscr{P}_L$
by $\LG_L(B)$ in~\eqref{eq:chviaprofiles}, as this would give wrong multiplicities.
\par
We abbreviate the difference of dimensions $r_{\Gamma,i} =
N-N_{\delta_i(\Gamma)}^\top$ and write $r_\Gamma = \prod_{i=1}^L r_{\Gamma,i}$.
It is useful to remember that $r_{\Gamma,i} = \sum_{j= i+1}^L N^{[-j]} =
\sum_{j=i+1}^L (d^{[-j]} + 1)$ is the sum of the unprojectivized dimensions of
the lower levels. If we work with profiles and the elements of $\twolev$
are numbered, we write $r_{j} = r_{\Gamma_j}$
and $\ell_i = \ell_{\Gamma_i}$. We can now state an additive and a multiplicative
decomposition of the Chern polynomial.
\par
\begin{theorem} \label{thm:cpoly}
The Chern polynomial of the logarithmic cotangent bundle is 
\ba \label{eq:chpoly}
&c(\cE_B) \= \prod_{L=0}^{N-1} \prod_{[j_1,\ldots,j_L] \in \mathscr{P}_L }
\prod_{I \subseteq \{1,\ldots,L\}} \bigl(1 + \xi + \sum_{i \in I} \ell_{j_i} [D_{j_i}]\big)
^{(-1)^{|I^\complement|}\cdot r_{j_L}} \\
&\= \sum_{L=0}^{N-1} \sum_{\Gamma \in \LG_L(B)} \!\!\!\!\ell_\Gamma \fraki_{\Gamma,*}
\Biggl( \sum_{\bk} \, \binom{N-\sum_{i=1}^L k_i}{k_0} \xi^{k_0} \cdot 
\prod_{i=1}^L \binom{r_{\Gamma,i} - \sum_{j>i}^L k_j}{k_i}
(\ell_{i}\nu_{\Gamma,i})^{k_i-1}
\Biggr)\,, 
\ea
where $\bk = (k_0,k_1,\ldots,k_L)$ is a tuple with $k_0 \geq 0$ and
$k_i \geq 1$ for $i=1,\ldots,L$ and where $\nu_{\Gamma,i}
=  c_1(\cN_{\Gamma/ \delta_i^\complement(\Gamma)})$.
For $L=0$ the exponent $r_{j_L}$ is to be interpreted as~$N$.
\end{theorem}
\par
\begin{proof} We first deduce the first line from~\eqref{eq:chviaprofiles}.
We compute the degree-$d$-part of its interior product to be 
\bes
\Bigl[e^{\xi_B} \prod_{i=1}^{L} \left(e^{\ell_{{j_i}} [D_{{j_i}}]} -1 \right) \Bigr]_d
\= \frac{1}{(d-1)!\cdot d} \sum_{I \subseteq \{1,\ldots,L\}} (-1)^{|I^\complement|}
\bigl(\xi + \sum_{i \in I} [D_{j_i}]\bigr)^d\,.
\ees
On the other hand, recall from \cite[586]{acgh2} that the Chern polynomial
is given in terms of the graded pieces of the Chern character by
\[c(\cE_B) \=  \exp\Bigl(\sum_{d \geq 1}(-1)^{d-1} (d-1)!\ch_d(\cE_B)\Bigr)\,.\]
Using the generating series of the logarithmic function, we then obtain the
first line of the statement by combining the previous two expressions.
\par
In order to pass to the second line, we first show that the first
line formally fits with Lemma~\ref{le:prodtosum}. We want to replace
the two exterior products over all~$L$ and profiles~$\mathscr{P}_L$
by all subsets of the integer interval $[[1,\ldots,M]]$ without altering
the value of the product. For this purpose we claim that for
each element of $\mathscr{P}_L$ the interior product
\bes
P \= \prod_{I \subseteq \{1,\ldots,L\}}
\bigl(1 + \xi + \sum_{i \in I} \ell_{j_i} [D_{j_i}]\big) ^{(-1)^{|I^\complement|}\cdot r_{j_L}}
\ees
considered as an element in the polynomial ring is in
$1 + D_1\cdots D_L\cdot \bQ[\xi,D_1,\ldots,D_L]$. This claim implies that 
the additional products give zero in the Chow ring and considering the
profiles as subsets of  $[[1,\ldots,M]]$ rather than as ordered tuples is
no loss of information thanks to Proposition~\ref{prop:ordering}.
To justify the claim we may assume $r_{j_L}=1$, since the claim persists
when raising to an integral power. For $L=1$ the claim is obvious
and for the inductive step one replaces~$\xi$ successively by
$\xi + \ell_{j_k}D_k$ to see that $P-1$ is divisible by $D_i$
for all $i \neq k$.
\par
Now we are in the situation to apply the image of the formula
of Lemma~\ref{le:prodtosum} in the Chow ring. To match the second
line of the lemma and the theorem we define for tuple 
$\bk = (k_0,k_1,\ldots,k_M)$ as in the lemma the integer~$L$ to
be the number of entries $k_{i}$ that are positive. Consider a summand
$\bk = (k_0,k_1,\ldots,k_M)$ in the second line of the line, and say that
$i_1,\ldots,i_L$ are those indices where the entries
$k_{i_j}$ are positive. Then the contribution of this summand to
the second line of~\eqref{eq:prodtosum} equals the contributions of
the (possibly empty) set of level graphs in $D_{i_1} \cap \cdots \cap
D_{i_L}$ to the second line of~\eqref{eq:chpoly}.
\end{proof}
\par
\begin{lemma} \label{le:prodtosum}
In the polynomial ring $\bQ[\xi,D_1,\ldots,D_M]$ the identity
\ba \label{eq:prodtosum}
&\phantom{\=}
 \prod_{[j_1,\ldots,j_L]  \subseteq \{1,\ldots,M\}}
\prod_{I \subseteq \{1,\ldots,L\}}
\bigl(1 + \xi + \sum_{i \in I} \ell_{j_i} D_{j_i} \big)
^{(-1)^{|I^\complement|}\cdot N_{j_L}} \\ 
&\= \sum_{\bk}
\binom{N-\sum_{i=1}^M k_i}{k_0} \,\xi^{k_0} \cdot 
\prod_{i=1}^M \binom{\sum_{j \geq i} N^{[-j]}
- \sum_{j>i}^M k_j}{k_i} (\ell_i D_i)^{k_i}
\ea
holds, where $\bk = (k_0,k_1,\ldots,k_M)$ is a tuple of non-negative
integers,  and where $N_s := \sum_{j=s+1}^M N^{[-j]}$
and $N = N_\emptyset = \sum_{j=0}^M N^{[-j]}$.
\end{lemma}
\par
\begin{proof} We proceed by induction, $M=0$ is the binomial expansion.
The effect of the passage from from $M-1$ to~$M$ is given on the left hand side
by replacing $N^{[-(M-1)]}$ with $N^{[-(M-1)]} + N^{[-M]}$  in all those factors
where $j_L < M$, and by multiplying by the factors where $j_L = M$, i.e.
by multiplication with 
\bas
&\phantom{\=}
 \prod_{[j_1,\ldots,j_{L-1}]  \subseteq \{1,\ldots,M-1\} \atop
I \subseteq \{1,\ldots,L-1\}}
\bigl(1 + \xi + D_M + \sum_{i \in I} \ell_{j_i} D_{j_i} \big)
^{(-1)^{|I^\complement|}\cdot N^{[M]}} \\
&\= \sum_{k_M \geq 0} \binom{r_M}{k_M} \,D^{k_M} \cdot  \!\!\!\!
 \prod_{[j_1,\ldots,j_{L-1}]  \subseteq \{1,\ldots,M-1\} \atop
I \subseteq \{1,\ldots,L-1\}}
\!\!\! \bigl(1 + \xi +  \sum_{i \in I} \ell_{j_i} D_{j_i} \big)
^{(-1)^{|I^\complement|}\cdot (N^{[M]}-k_M)}
\eas
Applying the induction hypothesis with $N^{[-(M-1)]}$ replaced by
$N^{[-(M-1)]} + N^{[-M]} -k_M$  gives the claim.
\end{proof}
\par
The following step concludes the proof of all main theorems.
\par
\begin{lemma}\label{lem:evaluation}
Suppose that $\alpha_\Gamma \in \CH_0(D_\Gamma)$ is a top degree class
and that $c_\Gamma^* \alpha_\Gamma = \prod_{i=0}^{-L(\Gamma)} p_\Gamma^{[i],*} \alpha_i$
for some $\alpha_i$. Then
\bes
\int_{D_\Gamma}\alpha_\Gamma \= \frac{K_\Gamma}{|\Aut(\Gamma)|\ell_\Gamma}
\prod_{i=0}^{-L(\Gamma)} \int_{B_\Gamma^{[i]}} \alpha_i\,.
\ees
\end{lemma}
\begin{proof} We have 
\bes
\int_{D_\Gamma} \alpha_\Gamma \= \frac{1}{\deg(c_{\Gamma})} \int_{D_\Gamma^s}
c_\Gamma^*(\alpha_\Gamma) \= \frac{\deg(p_\Gamma)}{\deg(c_{\Gamma})}
\prod_{i=0}^{-L(\Gamma)} \int_{B_\Gamma^{[i]}} \alpha_i.
\ees
and the claim follows from by Lemma~\ref{le:degreeratio}.
\end{proof}
\begin{proof}[Proof of Theoreom~\ref{intro:ECformula}] 
By Proposition~\ref{prop:chiviaTlog}, it is enough to compute the top Chern
class $c_d(\cE_B)$, where $d=\dim(B) = N-1$. We investigate
for each~$L$ and each~$\Gamma \in \LG_L(B)$ the contribution of the second
line of~\eqref{eq:chpoly} in Theorem~\ref{thm:cpoly} to $c_d(\cE_B)$.
It suffices then to show that the expression
inside the $\fraki_{\Gamma,*}$ is equal to $N_\Gamma^\top \prod_{i=0}^{L-1}
(\xi_\Gamma^{[i]})^{d_{\Gamma}^{[i]}}$. Note that by \autoref{prop:generalnormalbundle}
the first chern class of the normal bundle $\cN_{\Gamma/\delta_{i}^\complement(\Gamma)}$
is supported on the levels~$-i+1$ and $-i$ of~$\Gamma$. Considering the bottom
level we deduce that if the summand $\bk$ contributes non-trivially
to the top Chern class~$c_d$, then we must have $k_L \geq d^{[L]}+1$ so that
the $\nu_{\Gamma,L}$-power is large enough for its binomial expansion to
contain a top $\xi$-power for the bottom level. On the other hand, 
for the binomial coefficient in front of it to be non-zero we need
$r_{\Gamma,L} \geq k_L$, which is equivalent to $k_L \leq d^{[L]}+1$. So
in fact $k_L = d^{[L]}+1$, the binomial coefficient is one, and we have
to select from the expansion of $\nu_{\Gamma_L}^{d^{[L]}}$ the term that does
not contribute to level~$-i-1$. Since the top entry
of the binomial coefficient is $r_{\Gamma,i} - \sum_{j>i}^L k_j
= 1 + N_i + \sum_{j>i}^L (d^{[j]} + 1 -k_j)$ we can inductively repeat this
argument for all levels and deduce  $k_j = d^{[j]} + 1$ for all $j \geq 1$
and $k_0 = d^{[0]}$. The only non-trivial factor is now  $N_\Gamma^\top$ that
stems from the first binomial coefficient in the second line
of~\eqref{eq:chpoly}. The final shape of the statement follows then directly
from \autoref{prop:xiasppull} and \autoref{lem:evaluation}, after
noticing that the $\ell_{\Gamma}$ coefficients cancel.
\end{proof}

\section{Examples: Geometry and values} \label{sec:examples}

In this section we explain how to evaluate top degree classes, we
provide examples illustrating the geometry at infinity of the
compactification $\LMS$ and examples of our formulas for the normal
bundles, the Chern polynomials and the Euler characteristic. 

\subsection{Evaluation of top $\xi$-powers} \label{sec:topxi}

\par
\medskip
First of all we explain how to evaluate the expression in
Theoreom~\ref{intro:ECformula}, see \cite{CoMoZadiffstrata} for many
algorithmic details. We only need to explain how to evaluate
$\int_{\ol{B}} \xi^{d}$, i.e. top powers of $\xi$ on generalized strata. 
\par
Suppose that $B = \bP\Omega\cM_{g,1}(2g-2)$ is a stratum parametrizing
connected surfaces with a single zero. Then the generating series of
top $\xi$-powers is given by a simple power series inversion that 
arises in the computation of Masur-Veech volumes, see \cite{SauvagetMinimal}
and \cite[Theorem~3.1]{CMSZ}, and Table~\ref{cap:xitop} for some values.
\par
Suppose that $B = \bP\Omega\cM_{g,n}(\mu)$ is a stratum parametrizing
connected surfaces of holomorphic type, i.e., with all $m_i \geq 0$
and with $n \geq 2$. Then $\int_B \xi^d = 0$ by
\cite[Proposition~3.3]{SauvagetMinimal}.
\par
It remains to explain how to  evaluate $\int_{\ol{B}} \xi^{d}$ top powers of $\xi$ on meromorphic generalized strata. First of all, we write $\xi$ with the help of Proposition~\ref{prop:Adrienrel} as a $\psi$-class and boundary strata.
The product of such objects, which are standard generators as in~\ref{eq:addgenR} of
the tautological ring, can be rewritten as a sum of standard additive generators via the algorithm explained in the proof of \autoref{thm:addgen}, more specifically in the part in which we show that $R_{fg}^\bullet(\ol{B}) = R^\bullet(\ol{B})$ is a ring. Now that we have rewritten $\xi^{d}$ in terms of standard additive generators, by \autoref{lem:evaluation}, it only remains to explain how to evaluate a top-dimensional standard generator, i.e. the top power of a $\psi$-class, on
a generalized stratum~$\ol{B}$. Since $\psi$-classes are pulled back from $ \barmoduli[g,n]$, we can use a push-pull argument and express
\[\int_{\ol{B}} \psi_i^d=\int_{ \barmoduli[g,n]} \pi_*([\ol{B}])\psi_i^d\]
where $\pi:\ol{B}\to  \barmoduli[g,n]$ is the forgetful morphism and where we used as always the abuse of notation $\psi_i=\pi^*(\psi_i)$.
\par
If we can express the class $\pi_*([\ol{B}])$ in terms of the standard generators of $\barmoduli[g,n]$, we can use the sage package {\tt admcycles} in order to obtain a number.
\par
If $\ol{B}$ is a stratum parameterizing meromorphic differentials on connected surfaces without residue conditions, the class $\pi_*([\ol{B}])$  was computed in \cite{sauvagetstrata}
and \cite{BHPSS}, and the algorithmic task can be performed again by the sage
package {\tt admcycles}, which implements the algorithm based on
the formula in~\cite{schmittDimTh} and~\cite{BHPSS}.
\par
If the stratum $\ol{B}$ is more general parametrizing differentials on disconnected surfaces and with residue conditions, we first of all use repetitively Proposition~\ref{prop:AdrienR} to write the class of $\ol{B}$ into the associated stratum without residue conditions in terms of additive generators of the stratum with not conditions.
We then reduced to the computation of the class $\pi_*([\ol{B}])$ in the case that $\ol{B}$ has no
more residue conditions, but is potentially disconnected.
If~$\ol{B}$  is  disconnected then actually $\pi_*([\ol{B}])$ is zero. In fact, since we can scale the differentials on the components independently,
the fiber dimension to a product of $\barmoduli[g_i,n_i]$ is positive and by definition of push-forward we get the zero class.
\par
\begin{figure}[h]
	$$ \begin{array}{|c|c|c|c|c|c|c|}
	\hline  &&&&&& \\ [-\halfbls] 
	\mu  & (0) & (2) & (4) & (6) & (0,0,-2) & (2,-2)  \\
	[-\halfbls] &&&&&& \\ 
	\hline &&&&&& \\ [-\halfbls]
	\int_{\ol{B}} \xi^{\dim(B)} %
	& \frac{1}{24} &  -\frac{1}{640} & \frac{-305}{580608},
	& -\frac{87983}{199065600} & 1 &  -\frac{1}{8} \\
	[-\halfbls] &&&&&& \\
	\hline  &&&&&& \\ [-\halfbls] 
	\mu  & (1,1,-2) & (4,-2) & (3,1,-2) & (2,1,-3) & (5,-3) & (8,-2,-2,-2)\\
	[-\halfbls] &&&&&& \\ 
	\hline &&&&&& \\ [-\halfbls]
	\int_{\ol{B}} \xi^{\dim(B)} %
	& 0  & - \frac{-23}{1152},
	& 0 & \frac{5}{8} &  -\frac{21}{20} & -\frac{4527}{32}     \\
	[-\halfbls] &&&&&& \\
	\hline %
	\end{array}
	$$
	\captionof{table}[foo2]{Integrals of top $\xi$-powers for some
		connected strata}
	\label{cap:xitop}
\end{figure}
For the subsequent examples we present some cases where we can directly
evaluate the top $\xi$-power for meromorphic strata in genus~$0$ and
genus~$1$.
\par
\begin{prop}\label{prop:topxieasyexample}
The integrals of the top $\xi$-power are given ($a_i,k \geq 0$)
\bas
&\text{for} \quad B \= \bP\Omega_{0,n+1}(-2-\sum_{i=1}^{n} a_i,a_1,...,a_n) && \text{by} \quad 
\int_B \xi_B^{n-2} \= (-1-\sum_{i=1}^n a_i)^{n-2}\,, \\
&\text{for} \quad B \=\bP\Omega_{1,2}(-k,k) && \text{by} \quad 
\int_B \xi_B \= -\frac{(k-1)(k^2-1)}{24}\,, \\
&\text{for} \quad B \= \bP\Omega_{1,3}(-k-1,1,k)&& \text{by}\quad 
\int_B \xi^2_B\ =\frac{(k^4-1)}{24}\,.
\eas
\end{prop}
\begin{proof}
	The first statement follows easily from Proposition~\ref{prop:Adrienrel}, which in this case implies $\xi=(-1-\sum_{i=1}^{n}a_i)\psi_1$. Indeed there cannot be any divisors which have the pole on lower level. Hence 
	\[\int_B \xi^{n-2}=(-1-\sum_{i=1}^{n}a_i)^{n-2}\int_{\cM_{0,n+1}}\psi_1^{n-2}=(-1-\sum_{i=1}^{n}a_i)^{n-2}.\]
\par 
The second statement follows immediately as the previous one, since again there cannot be divisors where the pole is on lower level. Hence
\bes \int_B\xi
\= -(k-1)\int_{\cM_{1,2}}\pi_*([B])\psi_1=-\frac{(k-1)(k^2-1)}{24}\,,
\ees
where we used \cite[Proposition~3.1]{ChenEED} for the computation of $\pi_*([B])$. 
\par 
For the proof of the last statement, notice that there can only be one non-horizontal divisor $D_3$ which has the pole on lower level (see \autoref{sec:exk1} for the full boundary description). Using Proposition~\ref{prop:Adrienrel} we find then $\xi = -k\psi_1 -D_3$, which yields
\bas
\int_{\overline{B}} \xi^2
&\= - \int_{\overline{B}} \xi (D_3 + k\psi_1) 
\=  \Bigl(- 1/24 + \int_{\overline{B}} k\psi_1(D_3 + k\psi_1)\Bigr) \\
&\=\Bigl(-1/24 + k^2\int_{\cM_{1,3}}\pi_*([B]) \cdot\psi_1^2 \Bigr)=
(k^4 -1)/24
\eas
again by the computation in \cite[Proposition~3.1]{ChenEED} of the class of the class $\pi_*(B)$ of the stratum in $\cM_{1,3}$.
\end{proof}

\subsection{The minimal stratum $\bP\Omega\cM_{2,1}(2)$}

This stratum is small enough so that we can show all the level graphs,
including those with horizontal nodes and their adjacency in
Figure~\ref{cap:H2}.
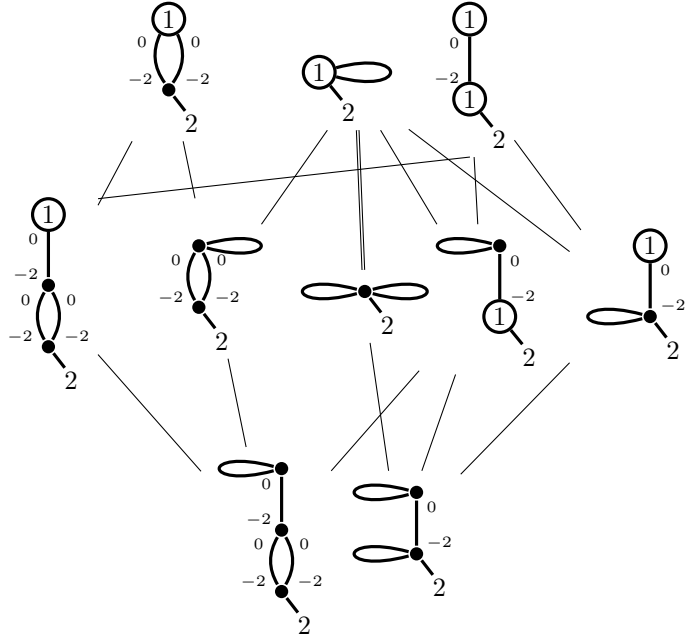
\begin{figure}
\def\eq{=}
\newlength\leveldist
\setlength{\leveldist}{1.3cm}
\newlength\hordist
\setlength{\hordist}{1cm}

\begin{tikzpicture}[
		baseline={([yshift=-.5ex]current bounding box.center)},
		scale=2,very thick,
		node distance=\HoG, 
		bend angle=30,
		every loop/.style={very thick},
     		comp/.style={circle,fill,black,,inner sep=0pt,minimum size=5pt},	
		order top left/.style={pos=\PotlI,left,font=\tiny},
		order top right/.style={pos=\PotrI,right,font=\tiny},		
		bottom right with distance/.style={below right,text height=10pt}]

\begin{scope}[local bounding box = l]
\node[circled number] (T) [] {$1$}; 
\node[comp] (B) [below=of T] {}
	edge [bend left] 
		node [order bottom left] {$-2$} 
		node [order top left] {$0$} (T)
	edge [bend right] 
		node [order bottom right] {$-2$} 
		node [order top right] {$0$} (T);
\node [bottom right with distance] (B-2) at (B.south east) {$2$};
\path (B) edge [shorten >=4pt] (B-2.center);
\end{scope}

\begin{scope}[shift={($(l.center) + (\hordist,0)$)}, local bounding box = h]
\node[circled number] (T) [] {$1$} 
	edge [loop right] (T);
\node [bottom right with distance] (T-2) at (T.south east) {$2$};
\path (T) edge [shorten >=4pt] (T-2.center);
\end{scope}

\begin{scope}[shift={($(l.base) + (2\hordist,0)$)}, local bounding box = ct]
\node[circled number] (T) [] {$1$}; 
\node[circled number] (B) [below=of T] {$1$}
	edge 
		node [order bottom left] {$-2$} 
		node [order top left] {$0$} (T);
\node [bottom right with distance] (B-2) at (B.south east) {$2$};
\path (B) edge [shorten >=4pt] (B-2.center);
\end{scope}

\begin{scope}[shift={($(l.base) + (-.8\hordist,-\leveldist)$)}, local bounding box = ctl]
\node[circled number] (T) [] {$1$}; 
\node[comp] (B1) [below=of T] {}
	edge 
		node [order bottom left] {$-2$} 
		node [order top left] {$0$} (T);
\node[comp] (B2) [below=of B1] {}
	edge [bend left] 
		node [order bottom left] {$-2$} 
		node [order top left] {$0$} (B1)
	edge [bend right] 
		node [order bottom right] {$-2$} 
		node [order top right] {$0$} (B1);
\node [bottom right with distance] (B2-2) at (B2.south east) {$2$};
\path (B2) edge [shorten >=4pt] (B2-2.center);
\end{scope}

\begin{scope}[shift={($(ctl.center) + (\hordist,.35cm)$)}, local bounding box = hl]
\node[comp] (T) {} 
	edge [loop right] (T);
\node[comp] (B) [below=of T] {}
	edge [bend left] 
		node [order bottom left] {$-2$} 
		node [order top left] {$0$} (T)
	edge [bend right] 
		node [order bottom right] {$-2$} 
		node [order top right] {$0$} (T);
\node [bottom right with distance] (B-2) at (B.south east) {$2$};
\path (B) edge [shorten >=4pt] (B-2.center);
\end{scope}

\begin{scope}[shift={($(hl.center) + (\hordist,0)$)}, local bounding box = hh]
\node[comp] (T) {}
	edge [loop right] (T)
	edge [loop left] (T);
\node [bottom right with distance] (T-2) at (T.south east) {$2$};
\path (T) edge [shorten >=4pt] (T-2.center);
\end{scope}

\begin{scope}[shift={($(ctl.center) + (3\hordist,.35cm)$)}, local bounding box = hct]
\node[comp] (T) {}
	edge [loop left] (T);
\node[circled number] (B) [below=of T] {$1$}
	edge 
		node [order bottom right] {$-2$} 
		node [order top right] {$0$} (T);
\node [bottom right with distance] (B-2) at (B.south east) {$2$};
\path (B) edge [shorten >=4pt] (B-2.center);
\end{scope}

\begin{scope}[shift={($(ctl.center) + (4\hordist,.35cm)$)}, local bounding box = cth]
\node[circled number] (T) [] {$1$}; 
\node[comp] (B) [below=of T] {}
	edge 
		node [order bottom right] {$-2$} 
		node [order top right] {$0$} (T)
	edge [loop left] (B);
\node [bottom right with distance] (B-2) at (B.south east) {$2$};
\path (B) edge [shorten >=4pt] (B-2.center);
\end{scope}

\begin{scope}[shift={($(l.base) + (.75\hordist,-2.3\leveldist)$)}, local bounding box = hctl]
\node[comp] (T) {}
	edge [loop left] (T);
\node[comp] (B1) [below=of T] {}
	edge 
		node [order bottom left] {$-2$} 
		node [order top left] {$0$} (T);
\node[comp] (B2) [below=of B1] {}
	edge [bend left] 
		node [order bottom left] {$-2$} 
		node [order top left] {$0$} (B1)
	edge [bend right] 
		node [order bottom right] {$-2$} 
		node [order top right] {$0$} (B1);
\node [bottom right with distance] (B2-2) at (B2.south east) {$2$};
\path (B2) edge [shorten >=4pt] (B2-2.center);
\end{scope}

\begin{scope}[shift={($(hctl.center) + (\hordist,.35cm)$)}, local bounding box = hcth]
\node[comp] (T) {}
	edge [loop left] (T);
\node[comp] (B) [below=of T] {}
	edge 
		node [order bottom right] {$-2$} 
		node [order top right] {$0$} (T)
	edge [loop left] (B);
\node [bottom right with distance] (B-2) at (B.south east) {$2$};
\path (B) edge [shorten >=4pt] (B-2.center);
\end{scope}

\draw[thin] (h) edge [double] (hh);

\draw[thin] (l) edge [shorten >=10pt] (hl) (l) edge (ctl);
\draw[thin] (h) edge (hl) (h) edge (hct) (h) edge (cth);
\draw[thin] (ct.south) edge (ctl.north east) (ct) edge (hct) (ct) edge (cth);

\draw[thin] (hctl) edge (ctl) (hctl) edge (hl) (hctl) edge (hct);
\draw[thin] (hcth) edge (hh) (hcth) edge (hct) (hcth) edge (cth);

\end{tikzpicture}
\caption{Level graphs appearing in the boundary of $\Omega\cM_{2,1}(2)$. Graphs
corresponding to components of the same dimension are in the same row (divisors
in the first row, points in the bottom row). The lines connecting the graphs
symbolize degeneration.} \label{cap:H2}
\end{figure}
The picture shows the dual graphs of stable curves in the boundary
of this stratum, the top level is on top of each graph. The number in
the vertex denotes the genus, a black dot corresponds to genus zero.
The numbers associated to the legs are the orders of zero. In this
stratum, all interior edges have enhancement~$\kappa_e =1$, so the
discussion of prong-matchings is void here. There are only three graphs
without horizontal nodes, in fact $|\LG_1(B)| = 2 $ and $|\LG_2(B)| = 1$,
where $B = \bP\Omega\cM_{2,1}(2)$ as usual. Taking into also account the
entire stratum and the stack structure of the banana graphs, and using
the values of top $\xi$-powers from Section~\ref{sec:topxi}, we get
\bes
(-1)^3 \cdot \chi(B) \= 4 \cdot \frac{-1}{640} + 0 + 2 \cdot \frac{1}{24}
\cdot \frac{-1}{8} + 2 \cdot \frac{1}{2} \cdot \frac{1}{24} \cdot 1 \cdot 1
\= \frac{1}{40}
\ees
as in the table in the introduction, in accordance with the fact that
this stratum is a $6$-fold unramified cover of~$\moduli[2]$
and $\chi(\moduli[2]) = -\tfrac{1}{240}$.

\subsection{The stratum $\bP\Omega\cM_{1,3}(-k-1,1,k)$}
\label{sec:exk1}

This example illustrates the quotient stack structure at the boundary of
the smooth compactification that result from prong-matchings, i.e., from points
with $\Tw[\Gamma] \neq \sTw[\Gamma]$. We have chosen a genus-one stratum
with a simple zero, since the projection to $\moduli[1,2]$ provides an
alternative way to compute all invariants in this case. We label the
points $z_1$ (pole), $z_2$ (simple zero) and $z_3$. The boundary divisors
here are $D_{\text{h}}$ and five more types of divisors, namely there are the
divisors
\bes
D_{1,a}=\left[
\begin{tikzpicture}[
baseline={([yshift=-.5ex]current bounding box.center)},
scale=2,very thick,
node distance=\HoG, %
bend angle=30]
\node[comp] (T)  {};
\node [minimum height=22pt,text height=0pt,above left] (T-1) at (T.north east) {$-k-1$};
\path (T) edge [shorten >=4pt] (T-1.center);
\node[comp] (B) [below=of T] {}
edge [bend left] 
node [order bottom left] {$-a-1$} 
node [order top left] {$a-1$} (T)
edge [bend right] 
node [order bottom right] {$-1-b$} 
node [order top right] {$b-1$} (T);
\node [text height=12pt,below right] (B-k) at (B.south east) {$k$};
\path (B) edge [shorten >=4pt] (B-k.center);
\node [text height=12pt,below left] (B-1) at (B.south west) {$1$};
\path (B) edge [shorten >=4pt] (B-1.center);
\end{tikzpicture}\right]
\quad 
D_2=\left[
\begin{tikzpicture}[
baseline={([yshift=-.5ex]current bounding box.center)},
scale=2,very thick,
node distance=\HoG, %
bend angle=30,
every loop/.style={very thick},
 comp/.style={circle,black,draw,thin,inner sep=0pt,minimum size=5pt,font=\tiny},
order top left/.style={pos=\PotlI,left,font=\tiny}]
\node[circle, draw, inner sep=0pt, minimum size=12pt] (T) [] {$1$}; 
\path (T) edge [shorten >=4pt] (T-1.center);
\node [minimum height=22pt,minimum width=1.5cm,above left] (T-1) at (T.north east) {$-k-1$};
\node[comp,fill] (B) [below=of T] {}
edge 
node [order bottom left] {$-k-3$} 
node [order top left] {$k+1$} (T);
\node [text height=12pt,below right] (B-2) at (B.south east) {$1$};
\path (B) edge [shorten >=4pt] (B-2.center);
\node [text height=12pt,below left] (B-2) at (B.south west) {$k$};
\path (B) edge [shorten >=4pt] (B-2.center);
\end{tikzpicture}\right],
\quad
D_3=\left[
\begin{tikzpicture}[
baseline={([yshift=-.5ex]current bounding box.center)},
scale=2,very thick,
node distance=\HoG, %
bend angle=30,
every loop/.style={very thick},
comp/.style={circle,fill,black,,inner sep=0pt,minimum size=5pt},
comp/.style={circle,black,draw,thin,inner sep=0pt,minimum size=5pt,font=\tiny},
order bottom left/.style={pos=.05,left,font=\tiny},
order top left/.style={pos=\PotlI,left,font=\tiny},
order bottom right/.style={pos=.05,right,font=\tiny},
order top right/.style={pos=\PotrI,right,font=\tiny}]
\node[circle, draw, inner sep=0pt, minimum size=12pt] (T) [] {$1$}; 
\node[comp, fill] (B) [below=of T] {}
edge 
node [order bottom left] {\medmuskip=0mu$-2$} 
node [order top left] {\medmuskip=0mu$0$} (T);
\node [minimum width=1cm,text height=10pt,below left] (B-2) at (B.south west) {\scriptsize$1$};
\path (B) edge [shorten >=6pt] (B-2.center);
\node [text height=14pt,below] (B-2) at (B.south) {\scriptsize$k$};
\path (B) edge [shorten >=2pt] (B-2.center);
\node [minimum width=1cm,text height=10pt,below right] (B-3) at (B.south east) {\medmuskip=0mu\scriptsize$-k-1$};
\path (B) edge [shorten >=6pt] (B-3.center);
\end{tikzpicture}\right],
\ees

where  $a,b \geq 1$ and $a+b = k+1$. Here $D_{1,a} = D_{1,k+1-a}$
and if $k$ is odd, the middle divisor $D_{1,(k+1)/2}$ has an
$\bZ/2$-involution. Moreover, there are the divisors
\bes
D_4 \= \left[
\begin{tikzpicture}[
baseline={([yshift=-.5ex]current bounding box.center)},
scale=2,very thick,
node distance=\HoG, %
bend angle=30]
\node[comp] (T) {};
\node [minimum height=12pt,text height=0pt,minimum width=1.5cm,above left] (T-1) at (T.north east) {$-k-1$};
\path (T) edge [shorten >=4pt] (T-1.center);
\node [minimum height=12pt,text height=0pt,minimum width=1cm,above right] (T-2) at (T.north east) {$1$};
\path (T) edge [shorten >=4pt] (T-2.center);
\node[circle, draw, inner sep=0pt, minimum size=12pt] (B) [below=of T] []{$1$}
edge 
node [order bottom left] {$-k$} 
node [order top left] {$k-2$} (T);
\node [text height=12pt,below right] (B-2) at (B.south east) {$k$};
\path (B) edge [shorten >=4pt] (B-2.center);
\end{tikzpicture}\right],
\quad 
D_{5,a'} \= \left[
\begin{tikzpicture}[
baseline={([yshift=-.5ex]current bounding box.center)},
scale=2,very thick,
node distance=\HoG, %
bend angle=30]
\node[comp] (T)  {};
\node [minimum height=12pt,text height=0pt,minimum width=1.5cm,above left] (T-1) at (T.north east) {$-k-1$};
\path (T) edge [shorten >=4pt] (T-1.center);
\node [minimum height=12pt,text height=0pt,minimum width=1cm,above right] (T-2) at (T.north east) {$1$};
\path (T) edge [shorten >=4pt] (T-2.center);
\node[comp] (B) [below=of T] {}
edge [bend left] 
node [order bottom left] {$-a'-1$} 
node [order top left] {$a'-1$} (T)
edge [bend right] 
node [order bottom right] {$-1-b'$} 
node [order top right] {$b'-1$} (T);
\node [text height=12pt,below right] (B-k) at (B.south east) {$k$};
\path (B) edge [shorten >=4pt] (B-k.center);
\end{tikzpicture}
  \right]
\ees

where $a'\in\{1,\dots, k-1\}$ and $b'=k-a'$. Again, $D_{5,a'} = D_{5,k-a'}$
with an involution on $D_{5,k/2}$ if $k$ is even. The local exponents
are
\bes
\ell_{1,a} \= \lcm(a,k+1-a), \quad \ell_2 = k+2, \quad   \ell_3 = 1, \quad
\ell_4\=k-1, \quad  \ell_{5,a'} = \lcm(a',k-a')\,
\ees
and the dimensions of the top level components are
\bes
N_1^\top =1, \quad N_2^\top =2, \quad N_3^\top =2, \quad N_4^\top =1,
\quad N_5^\top =2\,. 
\ees
We abbreviate 
$D_1 \= \frac12 \,\sum_{a=1}^{k} D_{1,a}$ and %
$D_5 \= \frac12 \,\sum_{a'=1}^{k-1} D_{5,a'}$.
\par
\medskip
\paragraph{\bf The local geometry of the boundary divisors} We give
a summary of the boundary points and intersection behaviour of the
boundary divisors listed above. We start with the boundary divisors that
map to the interior of $\moduli[1,2]$ under the map to $\moduli[1,1]$
forgetting the second point. These are represented by the thin lines in
Figure~\ref{fig:bd_stratumH1kmk-1}, while thick lines are mapped to the point
at infinity of $\moduli[1,1]$.
\begin{figure}
	\includegraphics[scale=1.4]{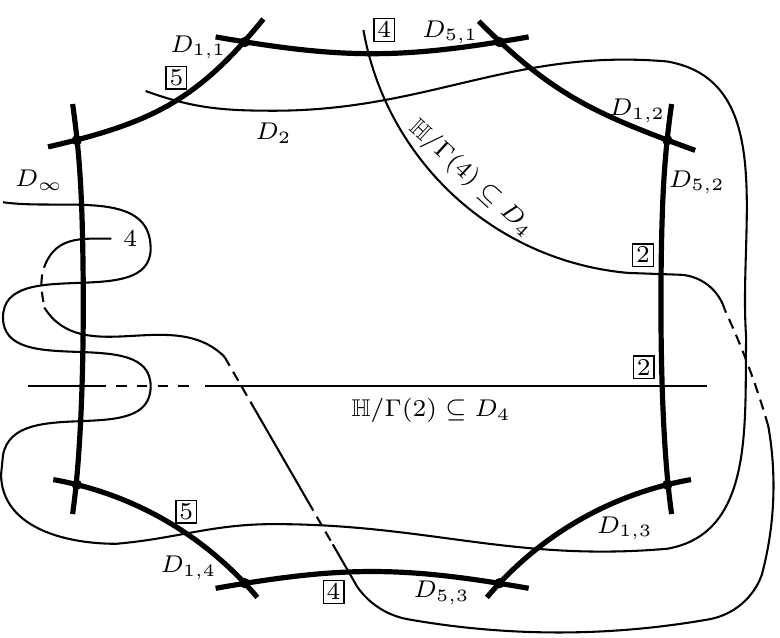}
\caption{The intersection behavior of the boundary in the stratum
$\bP\Omega\cM_{1,3}(-4,1,5)$. The figure has to be considered as quotient by
the elliptic involution that interchanges $D_{1,1}$ with~$D_{1,4}$ and
$D_{5,1}$ with~$D_{5,3}$ etc.
	} \label{fig:bd_stratumH1kmk-1}
\end{figure}
The divisor $D_3$ is simply a copy of the modular curve, intersecting
once $D_{\text{h}}$.
\par
The divisor $D_2$ minus its intersection with other boundary divisors is
the union of the modular curves $X_1(d) = \bH/\Gamma_1(d)$ for all
divisors~$d>1$ of~$k+1$. The only intersections with other boundary divisors
are $\left\lceil (k+1)/2\right \rceil -1 $ intersection points
with $D_{\text{h}}$ and $\gcd(a,b)$-points  with $D_{1,a}$. 
\par
The divisor $D_4$ minus its intersection with other boundary divisors is
the union of the modular curves $X_1(d) = \bH/\Gamma_1(d)$ for all
divisors~$d>1$ of~$k$. The only intersections with other boundary
divisors are $\left\lceil (k)/2\right \rceil -1 $ intersection points
with $D_{\text{h}}$ and $\gcd(a',b')$-points  with $D_{5,a'}$. 
\par
The curves $D_{1,a}$ and $D_{5,a'}$ form the exceptional divisor when
realizing the level compactification $\proj\LMS[-k-1,1,k][1,3]$ as
a blowup of $\barmoduli[1,2]$ in the node of the forgetful map to
$\barmoduli[1,1]$. Without prong-matchings the curve~$D_{1,a}$ 
were just an $\overline{M}_{0,4}$ (with a stack structure of an involution
if $a = (k+1)/2$). The three boundary points correspond to the intersection
with $D_{5,a-1}$ and $D_{5,a}$ (respectively with $D_{\text{h}}$ and $D_{5,a}$
if $a=1$), and with $D_2$. By the formulas in Section~\ref{sec:PM}, at
the generic point of $D_{1,a}$ (and also near the intersection with $D_2$)
there are $\gcd(a,b)$ prong-matching equivalence classes. At the
intersections with $D_5$ there is just one prong-matching
equivalence class. This implies that each $D_{1,a}$ is a $\gcd(a,b)$-fold cover
of $\barmoduli[0,4]$, totally ramified over the two points of intersection
with~$D_5$.
\par
Similarly, the divisor $D_5$ is a $\gcd(a',b')$-fold cover of
$\overline{M}_{0,4}$, totally ramified over the two points that correspond
to the intersections with~$D_1$. We compute the normal bundles of the divisor
using the special geometry of this example, independently of Theorem~\ref{thm:nb}.
\par
\begin{prop} 	\label{prop:selfintHk1mk}
The self-intersection number of $D_{1,a}$ is
\bes
D_{1,a}^2 \= - \delta_a^{k+1}\, \cdot k g_{1,a}/\ell_{1,a}
\quad \text{where} \quad g_{1,a} \= \gcd(a,b)
\ees
and where $\delta^{k+1}_{a}=1/2$ if $a = (k+1)/2$
and $\delta^{k+1}_{a}=1$ otherwise. The self-intersection number of $D_{5,a}'$ is 
  \bes
D_{5,a'}^2 \= - \delta_{a'}^{k}\, \cdot (k+1)\, g_{5,a'}/\ell_{5,a'}
\quad \text{where} \quad g_{5,a'} = \gcd(a',b')\,.
\ees
\end{prop}
\par
\begin{proof}
We consider the fibration $\pi:\proj\LMS[-k-1,1,k][1,3] \to \barmoduli[1,1]$
obtained from forgetting the last two marked points and take a smooth chart
of the quotient stack near the image of the curves $D_{1,a}$ and~$D_{5,a'}$. From
the intersection discussion above we deduce that the fiber over~$\infty$
in $\barmoduli[1,1]$ consists of a ring of rational curves intersecting in the
order
\bes
D_{\text{h}} -D_{1,1}-D_{5,1}-D_{1,2}-D_{5,2}-\cdots- D_{1,k-1} - D_{5,k} - D_{1,k+1}- D_{\text{h}}\,,
\ees
see again Figure~\ref{fig:bd_stratumH1kmk-1}.
We claim that the multiplicity of $D_{1,a}$ in the fiber $F = \pi^{-1}(\infty)$
is equal to $(k+1)/\gcd(a,k+1)$.  This can be deduced from the fact that $\pi|_{D_2}$
is a cover of degree $(k+1)^2 -1$ and from the order of the cusp stabilizers (see
\cite[Section~3.8]{DiaShu}, in particular the explanation around Figure~3.2) since
$D_2$ and $D_{1,a}$ intersect transversally in $\proj\LMS[-k-1,1,k][1,3]$.
Similarly, the multiplicity of $D_{5,a'}$ in this fiber is $k/\gcd(a',k)$.
Using the orbifold degree of the intersection points given in~\eqref{eq:eGamma}
and $D_{1,a}\cdot F = 0$ we find with $a' = a-1$ and $b'=b-1$ that
\bas
D_{1,a}^2 &\= - \frac{\gcd(a,b)}{k+1} \cdot \frac{abk}{\ell_{1,a}} \cdot
\Bigl(\frac{a'}{\ell_{5,a'} \gcd(a',k-a'))} +  \frac{b'}{\ell_{5,b'} \gcd(b',k-b'))}\Bigr) \\ 
&\= - \frac{\gcd(a,b)}{k+1} \cdot \frac{abk}{\ell_{1,a}} \cdot \frac{k+1}{ab}
\= - k \cdot g_{1,a}/\ell_{1,a}\,.
\eas
The proof of $D_{5,a'}$ is similar.
\end{proof} 
\autoref{prop:selfintHk1mk} agrees with \autoref{thm:intro:nb}. Indeed,
since the dimension of the top (resp. bottom) level stratum of $D_{1,a}$
(resp. $D_{5,a'}$) is zero dimensional, we compute
\bas
D_{1,a}^2&\=\c_1(\cN_{D_{1,a}})
\=\frac{1}{\ell_{1,a}}\left(-\xi_{D_{1,a}}^\top -\cL_{D_{1,a}}^\top+\xi_{D_{1,a}}^\bot\right)\\
&\=\frac{K_{1,a}}{\ell_{1,a}^2\Aut(D_{1,a})}\xi_{B_{1,a}^\bot}=\frac{g_{1,a}}{\ell_{1,a}\Aut(D_{1,a})}\cdot (-k)
\eas
where in the last two equalities we used \autoref{lem:evaluation} about evaluating top classes and the computation of top powers of $\xi$ which can be done analogously as in the first case of \autoref{prop:topxieasyexample}. Similarly we also get 
\bas
D_{5,a'}^2&\=\c_1(\cN_{D_{5,a'}})
\=\frac{1}{\ell_{5,a'}}\left(-\xi_{D_{5,a'}}^\top+\xi_{D_{5,a'}}^\bot -\cL_{D_{5,a'}}^\top\right)\\
&\=-\frac{K_{5,a'}}{\Aut(D_{5,a'})\ell_{5,a'}^2}\xi_{B_{5,a'}^\top}-\frac{1}{\ell_{5,a'}}\cL_{D_{5,a'}}^\top\\
&\=\frac{g_{5,a'}}{\Aut(D_{5,a'})\ell_{5,a'}}\cdot (-k)-\frac{1}{{\ell_{5,a'}}}\left([D_{5,a'}]\cdot([D_4]+[D_{1,a'}]+[D_{1,a'+1}])\right)\\
&\=\frac{g_{5,a'}}{\Aut(D_{5,a'})\ell_{5,a'}}\cdot (-k-1).
\eas
\smallskip
\paragraph{\bf The Euler characteristic} We give two ways proof of the
following fact.
\par
\begin{prop} The  moduli space $\bP\Omega\cM_{1,3}(-k-1,1,k)$
has Euler characteristic equal to $k(k+1)/6$.
\end{prop}
\par
\begin{proof} The first proof uses the description of $B =\bP\Omega\cM_{1,3}(-k-1,1,k)$
as the complement of $D_2$ and $D_4$ in $\moduli[1,2]$. By the above description
of $D_2$ and $D_4$ we need to compute
\bes
\sum_{\substack{d|k\\ d\not = k}} \chi(X_1(k/d))
\= \chi(\cM_{1,1}) \sum_{\substack{d|k\\ d\not = k}}[\SL_2(\bZ):\Gamma_1(k/d)]
\=-\frac{k^2-1}{12}\,,
\ees
which holds, since the rightmost sum counts the number of non-zero $k$-torsion points
in an elliptic curve. Together with $\chi(\moduli[1,2]) = -1/12$ implies the claim.
\par
The second proof evaluates Theorem~\ref{intro:ECformula}, given in the
surface case concretely by
\[
\chi(B) \= 3\xi_B^2  +\sum_{\Gamma \in \twolev }\frac{K_\Gamma \cdot N_\Gamma^\top}{|\Aut(\Gamma)|}
\left(\int_{B_\Gamma^\top}\xi_{B_\Gamma^\top}+\int_{B_\Gamma^\bot}\xi_{B_\Gamma^\bot}\right)+\sum_{\Delta \in \LG_2(B) }
\ell_{\delta_0(\Delta)}\ell_{\delta_1(\Delta)}[D_{\Delta}]\,.
\]
Using the third statement of \autoref{prop:topxieasyexample} we find 
\bas
3 \int_{\overline{B}} \xi^2\=3(k^4 -1)/24.
\eas 
For the divisors $D_{1,a}$ and~$D_4$
the contribution from $\xi_{B_\Gamma^\top}$ is zero and that of $\xi_{B_\Gamma^\bot}$
is non-zero, while for~$D_2$, $D_3$ and $D_{5,a'}$ the converse holds.
We evaluate in detail the contribution of that last divisor type. Its top
levels are $D_{5,a'}^\top=\bP\Omega_{0,4}(1,a'-1,b'-1,-k-1)$. Using
again \autoref{prop:topxieasyexample} we get
\bes
\sum_{\Gamma = D_{5,a} \atop a=1,\ldots k/2} \frac{K_\Gamma \cdot N_\Gamma^\top}{|\Aut(\Gamma)|}
\int_{B_\Gamma^\top}\xi_{B_\Gamma^\top}
\= \frac{1}{2}\sum_{a'=1}^{k-1} 2a'(k-a')
\cdot(-k)
\=\frac{-k^2(k^2-1)}{6}\,
\ees
Similar computations using again \autoref{prop:topxieasyexample} yield
\bas
\sum_{\Gamma = D_{1,a} \atop a=1,\ldots (k+1)/2} \frac{K_\Gamma N_\Gamma^\top}{|\Aut(\Gamma)|}
\int\xi_{B_\Gamma^\top} \= -k^2\frac{k^2+3k+2}{12}\,, \quad &2\ell_{D_3} \int\xi^\top_{D_3}
\= \frac{1}{12} \\
2\ell_{D_2} \int\xi^\top_{D_2} \= -2k(k+2)\frac{(k+1)^2-1}{24}\,,
\quad &\ell_{D_4} \int\xi^\top_{D_4} \= -(k-1)^2\frac{k^2-1}{24}
\eas
Using the evaluation result of \autoref{lem:evaluation} we finally get
\bes
\sum_{\Delta \in \LG_2(B) }\ell_{\delta_0(\Delta)}\ell_{\delta_1(\Delta)}[D_{\Delta}]
\= k(k+1)\frac{k^2+k+1}4\,.
\ees
Adding these contributions gives the claim. 
\end{proof}

\subsection{Hyperelliptic components} \label{sec:hypstrata}

We recall from \cite{kozo1} that strata of holomorphic Abelian differentials
have up to three connected components, distinguished by the parity of
the spin structure and possibly hyperelliptic components. The strata
$\omoduli[g,1](2g-2)$ and $\omoduli[g,2](g-1,g-1)$ have hyperelliptic
components. Their Euler characteristics are easy to compute.
\par
\begin{prop} \label{prop:HEEuler}
The Euler characteristic of the hyperelliptic components
are
\bas
\chi(\bP\omoduli[g,1](2g-2)^{\hyp}) &\= \frac{-1}{4g(2g+1)}
\quad \text{and} \quad \\
\chi(\bP\omoduli[g,2](g-1,g-1)^\hyp)
&\= \frac{1}{(2g+1)(2g+2)}\,.
\eas
\end{prop}
\par
\begin{proof} In the first case the surfaces are double covers of surfaces
the stratum $\cQ_0(-1^{2g+1},2g-3)$ of quadratic differentials with unnumbered
points, which is isomorphic to $\moduli[0,2g+2]/S_{2g+1}$. The claim follows
from $\chi(\moduli[0,n+3]) = (-1)^n \cdot n!$, taking into account the global
hyperelliptic involution on the stratum.
\par
Double covers of surfaces the stratum $\cQ_0(-1^{2g+2},2g-2)$ of quadratic
differentials with unnumbered points produce the Abelian differentials second
case, and this stratum is isomorphic to $\moduli[0,2g+3]/S_{2g+2}$.
The extra factors~$2$ from labelling the zeros of order~$g-1$ and~$1/2$
from the global hyperelliptic involution cancel each other.
\end{proof}

\subsection{Meromorphic strata and cross-checks} \label{sec:merocross}

In this section we provide in Table~\ref{cap:EulerMero} some Euler
characteristics for meromorphic strata. We abbreviate $\chi(\mu)
= \chi(\bP\omoduli[g,n](\mu))$.
\begin{figure}[h]
$$ \begin{array}{|c|c|c|c|c|c|}
\hline  &&&&& \\ [-\halfbls] 
\mu  &  (4,-2) & (3,1,-2) & (2,2,-2) & (2,1,1,-2) & (1^4,-2) \\
[-\halfbls] &&&&& \\ 
\hline &&&&& \\ [-\halfbls]
\chi(B) %
& -\frac{19}{24}  &  \frac{28}{15} & \frac{17}{10} & -6  &  26 \\
    [-\halfbls] &&&&& \\    
\hline &&&&& \\ %
\mu  &  (4,-1,-1) & (3,1,-1,-1) & (2,2,-1,-1) & (2,1^2,-1,-1) & (1^4,-1,-1) \\
[-\halfbls] &&&&& \\ 
\hline &&&&& \\ [-\halfbls]
\chi(B) %
& -\frac{8}{5}  &  -4 & -4 & 14  &  -63 \\
    [-\halfbls] &&&&& \\    
\hline %
\end{array}
$$
\captionof{table}[foo2]{Euler characteristics of some meromorphic strata}
\label{cap:EulerMero}
\end{figure} 
Moreover we provide several cross-checks for our values. First note
that the union of the strata of types $(4), (3,1), (2,2), (2,1,1)$ and
$(1^4)$ glue together to the projectivized Hodge bundle over $\barmoduli[3]$,
if all of them are taken with unmarked zeros. In fact, we read off
from Table~\ref{cap:EulerHolo} that
\bes
\chi(4) + \chi(3,1) + \frac12 \chi(2,2) + \frac12 \chi(2,1,1) +
\frac{1}{4!}\chi(1^4) \= \frac3{1008} \= \chi(\bP^2) \cdot \chi(\moduli[3])\,.
\ees
The value $\chi(4) = -55/504$ can also be retrieved from
Proposition~\ref{prop:HEEuler} and the computations of Bergvall
\cite[Table~4]{bergvall}, that gives the
cohomology of the stratum with odd spin structure $\omoduli[g,1](4)^\odd$
with $\bZ/2$-level structure. Computing the alternating sum weighted by
dimension gives~$-141120$. Since $|\Sp(6,\bZ)| = 1451520$ this checks with
\bes
\frac{-55}{504} \= \chi(4) \= \chi(4^\hyp) + \chi(4^\odd)
\= \frac{-1}{84} + \frac{-141120}{1451520}\,.
\ees
(A few other strata in $g=3$ might be cross-checked with table in
\cite{bergvall}, but one has to take into account that Bergvall glosses over
the existence of hyperelliptic curves in non-hyperelliptic strata.)
\par
Another cross-check is the Hodge bundle twisted by twice the universal
section over $\moduli[2,1]$. It decomposes into the unordered strata
$(4,-2), (3,1,-2), (2,2,-2)$, $(2,1,1,-2), (1^4,-2), (2,0), (1,1,0), (2)$
and the ordered stratum $(1,1)$, since the simple zero at the unique
marked point is distinguished. Note that $\chi(2,0) = 3 \chi(2)$ and
$\chi(1,1,0)=3 \chi(1,1)$. We can now add up the contributions
listed in Table~\ref{cap:EulerHolo} and Table~\ref{cap:EulerMero}
and find that the sum equals $ \frac{1}{40}
= \chi(\bP^2) \cdot \chi(\moduli[2,1])$. A similar cross-check can
be made for the Hodge bundle over $\moduli[2,2]$ twisted by every
section once, using the second row of Table~\ref{cap:EulerMero}.
\par

\printbibliography

\end{document}